\documentclass[twoside,11pt]{article}
\usepackage{geometry}
\usepackage{graphicx}
\usepackage{indentfirst}                                
\usepackage[colorlinks]{hyperref}
\hypersetup{linkcolor=blue,filecolor=black,urlcolor=blue, citecolor=black}   
\usepackage[numbers,sort&compress]{natbib}         
\usepackage[perpage,symbol*,marginal]{footmisc}   

\usepackage{amsmath}    
\usepackage{amsfonts}
\usepackage{amssymb}   
\usepackage{bm}             
\usepackage{mathrsfs}    
\usepackage{amsthm}     
\usepackage{amsxtra}

\ifxetex
\usepackage{letltxmacro}
\LetLtxMacro\SavedIncludeGraphics\includegraphics
\def\includegraphics#1#{
 \IncludeGraphicsAux{#1}%
}%
\newcommand*{\IncludeGraphicsAux}[2]{%
 \XeTeXLinkBox{%
  \SavedIncludeGraphics#1{#2}%
}}
\fi

\usepackage[titletoc,title]{appendix}
\usepackage{titlesec}
\titleformat{\section}{\large\bfseries\centering}{\thesection}{1em}{}
\titleformat{\subsection}{\centering\bfseries}{\thesubsection}{0.5em}{}
\titleformat{\subsubsection}[runin]{\bfseries}{\thesubsubsection.}{0.4em}{}[.]
\numberwithin{equation}{section}
\newtheorem{lemma}{Lemma}[section]
\newtheorem{proposition}[lemma]{Proposition}
\newtheorem{theorem}{Theorem}[section]

\newtheorem{definition}{Definition}[section]

\newcommand\w[1]{\makebox[1em]{$#1$}}

\def\d{\,\mathrm{d}}
\def\p{\partial}

\usepackage{accents}
\makeatletter
\def\wideubar{\underaccent{{\cc@style\underline{\mskip10mu}}}}
\def\Wideubar{\underaccent{{\cc@style\underline{\mskip8mu}}}}
\makeatother
\makeatletter
\def\widebar{\accentset{{\cc@style\underline{\mskip10mu}}}}
\def\Widebar{\accentset{{\cc@style\underline{\mskip8mu}}}}
\makeatother
\newcommand*\xbar[1]{%
	\hbox{%
		\vbox{%
			\hrule height 0.4pt 
			\kern0.3ex
			\hbox{%
				\kern-0.1em
				\ensuremath{#1}%
				\kern-0.1em
}}}}

\newcommand{\VERTiii}[1]{{\left\vert\kern-0.3ex\left\vert\kern-0.3ex\left\vert #1
		\right\vert\kern-0.3ex\right\vert\kern-0.3ex\right\vert}}
\newcommand{\VERT}{\vert\kern-0.3ex\vert\kern-0.3ex\vert}
\newcommand{\VERTl}{\left\vert\kern-0.3ex\left\vert\kern-0.3ex\left\vert}
\newcommand{\VERTr}{\right\vert\kern-0.3ex\right\vert\kern-0.3ex\right\vert}
\newcommand{\VERTbig}{\big\vert\kern-0.3ex\big\vert\kern-0.3ex\big\vert}
\newcommand{\VERTBig}{\Big\vert\kern-0.3ex\Big\vert\kern-0.3ex\Big\vert}
\usepackage{color}
\definecolor{DarkRed}{RGB}{139,0,0}
\definecolor{Purple}{RGB}{128,0,128}
\usepackage{ulem}

\begin{document}
\title{\bf Well-Posedness for the Free-Boundary Ideal Compressible Magnetohydrodynamic Equations with Surface Tension\let\thefootnote\relax\footnotetext{
		The research of {\sc Yuri Trakhinin}
		was partially supported by RFBR (Russian Foundation for Basic Research) under Grant 19-01-00261-a and by Mathematical Center in Akademgorodok under Agreement No.~075-15-2019-1675 with the Ministry of Science and Higher Education of the Russian Federation.
		The research of {\sc Tao Wang}
		was partially supported by the National Natural Science Foundation of China under Grants 11971359 and 11731008.
	}
}
\author{
{\sc Yuri Trakhinin}
\thanks{e-mail: trakhin@math.nsc.ru}\\
{\footnotesize Sobolev Institute of Mathematics, Koptyug av.~4, 630090 Novosibirsk, Russia}\\
{\footnotesize Novosibirsk State University, Pirogova str.~1, 630090 Novosibirsk, Russia}\\[2mm]
{\sc Tao Wang}
\thanks{e-mail: tao.wang@whu.edu.cn}\\
{\footnotesize School of Mathematics and Statistics, Wuhan University, Wuhan 430072, China}
}

\date{\empty}
\maketitle

\begin{abstract}

We establish the local existence and uniqueness of solutions to the free-boundary ideal compressible magnetohydrodynamic equations with surface tension in three spatial dimensions by a suitable modification of the Nash--Moser iteration scheme. The main ingredients in proving the convergence of the scheme are the tame estimates and unique solvability of the linearized problem in the anisotropic Sobolev spaces $H_*^m$ for $m$ large enough. In order to derive the tame estimates, we make full use of the boundary regularity enhanced from the surface tension. The unique solution of the linearized problem is constructed by designing some suitable $\varepsilon$--regularization and passing to the limit $\varepsilon\to 0$.

\end{abstract}

{\small

\noindent{\bf Keywords}:
Free boundary problem,
Ideal compressible magnetohydrodynamics,
Surface tension,
Well-posedness,
Nash--Moser iteration

\vspace*{2mm}

\noindent{\bf Mathematics Subject Classification (2010)}:\quad
35L65,  
35R35,  
76N10, 
76W05 

\setcounter{tocdepth}{3} 

\tableofcontents
}

\section{Introduction} \label{sec:intro}

We consider the free-boundary ideal compressible magnetohydrodynamic (MHD) equations with surface tension governing the dynamics of inviscid, compressible, and electrically conducting fluids in three spatial dimensions.
Let $\varOmega(t):=\{x\in\mathbb{R}^3:x_1>\varphi(t,x') \} $ be the changing volume occupied by the conducting fluid at time $t$, where $x':=(x_2,x_3)$ is the tangential coordinate.
The free boundary problem reads
\begin{subequations}
 \label{FBP1}
 \begin{alignat}{3}
 \label{MHD1a}
 &\p_t \rho+\nabla\cdot (\rho v)=0
 && \textrm{in } \varOmega(t), \\
 \label{MHD1b}
 &\p_t (\rho v)+\nabla\cdot (\rho v\otimes v-H\otimes H)+\nabla q=0
 && \textrm{in } \varOmega(t), \\
 \label{MHD1c}
 &\p_t H-\nabla\times(v\times H)=0
 && \textrm{in } \varOmega(t), \\
 \label{MHD1d}
 &\p_t(\rho E+\tfrac{1}{2}|H|^2) +\nabla\cdot ( v  (\rho E+p)+H\times(v\times H) )=0
 &\quad & \textrm{in } \varOmega(t),\\
 \label{BC1a}
 & \p_t \varphi=v\cdot N && \textrm{on } \varSigma(t), \\
 \label{BC1b}
 & q=\mathfrak{s}\mathcal{H}(\varphi) && \textrm{on } \varSigma(t),\\
 \label{IC1}
 &(\rho,v,H,S,\varphi)|_{t=0}=(\rho_0,v_0,H_0,S_0,\varphi_0),  &&
 \end{alignat}
\end{subequations}
supplemented with the constraints
\begin{alignat}{2}
\label{inv1}
&\nabla\cdot H=0 \quad &&\textrm{in } \varOmega(t),\\
\label{inv2}
&H\cdot N=0 &&\textrm{on } \varSigma(t),
\end{alignat}
for $\p_t:=\frac{\p}{\p t}$ and $\nabla:=(\p_1,\p_2,\p_3)^{\mathsf{T}}$ with $\p_i:=\frac{\p}{\p x_i}$,
where the density $\rho$, velocity $v\in\mathbb{R}^3$, magnetic field $H\in\mathbb{R}^3$, specific entropy $S$, and interface function $\varphi$ are to be determined.
Symbols $q=p+\frac{1}{2}|H|^2$ and $E=e+\frac{1}{2}|v|^2$ stand for the total pressure and specific total energy, respectively, where $p$ is the pressure and $e$ is the specific internal energy.
The thermodynamic variables $\rho$ and $e$ are given smooth functions of $p$ and $S$ satisfying the Gibbs relation
\begin{align} \label{Gibbs}
\vartheta \d S=\d e+p\d \left(\frac1\rho\right),
\end{align}
where $\vartheta>0$ is the absolute temperature.
We denote by $\varSigma(t):=\{x\in\mathbb{R}^3: x_1=\varphi(t,x')\}$ the moving vacuum boundary,
by $N:=(1,-\p_2\varphi,-\p_3\varphi)^{\mathsf{T}}$ the normal vector to $\varSigma(t)$,
by $\mathfrak{s}>0$ the constant coefficient of surface tension,
and by $\mathcal{H}(\varphi)$ twice the mean curvature of the boundary, that is,
\begin{align}
\label{H.cal:def}
\mathcal{H}(\varphi):= \mathrm{D}_{x'}\cdot\left( \frac{\mathrm{D}_{x'}\varphi }{\sqrt{1+|\mathrm{D}_{x'}\varphi|^2 }}\right)
\quad \textrm{with }
\mathrm{D}_{x'}:=
\begin{pmatrix} \p_2\\ \p_3 \end{pmatrix}.
\end{align}

The boundary condition \eqref{BC1a} states that the free boundary moves with the velocity of the conducting fluid and makes the free surface $\varSigma(t)$ {\it characteristic}.
The boundary condition \eqref{BC1b} results from surface tension and zero vacuum magnetic field. 
We refer to {\sc Landau--Lifshitz} \cite[\S 65]{LL84MR766230} and {\sc Delhaye} \cite{Delhaye1974} for the derivation of the MHD equations \eqref{MHD1a}--\eqref{MHD1d} and the boundary condition \eqref{BC1b}.
It should be noted that the effect of surface tension is especially important for modelling MHD flows in liquid metals (see, e.g., \cite{SDGX07MR2356385,Yang-Khodak-Stone} and references therein).
Even for MHD modelling of large scales phenomena like those in astrophysical plasmas, where the effect of surface tension and diffusion is usually neglected, it is still useful to keep surface tension as a stabilizing mechanism in numerical simulations of magnetic Rayleigh--Taylor instability \cite{SG1,SG2}.

{The absence of the magnetic field, {\it i.e.}, $H\equiv 0$, reduces the problem \eqref{FBP1} to the moving vacuum boundary problem for the compressible Euler equations, which has been studied extensively in the recent decades.
In the case of zero surface tension, $\mathfrak{s}=0$, an ill-posedness result have been shown by {\sc Ebin} \cite{E87MR886344} for the incompressible Euler equations when the Rayleigh--Taylor sign condition is violated, 
while the local well-posedness has been established
in \cite{L05AnnMR2178961,CS07MR2291920,ZZ08MR2410409,W99MR1641609} and \cite{L05CMPMR2177323,T09CPAMMR2560044} respectively for incompressible and compressible liquids under the Rayleigh--Taylor sign condition,
and in \cite{CS12MR2980528,JM15MR3280249} for compressible gases under the physical vacuum condition.
In the case of positive surface tension, $\mathfrak{s}>0$,
the local well-posedness has been proved independently by {\sc Coutand--Shkoller} \cite{CS07MR2291920} and {\sc Shatah--Zeng} \cite{SZ08aMR2388661,SZ11MR2763036} for the incompressible Euler equations,
and by {\sc Coutand et al.}~\cite{CHS13MR3139610} for compressible isentropic liquids,
without imposing the Rayleigh--Taylor sign condition. 
The results of \cite{CS07MR2291920,CHS13MR3139610,SZ08aMR2388661,SZ11MR2763036} indicate that the surface tension provides a regularizing effect on the vacuum boundary.}

{On the other hand, there have been only few results on the free boundary problem \eqref{FBP1} for the ideal compressible MHD, due to the difficulties caused by the complex interplay between the velocity and magnetic fields.
For the free-boundary ideal incompressible MHD without surface tension,
the {\it a priori} estimates and local well-posedness have been respectively proved by {\sc Hao--Luo} \cite{HL14MR3187678} and {\sc Gu--Wang} \cite{GW19MR3980843} under the generalized Rayleigh--Taylor sign condition on the total pressure, while a counterexample to well-posedness has been provided by {\sc Hao--Luo} \cite{HL20MR4093862} when the generalized sign condition fails.
Regarding the ideal compressible MHD equations \eqref{FBP1},
the authors \cite{TW21MR4201624} have recently established the local well-posedness for the case of zero surface tension $\mathfrak{s}=0$ under the generalized sign condition.
It would expect that the surface tension can have a stabilization effect on the evolution as for the cases without magnetic fields.
The goal of the present paper is to confirm this effect rigorously, or, more precisely, to construct the unique solution of the problem \eqref{FBP1} with surface tension replacing the generalized sign condition.}



Suppose that the sound speed $a:=a(\rho,S)$ is smooth and satisfies
\begin{align} \nonumber 
	a(\rho,S):=\sqrt{\frac{\p p}{\p \rho} (\rho, S)}>0
	\quad \textrm{for all \ } \rho \in (\rho_*,\rho^*) ,
\end{align}
where $\rho_*$ and $\rho^*$ are positive constants with $\rho_*<\rho^*$.
In this paper we consider the liquids for which the density $\rho$ belongs to $(\rho_*,\rho^*)$.
Consequently, it follows from the constraint \eqref{inv1} and the Gibbs relation \eqref{Gibbs} that the MHD equations \eqref{MHD1a}--\eqref{MHD1d} are equivalent to the symmetric hyperbolic system
\begin{align}
\label{MHD.vec}
A_0(U)\p_t U+\sum_{i=1}^{3}A_i(U)\p_i U=0\qquad \textrm{in } \varOmega(t),
\end{align}
where we choose $U:=(q,v,H,S)^{\mathsf{T}}$ as the primary unknowns, and
\setlength{\arraycolsep}{4pt}
\begin{align}
\label{A0.def}
&A_0(U):=
\begin{pmatrix}
\dfrac{1}{\rho a^2} & 0 & -\dfrac{1}{\rho a^2} H^{\mathsf{T}} & 0\\[2.5mm]
0 & \rho {I}_3 & {O}_3 & 0\\[1.5mm]
-\dfrac{1}{\rho a^2} H & {O}_3 & {I}_3+\dfrac{1}{\rho a^2}H\otimes H & 0\\[1.5mm]
\w{0}  & \w{0}  & \w{0} & \w{1}
\end{pmatrix},
\\ 
\label{Ai.def}
&A_i(U):=
\begin{pmatrix}
\dfrac{v_i}{\rho a^2} & \bm{e}_i^{\mathsf{T}} & -\dfrac{v_i}{\rho a^2} H^{\mathsf{T}} & 0\\[2.5mm]
\bm{e}_i & \rho v_i {I}_3 & -H_i {I}_3 & 0\\[1.5mm]
-\dfrac{v_i}{\rho a^2} H & -H_i {I}_3 & v_i{I}_3+\dfrac{v_i}{\rho a^2}H\otimes H & 0\\[1.5mm]
\w{0}  & \w{0}  & \w{0} & \w{v_i}
\end{pmatrix}
 \textrm{ for $i=1,2,3$.}
\end{align}
Here and in what follows, ${O}_m$ and ${I}_{m}$ denote the zero and identity matrices of order $m$, respectively,
and $\big\{\bm{e}_1:=(1,0,0)^{\mathsf{T}}$, $\bm{e}_2:=(0,1,0)^{\mathsf{T}}$, $\bm{e}_3:=(0,0,1)^{\mathsf{T}}\big\}$ is the standard basis of $\mathbb{R}^3$.

We reduce the free boundary problem \eqref{FBP1} into a fixed domain by straightening the unknown surface $\varSigma(t)$. More precisely, we introduce
\begin{align} \label{U.sharp}
U_{\sharp}(t,x):=U(t,\varPhi(t,x),x'),
\end{align}
where $\varPhi$ takes the following form that is suggested by {\sc M\'{e}tivier} \cite[p.\;70]{M01MR1842775}:
\begin{align} \label{varPhi.def}
\varPhi(t,x):=x_1+\kappa_{\sharp} \chi( x_1)\varphi(t,x'),
\end{align}
with the constant $\kappa_{\sharp}>0$ and function $\chi\in C^{\infty}_0(\mathbb{R})$ satisfying
\begin{align} \label{chi}
4\kappa_{\sharp} \|\varphi_0\|_{L^{\infty}(\mathbb{R}^{2}) }\leq  1,
\quad \|\chi'\|_{L^{\infty}(\mathbb{R})} <1,
\quad \chi\equiv 1\  \textrm{on } [0,1].
\end{align}
This change of variables is admissible on the time interval $[0,T]$,
provided $T>0$ is suitably small.
Without loss of generality we set $\kappa_{\sharp}=1$.
As in \cite{M01MR1842775,T09CPAMMR2560044,TW21MR4201624,T09ARMAMR2481071}, we employ the cut-off function $\chi$ to avoid assumptions about compact support of the initial data (shifted to an equilibrium state).

Then the moving boundary problem \eqref{FBP1} is reformulated as the following fixed-boundary problem:
\begin{subequations} \label{NP1}
 \begin{alignat}{2}
 \label{NP1a}
 &\mathbb{L}(U,\varPhi) :=L(U,\varPhi)U =0
 &\quad &\textrm{in  }  [0,T]\times \varOmega,\\
 \label{NP1b}
 &\mathbb{B}(U,\varphi):=
 \begin{pmatrix}
 \p_t \varphi -v\cdot N\\ q-\mathfrak{s}\mathcal{H}(\varphi)
 \end{pmatrix}=0
 & \quad  &\textrm{on  } [0,T]\times \varSigma,\\
 \label{NP1c}
 &(U,\varphi)|_{t=0}=(U_0,\varphi_0),
 & \quad &
 \end{alignat}
\end{subequations}
where we drop for convenience the subscript $``\sharp"$ in $U_{\sharp}$,
and define the fixed domain $\varOmega:=\{x\in \mathbb{R}^3:\, x_1>0\}$, the boundary $\varSigma:=\{x\in \mathbb{R}^3:\, x_1=0\}$, and the operator
\begin{align}
\label{L.def}
&L(U,\varPhi):=A_0(U)\partial_t+\widetilde{A}_1(U,\varPhi)\partial_1+ A_2(U)\partial_2
+A_3(U)\partial_3,
\end{align}
with
\begin{align}
\label{A1t.def}
&\widetilde{A}_1(U,\varPhi):=
\frac{1}{\partial_1\varPhi}\big(A_1(U)-\partial_t\varPhi A_0(U)-\partial_2\varPhi A_2(U)-\partial_3\varPhi A_3(U)\big).
\end{align}
In the new variables, the identities \eqref{inv1}--\eqref{inv2} become
\begin{alignat}{3}
\label{inv1b}
& \p_{1}^{\varPhi}H_1+\p_{2}^{\varPhi}H_2+\p_{3}^{\varPhi} H_3=0\qquad &&\textrm{if  }  x_1>0,\\
\label{inv2b}
&H\cdot N=0\qquad &&\textrm{if  }  x_1=0,
\end{alignat}
where  
\begin{align} \label{differential}
	\partial_t^{\varPhi}:=\partial_t-\frac{\partial_t\varPhi}{\partial_1\varPhi}\p_1,\ \
	\partial_1^{\varPhi}:=\frac{1}{\partial_1\varPhi}\partial_1,\ \
	\partial_i^{\varPhi}:=\partial_i-\frac{\partial_i\varPhi}{\partial_1\varPhi}\partial_1
	\ \   \textrm{for }  i=2,3.
\end{align}
The identities \eqref{inv1b}--\eqref{inv2b} can be taken as initial constraints; we refer the reader to \cite[Appendix A]{T09ARMAMR2481071} for the proof.

Let us denote by $\lfloor s \rfloor$ the floor function mapping $s\in\mathbb{R}$ to the greatest integer less than or equal to $s$.
Now we are in a position to state the local well-posedness theorem for the problem \eqref{NP1}, which clearly implies a corresponding theorem for the original free boundary problem \eqref{FBP1}.

\begin{theorem}
\label{thm:main}
Let $m\in\mathbb{N}$ with $m\geq 20$.
Suppose that the initial data $(U_0,\varphi_0)$ satisfy the hyperbolicity condition $\rho_*<\inf_{\varOmega}\rho({U}_0)\leq \sup_{\varOmega}\rho({U}_0)<\rho^*$, the constraints \eqref{inv1b}--\eqref{inv2b}, and the compatibility conditions up to order $m$ (see Definition \ref{def:1}).
Suppose further that $(U_0-\widebar{U},\varphi_0)$ belongs to
$H^{m+3/2}(\varOmega)\times H^{m+2}(\mathbb{R}^{2})$ for some constant equilibrium
$\widebar{U}$ with $ \rho(\widebar{U})\in(\rho_*,\rho^*) $.
Then there are a constant $T>0$ and a unique solution $(U,\varphi)$ of the problem \eqref{NP1} on the time interval $[0,T]$, such that
\begin{align*}
U-\widebar{U}\in H^{\lfloor (m-9)/2\rfloor}([0,T]\times\varOmega),\quad
(\varphi,\mathrm{D}_{x'}\varphi)\in H^{m-9}([0,T]\times\mathbb{R}^{2}).
\end{align*}
\end{theorem}

We will construct smooth solutions to the nonlinear characteristic problem \eqref{NP1} through a suitable modification of the Nash--Moser iteration scheme developed by {\sc H\"{o}rmander} \cite{H76MR0602181} and {\sc Coulombel--Secchi} \cite{CS08MR2423311}; see \cite{AG07MR2304160,S16MR3524197} for a general description of the Nash--Moser method.
The main ingredients in proving the convergence of the Nash--Moser iteration scheme are the {\it tame estimates} and {\it unique solvability} of the linearized problem in certain function spaces ({\it cf.}~\cite[Assumptions 2.1--2.2]{S16MR3524197}).

For the linearized problem around a basic state, if the generalized Rayleigh--Taylor sign condition on the total pressure were imposed as in our previous work \cite{TW21MR4201624} for the case of zero surface tension, then we could deduce a uniform-in-$\mathfrak{s}$ $L^2$ energy estimate with no loss of derivatives from the source term to the solution ({\it cf.}\;\eqref{est1c}), allowing us to construct the unique solution by a direct application of the classical duality argument in {\sc Lax--Phillips}
\cite{LP60MR0118949} and {\sc Chazarain--Piriou}
\cite[Chapter 7]{CP82MR0678605}.

However, in this paper we aim to release the Rayleigh--Taylor sign condition and consider the general case for which the $L^2$ energy estimate is not closed.
To deal with this situation, we propose here to build the basic {\it a priori} estimate in the anisotropic Sobolev space $H_*^1$ rather than in $L^2$ for the linearized problem by taking advantage of the boundary regularity gained from the surface tension.
More precisely,
we first derive the $L^2$ energy estimates for the solutions and their {\it tangential spatial derivatives}, where the surface tension can provide good terms for the first and second order spatial derivatives of the interface function, respectively ({\it cf.}\;\eqref{est1d} and \eqref{est2b}).
However, it is difficult to deduce the estimate of the {\it time derivative}, since the mean curvature operator ({\it cf.}~\eqref{H.cal:def}) does not involve the time derivative of the interface function.
To overcome this difficulty, we reduces the tough boundary term $\mathcal{J}_{1a}$ defined in \eqref{J1.est1} to the volume integral containing the normal derivative of the noncharacteristic variable $W_2$, which can be controlled by virtue of the interior equations ({\it cf.}\;\eqref{iden5}--\eqref{est3d}).
After that, using the estimate of the time derivative of the interface function, we achieve the $H_*^1$ {\it a priori} estimate for the linearized problem ({\it cf.}\;\eqref{est4}).
To obtain high-order energy estimates for the linearized problem,
we make full use of the improved boundary regularity and combine some estimates in \cite{TW21MR4201624} that are still available for our case $\mathfrak{s}>0$ (see \S \S \ref{sec:linear4}--\ref{sec:linear5} for the complete derivation).

The other main ingredient in our proof is to construct the unique solution of the linearized problem.
For this purpose,
we design for the linearized problem some suitable $\varepsilon$--regularization, for which we can deduce an $L^2$ {\it a priori} estimate with a constant $C(\varepsilon )$ depending on the small parameter $\varepsilon>0$.
Furthermore, an $L^2$ {\it a priori} estimate can be also shown for the corresponding dual problem.
Then for any fixed and small parameter $\varepsilon>0$, the existence and uniqueness of solutions in $L^2$ can be established by the duality argument.
However, our constant $C(\varepsilon)$ tends to infinity as $\varepsilon\to 0$, and hence we are not able to use the $L^2$ estimate obtained for the regularized problem to take the limit $\varepsilon\to 0$.
To overcome this difficulty,
we derive a uniform-in-$\varepsilon$ estimate in $H_*^1$ for the regularized problem,
which enables us to solve the linearized problem by passing to the limit $\varepsilon\to 0$.

It is worth mentioning that the anisotropic Sobolev spaces, introduced first by {\sc Chen} \cite{C07MR2289911}, have been shown to be appropriate and effective for studying general symmetric
hyperbolic problems with characteristic boundary; see {\sc Secchi} \cite{S96MR1405665} for a general theory and \cite{T09ARMAMR2481071,TW21MR4201624,CW08MR2372810,ST14MR3151094} for other results on characteristic problems in ideal compressible MHD.

We emphasize that our energy estimate \eqref{tame} for the linearized problem exhibits a {\it fixed} loss of derivatives from the basic state to the solution and hence is a so-called {\it tame estimate}.
To compensate this loss of derivatives and solve the nonlinear problem, we employ the modified Nash--Moser iteration technique, which has been also applied to the study of characteristic discontinuities \cite{CSW19MR3925528,MTT18MR3766987,CS08MR2423311,T09ARMAMR2481071,CW08MR2372810} and vacuum free-boundary problems \cite{TW21MR4201624,ST14MR3151094,T09CPAMMR2560044}.
Nevertheless,
in view of the aforementioned works \cite{CHS13MR3139610,CS07MR2291920,SZ08aMR2388661,SZ11MR2763036}, it is expected to avoid the loss of regularity and show the well-posedness for the nonlinear problem \eqref{FBP1} without resorting to the Nash--Moser method.
The proof of this expectation is an interesting open problem for future research.

The rest of this paper is organized as follows.
Section \ref{sec:linear} is devoted to the proof of Theorem \ref{thm:linear}, that is, the well-posedness of the linearized problem in the anisotropic Sobolev spaces $H_*^m$ for any integer $m$ large enough.
To be more precise, for the effective linear problem \eqref{ELP1}, we deduce the $H_*^1$ {\it a priori} estimate in \S \ref{sec:linear2}, construct the unique solution in \S \ref{sec:linear3} by passing to the limit $\varepsilon\to0$ in $H_*^1$ from some certain $\varepsilon$--regularization, and complete the proof of Theorem \ref{thm:linear} in \S \ref{sec:linear5} with the aid of the high-order energy estimates obtained in \S \ref{sec:linear4}.
For convenience, we collect a list of notation before the statement of Theorem \ref{thm:linear} in \S \ref{sec:linear1}.
In \S \ref{sec:proof}, the existence part of Theorem \ref{thm:main} is proved by using a modified Nash--Moser
iteration scheme ({\it cf.}~\S\S \ref{sec:proof1}--\ref{sec:proof4}), while the uniqueness part follows from the $H_*^1$ energy estimate for the difference of solutions ({\it cf.}~\S \ref{sec:unique}).

\section{Well-Posedness of the Linearized Problem} \label{sec:linear}

This section is devoted to showing the tame estimates and unique solvability for the linearized problem of \eqref{NP1} in anisotropic Sobolev spaces $H_*^m$ with integer $m$ large enough.

\subsection{Main Theorem for the Linearized Problem}\label{sec:linear1}
For $T>0$, we denote $\varOmega_T:=(-\infty,T)\times \varOmega$ and $\varSigma_T:=(-\infty,T)\times\varSigma$.
Let the basic state $(\mathring{U},\mathring{\varphi})$ with $\mathring{U}:=(\mathring{q},\mathring{v},\mathring{H},\mathring{S})^{\mathsf{T}}$ be sufficiently smooth and satisfy
\begin{alignat}{3}
\label{bas1a}
&\rho_*<\inf_{\varOmega}\rho(\mathring{U})\leq \sup_{\varOmega}\rho(\mathring{U})<\rho^*
\quad &&\textrm{in } \varOmega_T,\\
\label{bas1b}
&\p_t \mathring{\varphi}=\mathring{v}\cdot \mathring{N},\quad \mathring{H}\cdot \mathring{N}=0
\quad&&\textrm{on }\varSigma_T,
\end{alignat}
where
$$\mathring{N}:=(1,-\p_2\mathring{\varphi},-\p_3\mathring{\varphi})^{\mathsf{T}}
=(1,-\mathrm{D}_{x'}\mathring{\varphi})^{\mathsf{T}}.$$
We also denote $\mathring{\varPsi}:=\chi(x_1)\mathring{\varphi}(t,x')$ and
$\mathring{\varPhi}:=x_1+\mathring{\varPsi}$
with $\chi\in C_0^{\infty}(\mathbb{R})$ satisfying \eqref{chi}.
Then $\p_1\mathring{\varPhi}\geq 1/2$ on $\varOmega_T$, provided
we without loss of generality assume that $\|\mathring{\varphi}\|_{L^{\infty}(\varSigma_T)}\leq 1/2$.
Furthermore, we suppose that
\begin{align}
\label{bas1c}
\|\mathring{U}\|_{W^{3,\infty}(\varOmega_T)}+\|\mathring{\varphi}\|_{W^{4,\infty}(\varSigma_T)}\leq K,
\end{align}
for some constant $K>0$.

The linearized operator around the basic state $(\mathring{U},\mathring{\varphi})$ for \eqref{NP1a} reads
\begin{align}
\nonumber
\mathbb{L}'(\mathring{U}, \mathring{\varPhi})(V,\varPsi)
:=\;& \left.\frac{\mathrm{d}}{\mathrm{d}\theta}
\mathbb{L}\big(\mathring{U} +\theta V ,\,\mathring{\varPhi} +\theta \varPsi \big)\right|_{\theta=0}\\
\nonumber
=\;& L(\mathring{U}, \mathring{\varPhi})V
+\mathcal{C}( \mathring{U},\mathring{\varPhi})V
-L(\mathring{U}, \mathring{\varPhi})\varPsi\frac{\p_1\mathring{U}}{\p_1 \mathring{\varPhi}},
\end{align}
where $\varPsi:=\chi(x_1)\psi(t,x'),$
the operator $L$ is given in \eqref{L.def}, and $\mathcal{C}$ is defined by
\begin{align} \nonumber 
\mathcal{C}({U},{\varPhi})V:=
\sum_{k=1}^{8}V_k\bigg(\frac{\p A_0}{\p {U_k}}({U}) \partial_t {U}
+ \frac{\p \widetilde{A}_1}{\p {U_k}}({U},{\varPhi}) \partial_1 {U}
+\sum_{i=2,3}\frac{\p A_i}{\p {U_k}}({U}) \partial_i {U}
\bigg).
\end{align}
Introducing the good unknown of {\sc Alinhac} \cite{A89MR976971}:
\begin{align} \label{good}
\dot{V}:=V-\frac{\partial_1\mathring{U} }{\partial_1 \mathring{\varPhi} }\varPsi,
\end{align}
we have ({\it cf.}\;\cite[Proposition 1.3.1]{M01MR1842775})
\begin{align}
\mathbb{L}'(\mathring{U}, \mathring{\varPhi})(V,\varPsi)
= L(\mathring{U}, \mathring{\varPhi})\dot{V}
+\mathcal{C}( \mathring{U},\mathring{\varPhi})\dot{V}
+\frac{\varPsi}{\partial_1\mathring{\varPhi}}
\partial_1\big(L(\mathring{U},\mathring{\varPhi} )\mathring{U}\big).
\label{Alinhac}
\end{align}
Regarding the linearized operator for \eqref{NP1b},
we compute from \eqref{H.cal:def} that
\begin{align}
\left.\frac{\mathrm{d}}{\mathrm{d}\theta} \mathcal{H}(\mathring{\varphi} +\theta \psi)\right|_{\theta=0}=\;&
\mathrm{D}_{x'}\cdot
\frac{\mathrm{d}}{\mathrm{d}\theta}
\bigg(\frac{\mathrm{D}_{x'}(\mathring{\varphi} +\theta \psi)}{\sqrt{1+|\mathrm{D}_{x'}(\mathring{\varphi} +\theta \psi)|^2}}\bigg)
\bigg|_{\theta=0}
\nonumber\\
=\;&\mathrm{D}_{x'}\cdot
\bigg(\frac{\mathrm{D}_{x'}\psi}{|\mathring{N}|}- \frac{\mathrm{D}_{x'}\mathring{\varphi}\cdot\mathrm{D}_{x'}\psi}{|\mathring{N}|^3}\mathrm{D}_{x'}\mathring{\varphi}\bigg)
,
\nonumber
\end{align}
and hence
\begin{align}
\nonumber
\mathbb{B}'(\mathring{U} ,\mathring{\varphi})(V,\psi)
:=&\left.\frac{\mathrm{d}}{\mathrm{d}\theta}
\mathbb{B}\big(\mathring{U} +\theta V ,\,\mathring{\varphi} +\theta \psi \big) \right|_{\theta=0}\\
=&\begin{pmatrix}
(\p_t + \mathring{v}_2 \p_2 + \mathring{v}_3 \p_3) \psi-v\cdot\mathring{N}\\[1mm]
q-\mathfrak{s} \mathrm{D}_{x'}\cdot
\bigg(\dfrac{\mathrm{D}_{x'}\psi}{|\mathring{N}|}-
\dfrac{\mathrm{D}_{x'}\mathring{\varphi}\cdot\mathrm{D}_{x'}\psi}{|\mathring{N}|^3}\mathrm{D}_{x'}\mathring{\varphi}\bigg)
\end{pmatrix}.
\label{B'.bb:def}
\end{align}

Neglecting the last term of \eqref{Alinhac} as in \cite{  CS08MR2423311,T09ARMAMR2481071,T09CPAMMR2560044,TW21MR4201624},
we write down the {\it effective linear problem}
\begin{subequations} \label{ELP1}
 \begin{alignat}{3}
 \label{ELP1a}
 &\mathbb{L}'_{e}(\mathring{U}, \mathring{\varPhi}) \dot{V}
 :=L(\mathring{U}, \mathring{\varPhi})\dot{V}
 +\mathcal{C}( \mathring{U},\mathring{\varPhi})\dot{V}=f
 &\quad &\textrm{if } x_1>0,\\
 \label{ELP1b}
 &\mathbb{B}'_e(\mathring{U}, \mathring{\varphi}) (\dot{V},\psi)=g
 &&\textrm{if } x_1=0,\\
 \label{ELP1c}
 &(\dot{V},\psi)=0   &&\textrm{if } t<0,
 \end{alignat}
\end{subequations}
where the identity
$\mathbb{B}'_e(\mathring{U}, \mathring{\varphi}) (\dot{V},\psi)
=\mathbb{B}'(\mathring{U} ,\mathring{\varphi})(V,\psi)$ results in the exact form
\begin{align}
\nonumber
\mathbb{B}'_e(\mathring{U}, \mathring{\varphi}) (\dot{V},\psi)
:=
\begin{pmatrix}
(\p_t  +\mathring{v}_2 \p_2+\mathring{v}_3 \p_3)\psi-\p_1 \mathring{v}\cdot \mathring{N}\psi-\dot{v}\cdot\mathring{N}\\[1mm]
\dot{q}+\p_1 \mathring{q}\psi-\mathfrak{s} \mathrm{D}_{x'}\cdot
\bigg(\dfrac{\mathrm{D}_{x'}\psi}{|\mathring{N}|}- \dfrac{\mathrm{D}_{x'}\mathring{\varphi}\cdot\mathrm{D}_{x'}\psi}{|\mathring{N}|^3}\mathrm{D}_{x'}\mathring{\varphi}\bigg)
\end{pmatrix}.
\end{align}

As in our previous work \cite{TW21MR4201624} for the case of zero surface tension, we will study the linearized problem \eqref{ELP1} in anisotropic Sobolev spaces $H_*^m$ that are defined below.
\vspace*{4mm}

\noindent{\bf $\bullet$ Notation.}\quad
Throughout this paper we adopt the following notation.

\begin{itemize}
\item[(i)] We use letter $C$ to denote any universal positive constant.
Symbol $C(\cdot)$ denotes any generic positive constant depending on the quantities listed in the parenthesis.
We employ $X\lesssim Y$ or $Y\gtrsim X$ to denote the statement that $X \leq CY$ for some universal constant $C>0$.

\item[(ii)] Recall that $\nabla:=(\p_1,\p_2,\p_3)^{\mathsf{T}}$ and $\mathrm{D}_{x'}:=(\p_2,\p_3)^{\mathsf{T}}$ with $\p_i:=\frac{\p}{\p x_i}$ for $i=1,2,3$.
The time derivative $\p_t:=\frac{\p}{\p t}$ will be denoted in many cases by $\p_0:=\frac{\p}{\p t}$.
For any multi-index $\alpha:=(\alpha_{2},\alpha_{3})\in\mathbb{N}^{2}$ and $m\in\mathbb{N}$,
we define
\begin{align}
\mathrm{D}_{x'}^{\alpha}:=\p_2^{\alpha_2} \p_3^{\alpha_3},
\ \
|\alpha|:=\alpha_{2}+\alpha_{3},\ \
\mathrm{D}_{x'}^{m}:=
(\p_2^m,\p_2^{m-1}\p_3,\ldots, \p_3^m)^{\mathsf{T}}.
\label{Dx'}
\end{align}

\item[(iii)]
Symbol $\mathrm{D}$ will be employed to denote the space-time gradient
\begin{align*}
	\mathrm{D}:=(\p_t,\p_1,\p_2,\p_3)^{\mathsf{T}}.
\end{align*}
For $m\in\mathbb{N}$, we denote by $\mathring{\rm c}_m$  a generic and smooth matrix-valued function of
$\{(\mathrm{D}^{\alpha} \mathring{V},\mathrm{D}^{\alpha}\mathring{\varPsi}):|\alpha|\leq m\}$,
where $\mathrm{D}^{\alpha}:=\p_t^{\alpha_0} \p_1^{\alpha_1}\p_2^{\alpha_2} \p_3^{\alpha_3} $ and $\alpha:=(\alpha_0,\ldots,\alpha_{3})\in\mathbb{N}^{4}$ with the convention $|\alpha|:=\alpha_0+\cdots+\alpha_{3}$.
The exact form of $\mathring{\rm c}_m$ may vary at different places.

\item[(iv)] We employ symbol $\mathrm{D}_*^{\alpha}$ to mean that
\begin{align} \label{D*}
\mathrm{D}_*^{\alpha}:=
\p_t^{\alpha_0} (\sigma \p_1)^{\alpha_1}\p_2^{\alpha_2}  \p_3^{\alpha_3} \p_1^{\alpha_{4}},
\quad \alpha:=(\alpha_0,\ldots,\alpha_{4})\in\mathbb{N}^{5},
\end{align}
with $\langle \alpha \rangle :=|\alpha|+\alpha_{4}$ and $|\alpha|:= \alpha_{0}+\cdots+\alpha_4$,
where $\sigma=\sigma(x_1)$ is an increasing smooth function on $[0,+\infty)$ such that
$\sigma(x_1)=x_1$ for $0\leq x_1\leq 1/2$ and $\sigma(x_1)=1$ for $x_1\geq 1$.

\item[(v)] For $m\in\mathbb{N}$ and $I\subset \mathbb{R}$, we define the function space $H_*^{m}(I\times \varOmega)$ as
\begin{align*}
&H_*^m(I\times\varOmega):=
\{ u\in L^2(I\times\varOmega):\, \mathrm{D}_*^{\alpha} u\in L^2(I\times\varOmega) \textrm{ for } \langle \alpha \rangle\leq m  \},
\end{align*}
equipped with the norm ${\|}\cdot{\|}_{H^m_*(I\times \varOmega)}$, where
\begin{align} \label{norm.def}
{\|}u{\|}_{H^m_*(I\times \varOmega)}^2:=
\sum_{\langle \alpha\rangle\leq m} \|\mathrm{D}_*^{\alpha} u\|_{L^2(I \times\varOmega)}^2.
\end{align}
We will abbreviate
\begin{align} \label{C.rm:def}
\|{u}\|_{m,*,t}:={\|}u{\|}_{H^m_*( \varOmega_t)},\quad
\mathring{\rm C}_{m}:=1+{\|}(\mathring{V},\mathring{\varPsi}){\|}_{m,*,T}^2.
\end{align}
Clearly,
$H^m(I\times \varOmega)\hookrightarrow H_*^m(I\times \varOmega) \hookrightarrow H^{\lfloor m/2\rfloor}(I\times \varOmega)$ for all $m\in \mathbb{N}$  and $I\subset \mathbb{R}.$
\end{itemize}

The well-posedness for the effective linear problem \eqref{ELP1} are provided in the following theorem.

\begin{theorem}
\label{thm:linear}
Let $K_0>0$ and $m\in\mathbb{N}$ with $m\geq 6$.
Then there exist constants $T_0>0$ and $C(K_0)>0$ such that if for some $0<T\leq T_0$,
the basic state $(\mathring{U}, \mathring{\varphi})$ satisfies \eqref{bas1a}--\eqref{bas1c},
the source terms $f\in H_*^{m}(\varOmega_T)$, $g\in H^{m+1}(\varSigma_T)$ vanish in the past,
and $\mathring{V}:=\mathring{U}-\widebar{U}\in H_*^{{m+4}}(\varOmega_T)$, $\mathring{\varphi}\in H^{{m+4}}(\varSigma_T)$ satisfy
\begin{align}
\label{H:thm.linear}
{\|}\mathring{V}{\|}_{{10},*,T}+\|\mathring{\varphi}\|_{H^{{10}}(\varSigma_T)}\leq {K}_0,
\end{align}
then the problem \eqref{ELP1} has a unique solution $(\dot{V},\psi)\in H_*^{m}(\varOmega_T)\times H^{m}(\varSigma_T)$ satisfying the tame estimate
\begin{align}
&{\|}(\dot{V},\varPsi){\|}_{m,*,T}+\|(\psi,\mathrm{D}_{x'}\psi)\|_{H^{m}(\varSigma_T)}
\nonumber \\
&
\quad \leq   C(K_0)\Big\{  {\|}f{\|}_{m,*,T} +\|g\|_{H^{m+1}(\varSigma_T)}
\nonumber \\
&\qquad \qquad\quad \ \  +{\|}(\mathring{V},\mathring{\varPsi}){\|}_{{m+4},*,T}\left({\|}f{\|}_{6,*,T}+\|g\|_{H^{7}(\varSigma_T)}\right) \Big\}.
  \label{tame}
\end{align}
\end{theorem}

The above assumption that the source terms $f$ and $g$ vanish in the past corresponds to the nonlinear problem with zero initial data. The case of general initial data will be reduced in the subsequent nonlinear analysis.

\subsection{$H_*^1$ \textit{a priori} Estimate}\label{sec:linear2}

This subsection is devoted to obtaining the {\it a priori} estimate in $H_*^1$ for solutions of the linearized problem \eqref{ELP1} by exploiting the stabilization effect of the surface tension on the evolution of the interface. For clear presentation, we divide this subsection into five parts.

\subsubsection{Partial Homogenization}
It is convenient to reduce the problem \eqref{ELP1} to homogeneous boundary conditions.
As in \cite[Section 3.3]{TW21MR4201624}, there exists a function $V_{\natural}\in H_*^{m+2}(\varOmega_T)$ vanishing in the past such that
\begin{align} \label{V.natural}
\mathbb{B}'_e(\mathring{U}, \mathring{\varphi}) (V_{\natural},0)\big|_{\varSigma_T}=g,\ \
{\|}V_{\natural}{\|}_{s+2,*,T}\lesssim \|g\|_{H^{s+1}(\varSigma_T)}
\ \   \textrm{for }s=0,\ldots,m.
\end{align}
Consequently, the vector $V_{\flat}:=\dot{V}-V_{\natural}$ solves
\begin{subequations} \label{ELP2}
 \begin{alignat}{3}
 \label{ELP2a}
 &\mathbb{L}'_{e}(\mathring{U}, \mathring{\varPhi}) {V}=\tilde{f}
 := f-\mathbb{L}_{e}'(\mathring{U},\mathring{\varPhi})V_{\natural}
 &\qquad &\textrm{if } x_1>0,\\
 \label{ELP2b}
 &\mathbb{B}'_e(\mathring{U}, \mathring{\varphi}) ({V},\psi)=0
 &&\textrm{if } x_1=0,\\
 \label{ELP2c}
 &({V},\psi)=0   &&\textrm{if } t<0,
 \end{alignat}
\end{subequations}
where we have dropped subscript ``$\flat$'' for simplicity of notation.

To separate the noncharacteristic variables from others for the problem \eqref{ELP2},
we introduce the new unknown
\begin{align} \nonumber
W:=(q, v_1-\p_2\mathring{\varPhi}v_2-\p_3\mathring{\varPhi}v_3, v_2, v_3, H_1,H_2,H_3, S)^{\mathsf{T}}=J( \mathring{\varPhi})^{-1}V,
\end{align}
with
\begin{align}
\label{J.def}
& {J}(\varPhi):=
\begin{pmatrix}1 & 0 & 0 & 0 & 0 & 0 & 0 & 0  \\
0 & 1 & \partial_2 {\varPhi}  & \partial_3 {\varPhi}  & 0 & 0 & 0 & 0  \\
0 & 0 & 1 & 0 & 0 & 0 & 0 & 0  \\
0 & 0 & 0 & 1 & 0 & 0 & 0 & 0  \\
0 & 0 & 0 & 0 & 1& 0 & 0 & 0  \\
0& 0 & 0 & 0 &0  & 1& 0 & 0  \\
0& 0 & 0 & 0 & 0  & 0 & 1& 0  \\
\w0 & \w0 & \w0 & \w0 & \w0 & \w0 & \w0 & 1
\end{pmatrix}.
\end{align}
For $i=0,2,3$, we denote
\begin{align*}
{\bm{A}}_1:=J( \mathring{\varPhi})^{\mathsf{T}}\widetilde{A}_1(\mathring{U},\mathring{\varPhi})J( \mathring{\varPhi}),\
{\bm{A}}_4:=J( \mathring{\varPhi})^{\mathsf{T}}\mathbb{L}_e'(\mathring{U},\mathring{\varPhi})J( \mathring{\varPhi}),\
{\bm{A}}_{i}:=J( \mathring{\varPhi})^{\mathsf{T}}{A}_{i}(\mathring{U})J( \mathring{\varPhi}).
\end{align*}
Then we reformulate the problem \eqref{ELP2} into
\begin{subequations}
 \label{ELP3}
 \begin{alignat}{3}
 \label{ELP3a}
 &{{\bf L}}W:=\sum_{i=0}^3{\bm{A}}_i\p_i W +{\bm{A}}_4 W =\bm{f}
 &\quad &\textnormal{in }\varOmega_T,\\
 \label{ELP3b}
 &W_2=(\p_t+\mathring{v}_2\p_2+\mathring{v}_3\p_3)\psi-\p_1\mathring{v}\cdot\mathring{N} \psi
 =:{\rm B}\psi
 &\quad &\textnormal{on }\varSigma_T,\\[1mm]
 \label{ELP3c}
&W_1=-\p_1\mathring{q} \psi +\mathfrak{s}\mathrm{D}_{x'}\cdot
\bigg(\dfrac{\mathrm{D}_{x'}\psi}{|\mathring{N}|}-
\dfrac{\mathrm{D}_{x'}\mathring{\varphi}\cdot\mathrm{D}_{x'}\psi}{|\mathring{N}|^3}\mathrm{D}_{x'}\mathring{\varphi}
\bigg)
 &\quad &\textnormal{on }\varSigma_T,\\[1mm]
 \label{ELP3d}
 &(W,\psi)=0 &\quad &\textnormal{if }t<0,
 \end{alignat}
\end{subequations}
with $\p_0:=\p/\p t$ and $\bm{f}:=J( \mathring{\varPhi})^{\mathsf{T}}\tilde{f}$.
In view of \eqref{bas1b}, we calculate
\begin{align}
\label{decom}
{\bm{A}}_1\big|_{x_1=0}=
\begin{pmatrix}
0 & 1 &0 \\
1 & 0 & 0\\
0 & 0 & O_6
\end{pmatrix}
=:{\bm{A}}_1^{(1)}.
\end{align}
Set ${\bm{A}}_1^{(0)}:={\bm{A}}-{\bm{A}}_1^{(1)}$
so that ${\bm{A}}_1^{(0)}|_{x_1=0}=0.$
According to the kernel of ${\bm{A}}_1|_{x_1=0}$, we denote by $W_{\rm nc}:=(W_1,W_2)^{\mathsf{T}}$ the noncharacteristic part of the unknown $W$.

\subsubsection{$L^2$ Estimate of $W$}

We derive the $L^2$ energy estimate for solutions of \eqref{ELP3} as follows.
Taking the scalar product of the problem \eqref{ELP3a} with $W$ leads to
\begin{align} \label{est1a}
\int_{\varOmega} {\bm{A}}_0W\cdot W(t,x)\d x-\int_{\varSigma_t} {\bm{A}}_1W\cdot W 
\leq C(K)\|(\bm{f},W)\|_{L^2(\varOmega_t)}^2.
\end{align}
From \eqref{decom} and \eqref{ELP3b}--\eqref{ELP3c}, we have
\begin{align}
- {\bm{A}}_1&W\cdot W=-2W_1W_2=-2W_1  {\rm B}\psi
\nonumber\\
=\;&
\p_t\bigg\{\p_1\mathring{q} \psi^2 +\mathfrak{s}\bigg(\dfrac{|\mathrm{D}_{x'}\psi|^2}{|\mathring{N}|}- \dfrac{|\mathrm{D}_{x'}\mathring{\varphi}\cdot\mathrm{D}_{x'}\psi|^2}{|\mathring{N}|^3}\bigg)    \bigg\}
+\mathfrak{s}\mathring{\rm c}_2
\mathrm{D}_{x'}\psi
 \cdot \begin{pmatrix}
\psi\\ \mathrm{D}_{x'}\psi
\end{pmatrix}
\nonumber \\
\nonumber
&+\sum_{i=2,3}\p_i\bigg\{
\p_1\mathring{q} \mathring{v}_i \psi^2 +\mathfrak{s} \mathring{v}_i \bigg(\dfrac{|\mathrm{D}_{x'}\psi|^2}{|\mathring{N}|}-
\dfrac{|\mathrm{D}_{x'}\mathring{\varphi}\cdot\mathrm{D}_{x'}\psi|^2}{|\mathring{N}|^3}
\bigg)\bigg\}
+\mathring{\rm c}_2  \psi^2
\\
&-2\mathfrak{s} \mathrm{D}_{x'}\cdot
\bigg\{\bigg(\dfrac{\mathrm{D}_{x'}\psi}{|\mathring{N}|}-
\dfrac{\mathrm{D}_{x'}\mathring{\varphi}\cdot\mathrm{D}_{x'}\psi}{|\mathring{N}|^3}\mathrm{D}_{x'}\mathring{\varphi}
\bigg){\rm B}\psi
\bigg\}
\quad \textrm{on } \varSigma_T,
\label{iden1}
\end{align}
where $\mathring{\rm c}_s$ denotes a generic and smooth matrix-valued function of
$\{(\mathrm{D}^{\alpha} \mathring{V},\mathrm{D}^{\alpha}\mathring{\varPsi}):|\alpha|\leq s\}$
for any $s\in\mathbb{N}$ whose exact form may change from line to line.
We substitute the above identities into \eqref{est1a} to get
\begin{align}
\nonumber
&\int_{\varOmega} {\bm{A}}_0W\cdot W(t,x)\d x+\int_{\varSigma} \bigg(\p_1\mathring{q} \psi^2+ \mathfrak{s}\dfrac{|\mathrm{D}_{x'}\psi|^2}{|\mathring{N}|^3}\bigg)\d x' \\
&\qquad\quad\leq C(K)\Big(
\|(\bm{f},W)\|_{L^2(\varOmega_t)}^2+ \|(\psi,\sqrt{\mathfrak{s}}\mathrm{D}_{x'}\psi)\|_{L^2(\varSigma_t)}^2
\Big).
\label{est1b}
\end{align}
If the generalized Rayleigh--Taylor sign condition $\p_1\mathring{q} \gtrsim 1$ were assumed, then
we could apply Gr\"{o}nwall's inequality to \eqref{est1b} and obtain the {\it uniform-in-$\mathfrak{s}$ estimate}
\begin{align}
\|W(t)\|_{L^2(\varOmega)}^2+\|(\psi,\sqrt{\mathfrak{s}}\mathrm{D}_{x'}\psi)(t)\|_{L^2(\varSigma)}^2
\leq C(K)  \|\bm{f}\|_{L^2(\varOmega_t)}^2.
\label{est1c}
\end{align}
But in this paper, we focus on the case of positive surface tension and aim to release the sign condition.
For this purpose, we fix the coefficient $\mathfrak{s}>0$ and use integration by parts to get
\begin{align} \label{es1}
\|\psi(t)\|_{L^2(\varSigma)}^2=2\int_{\varSigma_t}\psi\p_t\psi
\lesssim 
\|(\psi,\p_t\psi)\|_{L^2(\varSigma_t)}^2.
\end{align}
Combining \eqref{est1b} and \eqref{es1} gives
\begin{align}
\nonumber
&\|W(t)\|_{L^2(\varOmega)}^2+\|(\psi,\mathrm{D}_{x'}\psi)(t)\|_{L^2(\varSigma)}^2 \\
&\qquad \leq 
C(K)\Big(  \|(\bm{f},W)\|_{L^2(\varOmega_t)}^2+\|(\psi,\mathrm{D}_{x'}\psi,\p_t\psi)\|_{L^2(\varSigma_t)}^2\Big),
\nonumber 
\end{align}
from which we deduce
\begin{align}
\label{est1d}
\|W(t)\|_{L^2(\varOmega)}^2+\|(\psi,\mathrm{D}_{x'}\psi)(t)\|_{L^2(\varSigma)}^2
\leq 
C(K)\Big(   \|\bm{f}\|_{L^2(\varOmega_t)}^2+\|\p_t\psi\|_{L^2(\varSigma_t)}^2\Big).
\end{align}

\subsubsection{$L^2$ Estimate of $\mathrm{D}_{x'}W$}

Let us proceed to close the energy estimate in $H_*^1$.
To this end, we set $k\in\{0,2,3\}$. Applying the operator $\p_k$ to the interior equations \eqref{ELP3a} yields
\begin{align}
\sum_{i=0}^3{\bm{A}}_i\p_i \p_k W  =\p_k\bm{f}
-\p_k({\bm{A}}_4 W)-\sum_{i=0}^3\p_k{\bm{A}}_i\p_i  W.
\label{equ:k}
\end{align}
It follows from \eqref{decom} that $\p_k {\bm{A}}_1=0$ on $\varSigma_T$, and hence
\begin{align} \label{es2}
\|\p_k {\bm{A}}_1(x_1)\|_{L^{\infty}((-\infty,T)\times\mathbb{R}^{2})}
\leq C(K)\sigma(x_1)
\ \ \textrm{for } x_1\geq 0\textrm{ and } k\in\{0,2,3\}.
\end{align}
In view of \eqref{decom} and \eqref{es2}, we take the scalar product of \eqref{equ:k} with $\p_k W$ to get
\begin{align}
\label{est2a}
\int_{\varOmega} {\bm{A}}_0\p_kW\cdot\p_k W(t,x)\d x
-2\int_{\varSigma_t} \p_kW_1 \p_kW_2
\leq C(K)\|(\bm{f},W)\|_{1,*,t}^2,
\end{align}
where the norm $\|\,{\cdot}\,\|_{1,*,t}:=\|\,{\cdot}\,\|_{H_*^1( \varOmega_t)}$ is defined by \eqref{D*}--\eqref{C.rm:def}.
We utilize \eqref{ELP3b}--\eqref{ELP3c} to obtain that on $\varSigma_T$,
\begin{alignat}{3}
-2\p_kW_1 \p_kW_2=\;&-2\p_kW_1\p_k{\rm B}\psi,
\label{iden2}
\end{alignat}
and
\begin{alignat}{3}
-2\p_kW_1\p_k{\rm B}\psi=\;
&
\p_t
\left\{\p_1\mathring{q} (\p_k\psi)^2 \right\}
+\sum_{i=2,3}\p_i\left\{
\p_1\mathring{q} \mathring{v}_i  (\p_k\psi)^2 \right\}
\nonumber \\
&+2\p_k\left(\p_k\p_1\mathring{q} \psi {\rm B}\psi \right)
+ \sum_{  |\alpha|\leq 1}
\mathring{\rm c}_3
\begin{pmatrix}
\psi\\ \p_k \psi
\end{pmatrix} \cdot \begin{pmatrix}
\mathrm{D}_{x'}^{\alpha}\psi \\ \p_t \psi
\end{pmatrix}
\nonumber \\
\nonumber
&
-2\mathfrak{s} \mathrm{D}_{x'}\cdot \bigg\{ \p_k  \bigg(\dfrac{\mathrm{D}_{x'}\psi}{|\mathring{N}|}- \dfrac{\mathrm{D}_{x'}\mathring{\varphi}\cdot\mathrm{D}_{x'}\psi}{|\mathring{N}|^3}\mathrm{D}_{x'}\mathring{\varphi}\bigg)\p_k {\rm B}\psi \bigg\}
\\
&+\underbrace{2\mathfrak{s}\p_k\bigg(\dfrac{\mathrm{D}_{x'}\psi}{|\mathring{N}|}- \dfrac{\mathrm{D}_{x'}\mathring{\varphi}\cdot\mathrm{D}_{x'}\psi}{|\mathring{N}|^3}\mathrm{D}_{x'}\mathring{\varphi} \bigg)\cdot  \mathrm{D}_{x'}\p_k {\rm B}\psi}_{\mathcal{T}_k},
\label{iden3}
\end{alignat}
where $\mathrm{D}_{x'}^{\alpha}:=\p_{2}^{\alpha_{2}}\p_{3}^{\alpha_{3}}$ for $\alpha=(\alpha_{2},\alpha_{3})\in\mathbb{N}^2$.
A lengthy but straightforward calculation gives
\begin{align}
\nonumber
\mathcal{T}_k=\;
&
\mathfrak{s} \p_t\bigg(  \dfrac{|\mathrm{D}_{x'}\p_k\psi|^2}{|\mathring{N}|}- \dfrac{|\mathrm{D}_{x'}\mathring{\varphi}\cdot\mathrm{D}_{x'}\p_k\psi|^2}{|\mathring{N}|^3}
+\mathring{\rm c}_2 \mathrm{D}_{x'}\psi\cdot \mathrm{D}_{x'}\p_k\psi
\bigg)
\\[1.5mm]
\nonumber
&+\sum_{i=2,3}\p_i\bigg\{
\mathfrak{s} \mathring{v}_i \bigg(\dfrac{|\mathrm{D}_{x'}\p_k\psi|^2}{|\mathring{N}|}- \dfrac{|\mathrm{D}_{x'}\mathring{\varphi}\cdot\mathrm{D}_{x'}\p_k\psi|^2}{|\mathring{N}|^3}
+\mathring{\rm c}_2 \mathrm{D}_{x'}\psi \cdot\mathrm{D}_{x'}\p_k\psi  \bigg) \bigg\}\\
&
+ \sum_{  |\alpha|\leq 2}\mathring{\rm c}_3
\begin{pmatrix}
\mathrm{D}_{x'}\psi \\  \p_k \mathrm{D}_{x'}\psi
\end{pmatrix} \cdot \begin{pmatrix}
\p_k\psi\\  \mathrm{D}_{x'}^{\alpha}\psi\\ \p_t\mathrm{D}_{x'}\psi
\end{pmatrix}
\quad \textrm{for }k=0,2,3.
\label{iden4}
\end{align}
Plug \eqref{iden2}--\eqref{iden3} into \eqref{est2a} with $k=2,3$, and use \eqref{iden4} to get
\begin{align}
\nonumber
\sum_{k=2,3}
\bigg(\int_{\varOmega} {\bm{A}}_0\,&\p_kW\cdot\p_k W(t,x)\d x
+\mathfrak{s}\int_{\varSigma}  \dfrac{|\mathrm{D}_{x'}\p_k\psi|^2}{|\mathring{N}|^3} \d x' \bigg)\\
\nonumber
 \leq C(K)\bigg(&\|(\bm{f},W)\|_{1,*,t}^2+\sum_{k=2,3}\int_{\varSigma}\left(|\p_k\psi|^2+|\mathrm{D}_{x'}\psi|| \mathrm{D}_{x'}\p_k \psi|  \right)\d x' \\
&
+   \boldsymbol{\epsilon} \| \p_t \psi \|_{L^2(\varSigma_t)}^2
+   C(\boldsymbol{\epsilon})\sum_{|\alpha|\leq 2} \|(\mathrm{D}_{x'}^{\alpha}\psi, \mathrm{D}_{x'}\p_t\psi)\|_{L^2(\varSigma_t)}^2
\bigg)
\nonumber 
\end{align}
for all $\boldsymbol{\epsilon}>0$.
Here we have applied Young's inequality with a constant $\boldsymbol{\epsilon}$.
Below we will always use the letter $\boldsymbol{\epsilon}$ to denote such temporary constants in analogous situations.
Using Cauchy's inequality and the basic estimate
\begin{align}
\|\mathrm{D}_{x'}\psi(t)\|_{L^2(\varSigma)}\lesssim
\|(\mathrm{D}_{x'}\psi,\mathrm{D}_{x'}\p_t\psi)\|_{L^2(\varSigma_t)},
\label{est:bas1}
\end{align}
we have
\begin{align}
&\|\mathrm{D}_{x'}W(t)\|_{L^2(\varOmega)}^2+
\|\mathrm{D}_{x'}^{2}\psi(t)\|_{L^2(\varSigma)}^2
\nonumber\\
&\quad \leq
C(K)
\bigg(\|(\bm{f},W)\|_{1,*,t}^2+  \sum_{|\alpha|\leq 2}
\|(\mathrm{D}_{x'}^{\alpha}\psi, \p_t \psi,\mathrm{D}_{x'}\p_t\psi)\|_{L^2(\varSigma_t)}^2
\bigg).
\label{est2b}
\end{align}

\subsubsection{$L^2$ Estimate of $\p_t W$}

For $k=0$, we use \eqref{ELP3b}--\eqref{ELP3c} to find
\begin{align}
-2\int_{\varSigma_t} \p_tW_1 \p_tW_2
= \underbrace{2\int_{\varSigma_t} \p_1\mathring{q} \p_t\psi\p_tW_2}_{
\mathcal{J}_1}+\underbrace{2\int_{\varSigma_t} \p_t\p_1\mathring{q} \psi\p_tW_2}_{\mathcal{J}_2}
+\int_{\varSigma_t} \mathcal{T}_0,
\label{est3a}
\end{align}
where $\mathcal{T}_0$ is defined in \eqref{iden3}.
By virtue of \eqref{ELP3b}, we infer
\begin{align}
\mathcal{J}_1=
\underbrace{2\int_{\varSigma_t} \p_1\mathring{q} W_2 \p_tW_2}_{\mathcal{J}_{1a}}
+\underbrace{2\int_{\varSigma_t} \p_1\mathring{q}  \p_tW_2(-\mathring{v}_2\p_2\psi-\mathring{v}_3\p_3\psi+\p_1\mathring{v}\cdot\mathring{N} \psi) }_{\mathcal{J}_{1b}}.
\label{J1.est1}
\end{align}
Passing to the volume integral yields
\begin{align*}
\mathcal{J}_{1a}=\;&-2\int_{\varOmega_t}\p_1(\p_1\mathring{q} W_2 \p_tW_2).
\end{align*}
Consequently, we have
\begin{align}
\mathcal{J}_{1a}
=\,&-2\int_{\varOmega}\p_1\mathring{q} W_2 \p_1W_2\d x
-2\int_{\varOmega_t}\left(\p_1^2\mathring{q} W_2 \p_tW_2 {\,-\,} \p_1\p_t\mathring{q} W_2 \p_1W_2\right)
\nonumber \\[1mm]
\geq \,&
-\boldsymbol{\epsilon}\|\p_1W_2(t)\|_{L^2(\varOmega)}^2
-C(K)C(\boldsymbol{\epsilon}) \| (W_2,\p_tW_2,\p_1W_2)\|_{L^2(\varOmega_t)}^2
\label{J1a.est1}
\end{align}
for all $\boldsymbol{\epsilon}>0$.
It follows from integration by parts that
\begin{align}
\mathcal{J}_{1b}+\mathcal{J}_{2}
&=2\int_{\varSigma}W_2\left\{\p_1\mathring{q}  (-\mathring{v}_2\p_2\psi-\mathring{v}_3\p_3\psi+\p_1\mathring{v}\cdot\mathring{N} \psi)
+\p_t\p_1\mathring{q} \psi
\right\}\d x'
\nonumber \\
& \quad -2\int_{\varSigma_t}W_2\p_t\left\{\p_1\mathring{q}  (-\mathring{v}_2\p_2\psi-\mathring{v}_3\p_3\psi+\p_1\mathring{v}\cdot\mathring{N} \psi)
+\p_t\p_1\mathring{q} \psi
\right\}
\nonumber \\
&\geq
-\|W_2(t)\|_{L^2(\varSigma)}^2-\|W_2\|_{L^2(\varSigma_t)}^2
\nonumber\\[1mm]
&\quad
-C\sum_{ i=0,1}\|\p_t^i\big\{\p_1\mathring{q}  (-\mathring{v}_2\p_2\psi-\mathring{v}_3\p_3\psi+\p_1\mathring{v}\cdot\mathring{N} \psi)
+\p_t\p_1\mathring{q} \psi
\big\} \|_{L^2(\varSigma_t)}^2
\nonumber\\
&\geq
-
\|W_2(t)\|_{L^2(\varSigma)}^2
-C(K)\sum_{|\alpha|\leq 2} \|(W_2,\mathrm{D}_{x'}^{\alpha}\psi, \p_t \psi,\mathrm{D}_{x'}\p_t\psi)\|_{L^2(\varSigma_t)}^2.
\label{J1b.est1}
\end{align}
By virtue of \eqref{iden4} and \eqref{est:bas1}, we get
\begin{align}
\int_{\varSigma_t} \mathcal{T}_0
\geq
\;&\dfrac{\mathfrak{s}}{2}\int_{\varSigma}
 \dfrac{|\mathrm{D}_{x'}\p_t\psi|^2}{|\mathring{N}|^3}\d x'
-C(K)\sum_{|\alpha|\leq 2} \|(\mathrm{D}_{x'}^{\alpha}\psi, \p_t \psi,\mathrm{D}_{x'}\p_t\psi)\|_{L^2(\varSigma_t)}^2.
\label{T0.est}
\end{align}
Substitute \eqref{est3a} into \eqref{est2a} and use \eqref{J1.est1}--\eqref{T0.est}
to obtain
\begin{align}
&\|\p_tW(t)\|_{L^2(\varOmega)}^2+\| \mathrm{D}_{x'}\p_t\psi(t)\|_{L^2(\varSigma)}^2
\nonumber\\
 & \leq
 C(K)\bigg( \|W_2(t)\|_{L^2(\varSigma)}^2 +
 \boldsymbol{\epsilon}\|\p_1W_2(t)\|_{L^2(\varOmega)}^2
+C(\boldsymbol{\epsilon})\|\p_1W_2\|_{L^2(\varOmega_t)}^2
\nonumber\\
&\qquad\qquad \
+C(\boldsymbol{\epsilon})\|(\bm{f},W)\|_{1,*,t}^2
+ \sum_{|\alpha|\leq 2} \|(W_2,\mathrm{D}_{x'}^{\alpha}\psi, \p_t \psi,\mathrm{D}_{x'}\p_t\psi)\|_{L^2(\varSigma_t)}^2
\bigg).
\label{est3b}
\end{align}
For the first term on the right-hand side, a direct computation gives
\begin{align}
\|W_2(t)\|_{L^2(\varSigma)}^2
=\;&-2\int_{\varOmega}W_2\p_1W_2
\lesssim \boldsymbol{\epsilon} \|\p_1W_2(t)\|_{L^2(\varOmega)}^2+{\boldsymbol{\epsilon}}^{-1} \|W(t)\|_{L^2(\varOmega)}^2
\label{est3c}
\end{align}
for all $ \boldsymbol{\epsilon}>0.$
The $L^2(\varOmega)$--norm of $\p_1W_2$ can be controlled by virtue of \eqref{ELP3a} and \eqref{decom}. More precisely, from \eqref{ELP3a} and \eqref{decom}, we have
\begin{align} \label{iden5}
\begin{pmatrix}
\p_1 W_2\\\p_1 W_1\\0
\end{pmatrix}
=\bm{f}-{\bm{A}}_4 W-\sum_{i=0,2,3}{\bm{A}}_i\p_i W-{\bm{A}}_1^{(0)}\p_1W,
\end{align}
which implies the estimate
\begin{align}
\|\p_1 W_{\rm nc}(t)\|_{L^2(\varOmega)}^2
&\leq C(K) \sum_{\langle\beta\rangle\leq 1}\|\mathrm{D}_*^{\beta}W(t)\|_{L^2(\varOmega)}^2+C(K)\|\bm{f}(t)\|_{L^2(\varOmega)}^2
\label{est3d}
\end{align}
for the noncharacteristic variables $W_{\rm nc}:=(W_1,W_2)^{\mathsf{T}}$.
Here we recall that the operator $\mathrm{D}_*^{\beta}$ is defined by \eqref{D*}.
Substituting \eqref{est3d} into \eqref{est3c}, we deduce
\begin{align}
 \|W_2(t)\|_{L^2(\varSigma)}^2
 \leq C(K) \bigg( \boldsymbol{\epsilon} \sum_{\langle\beta\rangle\leq 1}\|\mathrm{D}_*^{\beta}W(t)\|_{L^2(\varOmega)}^2+C(\boldsymbol{\epsilon})\|(\bm{f},W)\|_{1,*,t}^2\bigg).
\label{est3e}
\end{align}
Then we combine \eqref{est3b} and \eqref{est3d}--\eqref{est3e} to obtain
\begin{align}
\nonumber
&\|\p_t W(t)\| _{L^2(\varOmega)}^2  +\,
 \boldsymbol{\epsilon}\|\p_1 W_{2}(t)\|_{L^2(\varOmega)}^2
+ \|(W_2, \mathrm{D}_{x'}\p_t\psi)(t)\|_{L^2(\varSigma)}^2\\[2mm]
\nonumber
 &\leq C(K) \bigg( C(\boldsymbol{\epsilon}) \|(\bm{f},W)\|_{1,*,t}^2
 +C(\boldsymbol{\epsilon})\|\p_1 W_2\|_{L^2(\varOmega_t)}^2
  +\boldsymbol{\epsilon} \sum_{\langle\beta\rangle\leq 1}\|\mathrm{D}_*^{\beta}W(t)\|_{L^2(\varOmega)}^2
 \\
&\qquad\qquad\ \
 + \sum_{|\alpha|\leq 2} \|(W_2,\mathrm{D}_{x'}^{\alpha}\psi, \p_t \psi,\mathrm{D}_{x'}\p_t\psi)\|_{L^2(\varSigma_t)}^2\bigg)
  \ \textrm{ for all }\boldsymbol{\epsilon}>0.
\label{est3f}
\end{align}

\subsubsection{$H_*^1$ Estimate of $W$}

In view of the boundary condition \eqref{ELP3b}, we infer
\begin{align}
\|\p_t\psi\|_{L^2(\varSigma_t)}^2
\leq C(K)\|(W_2,\psi,\mathrm{D}_{x'}\psi)\|_{L^2(\varSigma_t)}^2.
\label{est3g}
\end{align}
Apply operator $\sigma\p_1$ to \eqref{ELP3a} and take the resulting equation with $\sigma \p_1 W$ to get
\begin{align}
\|\sigma \p_1W(t)\|_{L^2(\varOmega)}^2 \leq C(K)\|(\bm{f},W)\|_{1,*,t}^2.
\nonumber
\end{align}
Combining the last estimate with \eqref{est1d}, \eqref{est2b},
and \eqref{est3f}--\eqref{est3g}, we choose $\boldsymbol{\epsilon}>0$ small enough to obtain
\begin{align}
\nonumber&
\sum_{\langle\beta\rangle\leq 1}\|\mathrm{D}_*^{\beta}W(t)\|_{L^2(\varOmega)}^2
+\|\p_1 W_{2}(t)\|_{L^2(\varOmega)}^2
+ \sum_{|\alpha|\leq 2}\|(W_2, \mathrm{D}_{x'}^{\alpha}\psi,  \mathrm{D}_{x'}\p_t\psi)(t)\|_{L^2(\varSigma)}^2\\
&  \ \ \leq  C(K) \bigg(\|(\bm{f},W)\|_{1,*,t}^2
+\|\p_1 W_2\|_{L^2(\varOmega_t)}^2
+\sum_{|\alpha|\leq 2}\|(W_2, \mathrm{D}_{x'}^{\alpha}\psi,   \mathrm{D}_{x'}\p_t\psi)\|_{L^2(\varSigma_t)}^2 \bigg).
\nonumber
\end{align}
By virtue of Gr\"{o}nwall's inequality, we infer
\begin{align}
&\sum_{\langle\beta\rangle\leq 1}\|\mathrm{D}_*^{\beta}W(t)\|_{L^2(\varOmega)}^2
+ \sum_{|\alpha|\leq 2}\|(W_2, \mathrm{D}_{x'}^{\alpha}\psi,  \mathrm{D}_{x'}\p_t\psi)(t)\|_{L^2(\varSigma)}^2
 \leq  C(K)  \|\bm{f} \|_{1,*,t}^2,
\label{est4a}
\end{align}
which together with \eqref{est3d}, \eqref{est3g}, and \eqref{ELP3c} yields
\begin{multline}
	\|W\|_{1,*,T}+ \|\partial_1W_{\rm nc}\|_{L^2(\varOmega_T)}+\|W_{\rm nc}\|_{L^2(\varSigma_T)} \\
	+ \|(\psi, \mathrm{D}_{x'}\psi)\|_{H^1(\varSigma_T)}
	\leq  C(K)  \|\bm{f}\|_{1,*,T}.
\label{est4}
\end{multline}
This is the desired $H_*^1$ {\it a priori} estimate for solutions $W$ of the linearized problem \eqref{ELP3},
which also contains the $L^2$ estimate for the traces of the noncharacteristic variables $W_{\rm nc}:=(W_1,W_2)^{\mathsf{T}}$ on the boundary $\{x_1=0\}$.

\subsection{Well-Posedness in $H_*^1$}\label{sec:linear3}

The main purpose of this subsection is to construct the unique solution of the linearized problem \eqref{ELP3}.
Since the $L^2$ {\it a priori} estimate \eqref{est1b} is not closed without assuming the generalized Rayleigh--Taylor sign condition, the duality argument cannot be applied directly for the solvability of the problem \eqref{ELP3}.
To overcome this difficulty, we design for the linearized problem \eqref{ELP3} some suitable $\varepsilon$--regularization, for which we can deduce an $L^2$ {\it a priori} estimate with a constant $C(\varepsilon )$ depending on the small parameter $\varepsilon\in(0,1)$.
It will turn out that an $L^2$ {\it a priori} estimate can be also achieved for the corresponding dual problem. Then we prove the existence and uniqueness of solutions in $L^2$ for any fixed and small parameter $\varepsilon\in(0,1)$ by the duality argument. However, our constant $C(\varepsilon)$ tends to infinity as $\varepsilon\to 0$, and hence we are not able to use the $L^2$ estimate obtained for the regularized problem to take the limit $\varepsilon\to 0$. To deal with this situation, we then derive a uniform-in-$\varepsilon$ estimate in $H_*^1$ for the regularized problem, which enables us to solve the linearized problem \eqref{ELP3} by passing to the limit $\varepsilon\to 0$.

More precisely, we define the following regularized problem:
\begin{subequations} \label{Reg}
\begin{alignat}{3}
	\label{Reg.a}
	&{{\bf L}}_{\varepsilon}W:=\sum_{i=0}^3{\bm{A}}_i\p_i W-\varepsilon\bm{J}\p_1 W +{\bm{A}}_4 W =\bm{f}
	&\quad &\textnormal{in }\varOmega_T,\\
	\label{Reg.b}
	&W_2=(\p_t+\mathring{v}_2\p_2+\mathring{v}_3\p_3)\psi-\p_1\mathring{v}\cdot\mathring{N} \psi
	+\varepsilon(\p_2^4+\p_3^4)\psi
	&\quad &\textnormal{on }\varSigma_T,\\[1mm]
	\label{Reg.c}
	&W_1=-\p_1\mathring{q} \psi +\mathfrak{s}\mathrm{D}_{x'}\cdot
	\bigg(\dfrac{\mathrm{D}_{x'}\psi}{|\mathring{N}|}-
	\dfrac{\mathrm{D}_{x'}\mathring{\varphi}\cdot\mathrm{D}_{x'}\psi}{|\mathring{N}|^3}\mathrm{D}_{x'}\mathring{\varphi}
	\bigg)
	&\quad &\textnormal{on }\varSigma_T,\\[1mm]
	\label{Reg.d}
	&(W,\psi)=0 &\quad &\textnormal{if }t<0,
	\end{alignat}
\end{subequations}
where $ \bm{J}:={\rm diag}\,(0,1,0,\ldots,0).$
The term $-\varepsilon\bm{J}\p_1 W $ containing in the interior equations \eqref{Reg.a} helps us to deduce the $L^2$ energy estimate for the problem \eqref{Reg},
while the term $\varepsilon(\p_2^4+\p_3^4)\psi$ added in the boundary condition \eqref{Reg.b} is particularly useful in the derivation of the $L^2$ energy estimate for the dual problem of \eqref{Reg}.

This subsection is divided into three parts as follows.

\subsubsection{$L^2$ Estimate for the Regularized Problem}
Let us first show the $L^2$ {\it a priori} estimate for the problem \eqref{Reg}.
Taking the scalar product of \eqref{Reg.a} with $W$ implies
\begin{align} \label{est:0a.R}
\int_{\varOmega} {\bm{A}}_0W\cdot W(t,x)\d x+\int_{\varSigma_t} (\varepsilon\bm{J}-{\bm{A}}_1)W\cdot W
\leq C(K)\|(\bm{f},W)\|_{L^2(\varOmega_t)}^2.
\end{align}
By virtue of \eqref{decom} and \eqref{Reg.b}, we have
\begin{align}
(\varepsilon\bm{J}-{\bm{A}}_1)W\cdot W
=\varepsilon W_2^2-2 W_1   {\rm B}\psi
-2\varepsilon \sum_{i=2,3}W_1 \p_i^4\psi
\quad \textrm{on }\varSigma_T,
\label{est:0b.R}
\end{align}
where the operator ${\rm B}$ is defined in \eqref{ELP3b}. For $i=2,3$, we get from \eqref{Reg.c} that
\begin{align}
-\int_{\varSigma_t} W_1 \p_i^4\psi
&=\mathfrak{s}\int_{\varSigma_t}\p_i^2
\bigg(\dfrac{\mathrm{D}_{x'}\psi}{|\mathring{N}|}-
\dfrac{\mathrm{D}_{x'}\mathring{\varphi}\cdot\mathrm{D}_{x'}\psi}{|\mathring{N}|^3}\mathrm{D}_{x'}\mathring{\varphi}
\bigg)
\cdot \p_i^2\mathrm{D}_{x'}\psi+
\int_{\varSigma_t} \p_i^2(\p_1\mathring{q} \psi )\p_i^2\psi
\nonumber
\\
 &  \geq
\mathfrak{s}\int_{\varSigma_t} \dfrac{|\p_i^2\mathrm{D}_{x'}\psi|^2}{|\mathring{N}|^3}
-\int_{\varSigma_t}  \left|[\p_i^2,\mathring{\rm c}_1]\mathrm{D}_{x'}\psi \cdot \p_i^2\mathrm{D}_{x'}\psi
+   \p_i^2(\p_1\mathring{q} \psi )\p_i^2\psi\right|
\nonumber \\
&  \geq
\frac{\mathfrak{s}}{2}\int_{\varSigma_t} \dfrac{|\p_i^2\mathrm{D}_{x'}\psi|^2}{|\mathring{N}|^3}
-C(K)\sum_{|\alpha|\leq 2}\|\mathrm{D}_{x'}^{\alpha}\psi\|_{L^2(\varSigma_t)}^2.
\label{est:0c.R}
\end{align}
Substituting \eqref{est:0b.R} into \eqref{est:0a.R}, we utilize \eqref{est:0c.R} and the last identity in \eqref{iden1} to infer
\begin{align}
&\|W(t)\|_{L^2(\varOmega)}^2+\|\mathrm{D}_{x'}\psi(t)\|_{L^2(\varSigma)}^2
+\varepsilon\|(W_2,\p_2^2\mathrm{D}_{x'}\psi,\p_3^2\mathrm{D}_{x'}\psi)\|_{L^2(\varSigma_t)}^2
\nonumber\\[1mm]
& \leq
C(K)\Big(
\|(\bm{f},W)\|_{L^2(\varOmega_t)}^2
+\sum_{|\alpha|\leq 2}\|(\psi, \mathrm{D}_{x'}\psi,\sqrt{\varepsilon}\mathrm{D}_{x'}^{\alpha} \psi)\|_{L^2(\varSigma_t)}^2
+\|\psi(t)\|_{L^2(\varSigma)}^2
\Big).
\label{est:0d.R}
\end{align}

To estimate the last term in \eqref{est:0d.R}, we multiply \eqref{Reg.b} with $\psi$ and obtain
\begin{align}
\nonumber &\|\psi(t)\|_{L^2(\varSigma)}^2
+2\varepsilon \|(\p_2^2\psi,\p_3^2\psi)\|_{L^2(\varSigma_t)}^2
 \\
&\qquad  \leq C(K)\|\psi\|_{L^2(\varSigma_t)}^2
+2\int_{\varSigma_t}|\psi W_2|
\nonumber \\
 & \qquad \leq \boldsymbol{\epsilon}\varepsilon \|W_2\|_{L^2(\varSigma_t)}^2
+\left(C(K)+\boldsymbol{\epsilon}^{-1}\varepsilon^{-1}\right)\|\psi\|_{L^2(\varSigma_t)}^2
\quad \textrm{for all }\boldsymbol{\epsilon}>0.
\label{est:0e.R}
\end{align}
Since
\begin{align}
\|\p_2\p_3\mathrm{D}_{x'}^{\alpha}\psi\|_{L^2(\varSigma)}^2
=\int_{\varSigma}\p_2^2\mathrm{D}_{x'}^{\alpha}\psi \p_3^2\mathrm{D}_{x'}^{\alpha}\psi
\leq  \|(\p_2^2\mathrm{D}_{x'}^{\alpha}\psi,\p_3^2\mathrm{D}_{x'}^{\alpha}\psi)\|_{L^2(\varSigma)}^2
\label{est:0f.R}
\end{align}
for any $\alpha\in\mathbb{N}^2$, it follows from \eqref{est:0e.R} that
\begin{align}
& \|\psi(t)\|_{L^2(\varSigma)}^2
+ \varepsilon\|\mathrm{D}_{x'}^{2} \psi \|_{L^2(\varSigma_t)}^2
   \leq \boldsymbol{\epsilon}\varepsilon \|W_2\|_{L^2(\varSigma_t)}^2
+C(K)C(\boldsymbol{\epsilon}\varepsilon)\|\psi\|_{L^2(\varSigma_t)}^2
\label{est:0e.R1}
\end{align}
for all $\boldsymbol{\epsilon}>0$.
Here we have formally that $C(\boldsymbol{\epsilon}\varepsilon)\to +\infty$ as $\varepsilon\to 0$. The same is true for all constants depending on $\varepsilon$ which appear below. This is however unimportant because the parameter $\varepsilon$ is now fixed.
Plug \eqref{est:0e.R1} into \eqref{est:0d.R}, take $\boldsymbol{\epsilon}>0$ sufficiently small,
and use \eqref{est:0f.R} with $|\alpha|=1$ to deduce
\begin{multline}
\|W(t)\|_{L^2(\varOmega)}^2+\|(\psi,\mathrm{D}_{x'}\psi)(t)\|_{L^2(\varSigma)}^2
+\varepsilon \|(W_2,\mathrm{D}_{x'}^{2} \psi,\mathrm{D}_{x'}^{3} \psi)\|_{L^2(\varSigma_t)}^2
\nonumber\\
 \leq
C(K)\left(
\|(\bm{f},W)\|_{L^2(\varOmega_t)}^2
+C(\varepsilon)\|(\psi, \mathrm{D}_{x'}\psi)\|_{L^2(\varSigma_t)}^2
\right).
\end{multline}
Then we apply Gr\"{o}nwall's inequality and utilize \eqref{Reg.c} to get
\begin{multline}
\|W(t)\|_{L^2(\varOmega)}^2
+\|(\psi,\mathrm{D}_{x'}\psi)(t)\|_{L^2(\varSigma)}^2
\\
+ \|(W_{\rm nc},\mathrm{D}_{x'}^{2} \psi,\mathrm{D}_{x'}^{3} \psi)\|_{L^2(\varSigma_t)}^2
 \leq
C(K,\varepsilon)   \|\bm{f}\|_{L^2(\varOmega_t)}^2.
\label{est:0h.R}
\end{multline}

Moreover, we differentiate \eqref{Reg.b} two times with respect to $x_i$ and multiply the resulting identity with $\p_i^2\psi$, for $i=2,3$, to deduce
\begin{align}
&\sum_{i=2,3}\|\p_i^2\psi(t)\|_{L^2(\varSigma)}^2+2\varepsilon\sum_{i=2,3} \|(\p_2^2\p_i^2\psi,\p_3^2\p_i^2\psi)\|_{L^2(\varSigma_t)}^2
\nonumber\\
&\quad \ \leq C(K)\sum_{|\alpha|\leq 2}\|\mathrm{D}_{x'}^{\alpha}\psi\|_{L^2(\varSigma_t)}^2
+2\sum_{i=2,3}\int_{\varSigma_t}\left|  \p_i^4\psi W_2\right|
\nonumber \\
& \quad\   \leq \boldsymbol{\epsilon}\varepsilon \sum_{i=2,3}\|\p_i^4\psi\|_{L^2(\varSigma_t)}^2
+ \left(C(K)+\frac{1}{\boldsymbol{\epsilon}\varepsilon}\right) \sum_{|\alpha|\leq 2}\|(W_2,\mathrm{D}_{x'}^{\alpha}\psi)\|_{L^2(\varSigma_t)}^2
\nonumber
\end{align}
for all $\boldsymbol{\epsilon}>0$. Taking $\boldsymbol{\epsilon}>0$ sufficiently small in the last estimate, we use \eqref{est:0f.R} and \eqref{est:0h.R} to infer
\begin{multline}
\|W(t)\|_{L^2(\varOmega)}^2+\sum_{ |\alpha|\leq 2}\|\mathrm{D}_{x'}^{\alpha}\psi(t)\|_{L^2(\varSigma)}^2
 \\
 + \|(W_{\rm nc},\mathrm{D}_{x'}^{2} \psi,\mathrm{D}_{x'}^{3} \psi,\mathrm{D}_{x'}^{4} \psi)\|_{L^2(\varSigma_t)}^2
 \leq
C(K,\varepsilon)   \|\bm{f}\|_{L^2(\varOmega_t)}^2,
\nonumber 
\end{multline}
which combined with \eqref{Reg.b} implies
\begin{align}
\|W\|_{L^2(\varOmega_{T})}+\sum_{ |\alpha|\leq 4}\|(W_{\rm nc},\mathrm{D}_{x'}^{\alpha} \psi,\p_t\psi)\|_{L^2(\varSigma_{T})}
 \leq
C(K,\varepsilon)   \|\bm{f}\|_{L^2(\varOmega_{T})}.
\label{est:0j.R}
\end{align}

The {\it a priori} $L^2$ estimate \eqref{est:0j.R} (together with  the $L^2$ estimate \eqref{dual.es0h} below for the dual problem) is suitable to prove the existence of solutions of the regularized problem \eqref{Reg} for any {\it fixed} and small $\varepsilon>0$. However, since $C(K,\varepsilon)\to +\infty$ as $\varepsilon\to 0$, for passing to the limit $\varepsilon\to 0$ we will have to deduce an additional uniform-in-$\varepsilon$ estimate.

\subsubsection{Existence of Solutions for the Regularized Problem}
We solve the regularized problem \eqref{Reg} by means of the duality argument.
To this end,
it suffices to derive an $L^2$ {\it a priori} estimate without loss of derivatives for the
following dual problem of \eqref{Reg}:
\begin{subequations} \label{dual}
	\begin{alignat}{3}
	\label{dual.a}
	&{{\bf L}}_{\varepsilon}^*W^*=\bm{f}^*
	&\quad &\textnormal{if }x_1>0,\\
	&\p_tw^*+\p_2(\mathring{v}_2w^*)+\p_3(\mathring{v}_3w^*)-\varepsilon(\p_2^4+\p_3^4)w^*
	+\p_1\mathring{v}\cdot \mathring{N}w^*&\quad &
	\nonumber\\
	&\quad +\p_1\mathring{q} W_2^*-\mathfrak{s}\mathrm{D}_{x'}\cdot
	\bigg(\dfrac{\mathrm{D}_{x'}W_2^*}{|\mathring{N}|}-
	\dfrac{\mathrm{D}_{x'}\mathring{\varphi}\cdot\mathrm{D}_{x'}W_2^*}{|\mathring{N}|^3}\mathrm{D}_{x'}\mathring{\varphi}
	\bigg)=0
	&\quad &\textnormal{if }x_1=0,	\label{dual.b}\\
	\label{dual.c}
	&W^*|_{{t>T} }=0, &\quad &
	\end{alignat}
\end{subequations}
where $w^*:=W_1^*-\varepsilon W_2^*$, and ${{\bf L}}_{\varepsilon}^*$ is the formal adjoint of ${{\bf L}}_{\varepsilon}$ ({\it cf.}~\eqref{Reg.a}):
\begin{align}
{{\bf L}}_{\varepsilon}^*:=-\sum_{i=0}^{3}\bm{A}_i\p_i+\varepsilon\bm{J}\p_1+\bm{A}_4^{\mathsf{T}}-\sum_{i=0}^{3}\p_i\bm{A}_i.
\nonumber
\end{align}
The boundary condition \eqref{dual.b} is imposed to guarantee that
\begin{align*}
 &\int_{\varOmega_{T}}\left({{\bf L}}_{\varepsilon}W\cdot W^* -W\cdot {{\bf L}}_{\varepsilon}^* W^*\right)
\\
&\quad =\int_{\varSigma_{T}}\left(\varepsilon\bm{J}-\bm{A}_1 \right)W\cdot W^*
=-\int_{\varSigma_{T}}\left( W_2w^*+W_1W_2^*\right) =0,
\end{align*}
thanks to \eqref{Reg.b}--\eqref{Reg.d} and \eqref{dual.c}.
To derive the energy estimate for the dual problem \eqref{dual}, we let $\tilde{t}:={T}-t$ and define
\[\widetilde{W}^*(\tilde{t},x):={W}^*(t,x), \quad
\widetilde{\bm{f}}^*(\tilde{t},x):= {\bm{f}}^*(t,x).
\]
Then
\begin{subequations} \label{dual2}
	\begin{alignat}{3}
	\label{dual2.a}
	&\bigg(\bm{A}_0\p_{\tilde{t}}-\sum_{i=1}^{3}\bm{A}_i\p_i+\varepsilon\bm{J}\p_1+\bm{A}_4^{\mathsf{T}}-\sum_{i=0}^{3}\p_i\bm{A}_i\bigg)\widetilde{W}^*=\widetilde{\bm{f}}^*
	&\quad &\textnormal{if }x_1>0,\\
	&\p_{\tilde{t}}\widetilde{w}^*-\p_2(\mathring{v}_2\widetilde{w}^*)-\p_3(\mathring{v}_3\widetilde{w}^*)+\varepsilon(\p_2^4+\p_3^4)\widetilde{w}^*
	-\p_1\mathring{v}\cdot \mathring{N}\widetilde{w}^*&\quad &
	\nonumber\\
	&\quad -\p_1\mathring{q} \widetilde{W}_2^*+\mathfrak{s}\mathrm{D}_{x'}\cdot
	\bigg(\dfrac{\mathrm{D}_{x'}\widetilde{W}_2^*}{|\mathring{N}|}-
	\dfrac{\mathrm{D}_{x'}\mathring{\varphi}\cdot\mathrm{D}_{x'}\widetilde{W}_2^*}{|\mathring{N}|^3}\mathrm{D}_{x'}\mathring{\varphi}
	\bigg)=0
	&\quad &\textnormal{if }x_1=0,	\label{dual2.b}\\
	\label{dual2.c}
	&\widetilde{W}^*|_{{\tilde{t}<0} }=0, &\quad &
	\end{alignat}
\end{subequations}
where $\widetilde{w}^*:=\widetilde{W}_1^*-\varepsilon \widetilde{W}_2^*$.
Taking the scalar product of \eqref{dual2.a} with $\widetilde{W}^*$ yields
\begin{align} \nonumber 
\int_{\varOmega} {\bm{A}}_0\widetilde{W}^*\cdot \widetilde{W}^*(\tilde{t},x) \d x
+\int_{\varSigma_{\tilde{t}}} ({\bm{A}}_1-\varepsilon\bm{J})\widetilde{W}^*\cdot \widetilde{W}^*
\leq C(K)\|(\widetilde{\bm{f}}^*,\widetilde{W}^*)\|_{L^2(\varOmega_{\tilde{t}})}^2.
\end{align}
Since
\begin{align*}
({\bm{A}}_1-\varepsilon\bm{J})\widetilde{W}^*\cdot \widetilde{W}^*
=2\widetilde{W}^*_1\widetilde{W}^*_2-\varepsilon (\widetilde{W}^*_2)^2
=2\widetilde{w}^*\widetilde{W}^*_2+\varepsilon (\widetilde{W}^*_2)^2
\ \ \textrm{if }x_1=0,
\end{align*}
we use Cauchy's inequality to infer
\begin{align}
&\|\widetilde{W}^*(\tilde{t})\|_{L^2(\varOmega)}^2
+\varepsilon\|\widetilde{W}^*_2\|_{L^2(\varSigma_{\tilde{t}})}^2
\nonumber\\
&\qquad \leq
C(K)\left(
\|(\widetilde{\bm{f}}^*,\widetilde{W}^*)\|_{L^2(\varOmega_{\tilde{t}})}^2
+\varepsilon^{-1}\| \widetilde{w}^*\|_{L^2(\varSigma_{\tilde{t}})}^2
\right).
\label{dual.es0b}
\end{align}
Multiply the boundary condition \eqref{dual2.b} by $\widetilde{w}^*$ to deduce
\begin{align}
&\|\widetilde{w}^*(\tilde{t})\|_{L^2(\varSigma)}^2+2\varepsilon\|(\p_2^2\widetilde{w}^*,\p_3^2\widetilde{w}^*)\|_{L^2(\varSigma_{\tilde{t}})}^2
\nonumber \\[1mm]
&
\leq C(K) \|(\widetilde{w}^*,\widetilde{W}^*_2)\|_{L^2(\varSigma_{\tilde{t}})}^2
+2\mathfrak{s}\left| \int_{\varSigma_{\tilde{t}}} \widetilde{W}^*_2 \mathrm{D}_{x'}\cdot
\bigg(\dfrac{\mathrm{D}_{x'}\widetilde{w}^*}{|\mathring{N}|}-
\dfrac{\mathrm{D}_{x'}\mathring{\varphi}\cdot\mathrm{D}_{x'}\widetilde{w}^*}{|\mathring{N}|^3}\mathrm{D}_{x'}\mathring{\varphi}
\bigg) \right|
\nonumber\\[1mm]
&
\leq \boldsymbol{\epsilon}\varepsilon\sum_{ |\alpha|\leq 2}\|\mathrm{D}_{x'}^{\alpha}\widetilde{w}^*\|_{L^2(\varSigma_{\tilde{t}})}^2
+C(K,\boldsymbol{\epsilon}\varepsilon) \|(\widetilde{w}^*,\widetilde{W}^*_2)\|_{L^2(\varSigma_{\tilde{t}})}^2
\ \ \textrm{for all }\boldsymbol{\epsilon}>0.
\label{dual.es0c}
\end{align}
It follows from integration by parts that
\begin{align}
&\|\p_2\p_3\widetilde{w}^*\|_{L^2(\varSigma_{\tilde{t}})}^2
=\int_{\varSigma_{\tilde{t}}}\p_2^2\widetilde{w}^* \p_3^2\widetilde{w}^*
\leq  \|(\p_2^2\widetilde{w}^*,\p_3^2\widetilde{w}^*)\|_{L^2(\varSigma_{\tilde{t}})}^2,
\nonumber \\
&\|\mathrm{D}_{x'}\widetilde{w}^*\|_{L^2(\varSigma_{\tilde{t}})}^2
=-\int_{\varSigma_{\tilde{t}}}\widetilde{w}^*\mathrm{D}_{x'}\cdot \mathrm{D}_{x'} \widetilde{w}^*
\leq
\|(\widetilde{w}^*,\mathrm{D}_{x'}^{2}\widetilde{w}^*)\|_{L^2(\varSigma_{\tilde{t}})}^2.
\nonumber
\end{align}
Plugging the above estimates into \eqref{dual.es0c} and taking $\boldsymbol{\epsilon}>0$ suitably small, we get
\begin{align}
\|\widetilde{w}^*(\tilde{t})\|_{L^2(\varSigma)}^2+\varepsilon\sum_{   |\alpha|\leq 2}\|\mathrm{D}_{x'}^{\alpha}\widetilde{w}^*\|_{L^2(\varSigma_{\tilde{t}})}^2
\leq  C(K,\varepsilon) \|(\widetilde{w}^*,\widetilde{W}^*_2)\|_{L^2(\varSigma_{\tilde{t}})}^2.
\label{dual.es0d}
\end{align}
Combine \eqref{dual.es0b} and \eqref{dual.es0d} to derive
\begin{multline}
\|\widetilde{W}^*(\tilde{t})\|_{L^2(\varOmega)}^2
+\|\widetilde{w}^*(\tilde{t})\|_{L^2(\varSigma)}^2+\sum_{  |\alpha|\leq 2}\|(\widetilde{W}^*_2,\mathrm{D}_{x'}^{\alpha}\widetilde{w}^*)\|_{L^2(\varSigma_{\tilde{t}})}^2
\nonumber\\
 \leq
C(K,\varepsilon)\left(
\|(\widetilde{\bm{f}}^*,\widetilde{W}^*)\|_{L^2(\varOmega_{\tilde{t}})}^2
+\| \widetilde{w}^*\|_{L^2(\varSigma_{\tilde{t}})}^2
\right).
\end{multline}
We apply Gr\"{o}nwall's inequality to the last estimate and obtain
\begin{multline}
\|\widetilde{W}^*(\tilde{t})\|_{L^2(\varOmega)}^2
+\|\widetilde{w}^*(\tilde{t})\|_{L^2(\varSigma)}^2+\sum_{  |\alpha|\leq 2}\|(\widetilde{W}^*_2,\mathrm{D}_{x'}^{\alpha}\widetilde{w}^*)\|_{L^2(\varSigma_{\tilde{t}})}^2
 \nonumber
 \leq
C(K,\varepsilon)
\|\widetilde{\bm{f}}^*\|_{L^2(\varOmega_{\tilde{t}})}^2.
\end{multline}
As a result, for the dual problem \eqref{dual} we obtain the following estimate:
\begin{align}
\|{W}^*\|_{L^2(\varOmega_{{T}})}
+\sum_{  |\alpha|\leq 2}\|({W}^*_2,\mathrm{D}_{x'}^{\alpha}{w}^*)\|_{L^2(\varSigma_{{T}})}
\leq
C(K,\varepsilon)
\|{\bm{f}}^*\|_{L^2(\varOmega_{{T}})}.
\label{dual.es0h}
\end{align}
Here, as in \eqref{est:0h.R}, $C(K,\varepsilon)\to +\infty$ as $\varepsilon\to 0$, but this is not important if we consider a fixed parameter $\varepsilon>0$.

With the $L^2$ estimates \eqref{est:0j.R} and \eqref{dual.es0h} in hand, we can prove the existence and uniqueness of a weak solution $W\in L^2(\varOmega_T)$ of the regularized problem \eqref{Reg} with the traces $W_{\rm nc}|_{x_1=0}$ belonging to $ L^2(\varSigma_T)$ for any fixed and small parameter $\varepsilon\in(0,1)$ by the classical duality argument in \cite{LP60MR0118949}.

We then consider \eqref{Reg.b} as a fourth-order parabolic equation for $\psi$ with the given source term $W_2|_{x_1=0}\in L^2(\varSigma_T)$ and zero initial data $\psi |_{t=0}=0$ ({\it cf.}~\eqref{Reg.d}).
Referring to \cite[Theorem 5.2]{CP82MR0678605},  we conclude that the Cauchy problem for this parabolic equation has a unique solution $\psi\in C ([0,T],H^4(\mathbb{R}^2))\bigcap C^1 ([0,T],L^2(\mathbb{R}^2))$ (implying that $\psi \in L^2((-\infty ,T],H^4(\mathbb{R}^2))$ and $\partial_t\psi\in L^2(\varSigma_T)$).
In fact, we have already obtained the {\it a priori} estimate for solutions $\psi$ of this Cauchy problem in \eqref{est:0j.R}.

Therefore, we have proved the existence of a unique solution $(W,\psi )\in L^2(\varOmega_T)\times L^2((-\infty ,T],H^4(\mathbb{R}^2))$ of the regularized problem \eqref{Reg} for any fixed and small parameter $\varepsilon>0$ with $W_{\rm nc}|_{x_1=0}\in L^2(\varSigma_T)$ and $\partial_t\psi\in L^2(\varSigma_T)$.
Then, tangential differentiation gives the existence of  a unique solution  $(W,\psi)\in H^1_*(\varOmega_T)\times H^1((-\infty ,T],H^4(\mathbb{R}^2))$, again for any fixed and small parameter $\varepsilon>0$.
Moreover, in next subsection we will prove a uniform-in-$\varepsilon$ {\it a priori} estimate for this solution.

\subsubsection{Uniform Estimate and Passing to the Limit}
We are going to show the {\it uniform-in-$\varepsilon$} energy estimate for the regularized problem \eqref{Reg} in $H_*^1$.
\vspace*{4mm}

\noindent{\bf $\bullet$ $L^2$ estimate of $W$.}\quad
Substitute \eqref{es1} into \eqref{est:0d.R} and use \eqref{est:0f.R} to get
\begin{multline}
\|W(t)\|_{L^2(\varOmega)}^2+\|(\psi,\mathrm{D}_{x'}\psi)(t)\|_{L^2(\varSigma)}^2
+\varepsilon \|(W_2,  \mathrm{D}_{x'}^{3}\psi)\|_{L^2(\varSigma_t)}^2
\\
 \leq
C(K)\Big(
\|(\bm{f},W)\|_{L^2(\varOmega_t)}^2
+ \sum_{  |\alpha|\leq 2}\|(\psi, \mathrm{D}_{x'}\psi,\p_t\psi,\sqrt{\varepsilon}\mathrm{D}_{x'}^{\alpha}\psi)\|_{L^2(\varSigma_t)}^2
\Big).
\label{uni1}
\end{multline}
\vspace*{2mm}

\noindent{\bf $\bullet$ $L^2$ estimate of $\mathrm{D}_{x'}W$.}\quad
Let  $k\in\{0,2,3\}$. Apply the operator $\p_k$ to \eqref{Reg.a}, take the scalar product of the resulting equation with $\p_k W$, and use \eqref{es2} to discover
\begin{align}
\int_{\varOmega} {\bm{A}}_0\p_kW\cdot\p_k W \d x
+\int_{\varSigma_t} (\varepsilon\bm{J}-\bm{A}_1)\p_kW\cdot \p_kW
\leq C(K)\|(\bm{f},W)\|_{1,*,t}^2.
\label{uni2}
\end{align}
It follows from \eqref{decom} that
\begin{align}
&(\varepsilon\bm{J}-\bm{A}_1)\p_kW\cdot \p_kW
=\varepsilon (\p_k W_2)^2-2\p_kW_1\p_kW_2
\nonumber\\
&\quad =\varepsilon (\p_k W_2)^2
-2\p_kW_1\p_k{\rm B}\psi
-2\varepsilon\sum_{i=2,3}\p_kW_1\p_k\p_i^4\psi
\quad \textrm{on } \varSigma_T,
\label{uni2a}
\end{align}
with ${\rm B}$ being defined in \eqref{ELP3b}.
In view of \eqref{Reg.c}, we have
\begin{align}
-2\varepsilon\sum_{i=2,3}\int_{\varSigma_t}  \p_kW_1\p_k\p_i^4\psi
=\underbrace{-2\varepsilon\sum_{i=2,3}\int_{\varSigma_t}  \p_i\p_k(\p_1\mathring{q}\psi)\p_k\p_i^3\psi}_{\mathcal{J}_3^{(k)}}
+\mathcal{J}_4^{(k)},
\label{uni2b}
\end{align}
where
\begin{align}
\mathcal{J}_4^{(k)}:=2\varepsilon \mathfrak{s} \sum_{i=2,3}\int_{\varSigma_t}
\p_k\p_i^2
\left(\dfrac{\mathrm{D}_{x'}\psi}{|\mathring{N}|}-
\dfrac{\mathrm{D}_{x'}\mathring{\varphi}\cdot\mathrm{D}_{x'}\psi}{|\mathring{N}|^3}\mathrm{D}_{x'}\mathring{\varphi}
\right)\cdot\p_k\p_i^2\mathrm{D}_{x'}\psi.
\label{J5.cal}
\end{align}
Applying Cauchy's inequality yields
\begin{align}
\mathcal{J}_4^{(k)}
=\;& 2\varepsilon \mathfrak{s}\sum_{i=2,3}\int_{\varSigma_t}
\bigg(\dfrac{|\p_k\p_i^2\mathrm{D}_{x'}\psi|^2}{|\mathring{N}|}-
\dfrac{|\mathrm{D}_{x'}\mathring{\varphi}\cdot\p_k\p_i^2\mathrm{D}_{x'}\psi|^2}{|\mathring{N}|^3}
 \bigg)
\nonumber\\
&+2\varepsilon \mathfrak{s} \sum_{i=2,3}\int_{\varSigma_t}
[\p_k\p_i^2,\, h(\mathrm{D}_{x'}\mathring{\varphi})]  \mathrm{D}_{x'}\psi\cdot \p_k\p_i^2\mathrm{D}_{x'}\psi
\nonumber\\
\geq \;&\varepsilon \mathfrak{s} \sum_{i=2,3}\int_{\varSigma_t}\dfrac{|\p_k\p_i^2\mathrm{D}_{x'}\psi|^2}{|\mathring{N}|^3}
-\varepsilon C(K) \sum_{|\alpha|\leq 2}
\|(\mathrm{D}_{x'}^{\alpha}\p_k\psi,
\mathrm{D}_{x'}^{\alpha}\mathrm{D}_{x'}\psi)\|_{L^2(\varSigma_t)}^2,
\label{uni2c}\\
 \mathcal{J}_3^{(k)} \geq
\;& -\frac{\varepsilon \mathfrak{s}}{2} \sum_{i=2,3}\int_{\varSigma_t}\dfrac{|\p_k\p_i^2\mathrm{D}_{x'}\psi|^2}{|\mathring{N}|^3}
-\varepsilon C(K)\sum_{|\alpha|\leq 1}\|(  \mathrm{D}_{x'}^{\alpha}\psi,\mathrm{D}_{x'}^{\alpha} \p_k\psi)\|_{L^2(\varSigma_t)}^2,
\label{uni2d}
\end{align}
where $h(\mathrm{D}_{x'}\mathring{\varphi})$ is some smooth matrix-valued function of $\mathrm{D}_{x'}\mathring{\varphi}$.
Plugging \eqref{uni2a} into \eqref{uni2}, we use \eqref{iden3}--\eqref{est:bas1},  \eqref{uni2c}--\eqref{uni2d}, and \eqref{est:0f.R} to infer
\begin{align}
\|\mathrm{D}_{x'}W(t)\|_{L^2(\varOmega)}^2
&+ \|\mathrm{D}_{x'}^{2}\psi(t)\|_{L^2(\varSigma)}^2
+\varepsilon \|(\mathrm{D}_{x'}W_2,\mathrm{D}_{x'}^{4}\psi)\|_{L^2(\varSigma_t)}^2
\nonumber\\[1.5mm]
 \leq C(K)\bigg(&\|(\bm{f},W)\|_{1,*,t}^2
 + \varepsilon\sum_{|\alpha|\leq 2}  \|\mathrm{D}_{x'}^{\alpha}\mathrm{D}_{x'}\psi\|_{L^2(\varSigma_t)}^2
 +   \boldsymbol{\epsilon} \| \p_t \psi \|_{L^2(\varSigma_t)}^2
\nonumber\\
&
+C(\boldsymbol{\epsilon})\sum_{|\alpha|\leq 2}
\|(\mathrm{D}_{x'}^{\alpha}\psi,  \mathrm{D}_{x'}\p_t\psi)\|_{L^2(\varSigma_t)}^2
 \bigg)
 \quad \textrm{for all }\boldsymbol{\epsilon}>0.
\label{uni3}
\end{align}
In view of \eqref{Reg.b}, we obtain
\begin{align}
\| \p_t \psi \|_{L^2(\varSigma_t)}^2
\leq C(K)\|(W_2,\psi,\mathrm{D}_{x'} \psi,\varepsilon \p_2^4 \psi, \varepsilon \p_3^4 \psi)\|_{L^2(\varSigma_t)}^2.
\nonumber
\end{align}
We plug the last inequality into \eqref{uni3} and take $\boldsymbol{\epsilon}>0$ sufficiently small to get
\begin{align}
&\|\mathrm{D}_{x'}W(t)\|_{L^2(\varOmega)}^2
+ \|\mathrm{D}_{x'}^{2}\psi(t)\|_{L^2(\varSigma)}^2
+ \|(\p_t \psi,\sqrt{\varepsilon}\mathrm{D}_{x'}W_2,\sqrt{\varepsilon}\mathrm{D}_{x'}^{4}\psi)\|_{L^2(\varSigma_t)}^2
\nonumber\\[1mm]
&  \leq C(K)\bigg(\|(\bm{f},W)\|_{1,*,t}^2
+  \sum_{|\alpha|\leq 2}
\|(W_2,\mathrm{D}_{x'}^{\alpha}\psi,  \mathrm{D}_{x'}\p_t\psi,\sqrt{\varepsilon}\mathrm{D}_{x'}^{\alpha}\mathrm{D}_{x'}\psi)\|_{L^2(\varSigma_t)}^2\bigg).
\label{uni3b}
\end{align}
\vspace*{2mm}

\noindent{\bf $\bullet$ $L^2$ estimate of $\p_tW$.}\quad
For $k=0$, in view of \eqref{decom} and \eqref{Reg.c}, we have
\begin{align*}
(\varepsilon\bm{J}-\bm{A}_1)\p_tW\cdot \p_tW
=\;&\varepsilon   (\p_t W_2)^2
+2 \p_1\mathring{q} \p_t\psi\p_tW_2
+2 \p_t\p_1\mathring{q} \psi\p_tW_2 \\
&-2\mathfrak{s}
\p_t\mathrm{D}_{x'}\cdot
\bigg(\dfrac{\mathrm{D}_{x'}\psi}{|\mathring{N}|}-
\dfrac{\mathrm{D}_{x'}\mathring{\varphi}\cdot\mathrm{D}_{x'}\psi}{|\mathring{N}|^3}\mathrm{D}_{x'}\mathring{\varphi}
\bigg) \p_t W_2
\ \ \textrm{on } \varSigma_T.
\end{align*}
Then it follows from \eqref{Reg.b} that
\begin{align}
\int_{\varSigma_t} (\varepsilon\bm{J}-\bm{A}_1)\p_tW\cdot \p_tW
=\varepsilon \int_{\varSigma_t} (\p_t W_2)^2
+{\mathcal{J}_1}+{\mathcal{J}_2}
+ \int_{\varSigma_t} \mathcal{T}_0+ \mathcal{J}_4^{(0)},
\label{uni4}
\end{align}
where the terms $\mathcal{J}_1$, $\mathcal{J}_2$, $\mathcal{T}_0$, and $\mathcal{J}_4^{(0)}$ are defined in \eqref{est3a}, \eqref{iden3}, and \eqref{J5.cal}.
Using the boundary condition \eqref{Reg.b} again, we get
\begin{align}
\mathcal{J}_1=\mathcal{J}_{1a}+\mathcal{J}_{1b}
\underbrace{-2\varepsilon\sum_{i=2,3} \int_{\varSigma_t} \p_1\mathring{q}\p_tW_2\p_i^4\psi}_{\mathcal{J}_{1c}},
\label{uni4a}
\end{align}
where the terms $\mathcal{J}_{1a}$ and $\mathcal{J}_{1b}$ are defined in \eqref{J1.est1}.
Clearly, we have
\begin{align}
|\mathcal{J}_{1c}|\leq \frac{\varepsilon}{2}\int_{\varSigma_t} (\p_t W_2)^2
+\varepsilon C(K) \|(\p_2^4\psi,\p_3^4\psi)\|_{L^2(\varSigma_t)}^2.
\label{uni4b}
\end{align}
Pugging \eqref{uni4}--\eqref{uni4a} into \eqref{uni2} with $k=0$, and utilizing
\eqref{J1a.est1}--\eqref{J1b.est1}, \eqref{uni4b}, \eqref{iden4}--\eqref{est:bas1}, \eqref{uni2c},  \eqref{est3c}, and \eqref{est:0f.R} imply
\begin{align}
&\|\p_tW(t)\|_{L^2(\varOmega)}^2 +\|(W_2,\p_t\mathrm{D}_{x'}\psi)(t)\|_{L^2(\varSigma)}^2
+\varepsilon \|(\p_tW_2,\mathrm{D}_{x'}^{3}\p_t\psi)\|_{L^2(\varSigma_t)}^2
\nonumber\\[1mm]
& \leq
C(K) \bigg(\boldsymbol{\epsilon} \|\p_1W_2(t)\|_{L^2(\varOmega)}^2
+C(\boldsymbol{\epsilon})\|\p_1W_2\|_{L^2(\varOmega_t)}^2
\nonumber\\[1mm]
&\qquad\qquad\ \
+C(\boldsymbol{\epsilon})\|(\bm{f},W)\|_{1,*,t}^2
+\sum_{|\alpha|\leq 2}
\|(W_2,\mathrm{D}_{x'}^{\alpha}\psi, \mathrm{D}_{x'}\p_t\psi )\|_{L^2(\varSigma_t)}^2
\nonumber\\
&\qquad\qquad\ \
+\|\p_t \psi\|_{L^2(\varSigma_t)}^2
+ {\varepsilon}\sum_{|\alpha|\leq 2} \|(\mathrm{D}_{x'}^{\alpha}\mathrm{D}_{x'}\psi,
 \mathrm{D}_{x'}^{\alpha}\p_t\psi,\p_2^4\psi,\p_3^4\psi)\|_{L^2(\varSigma_t)}^2
\bigg).
\label{uni5}
\end{align}
\vspace*{2mm}

\noindent{\bf $\bullet$ $H_*^1$ estimate of $W$.}\quad
To estimate the first term on the right-hand side, we use \eqref{Reg.a} and \eqref{decom} to get
\begin{align}
\nonumber 
\begin{pmatrix}
\p_1 W_2\\\p_1 W_1-\varepsilon \p_1W_2 \\0
\end{pmatrix}
=\bm{f}-{\bm{A}}_4 W-\sum_{i=0,2,3}{\bm{A}}_i\p_i W-{\bm{A}}_1^{(0)}\p_1W,
\end{align}
from which we obtain
\begin{align}
\|\p_1 W_{\rm nc}(t)\|_{L^2(\varOmega)}^2
&\leq C(K) \sum_{\langle\beta\rangle\leq 1}\|\mathrm{D}_*^{\beta}W(t)\|_{L^2(\varOmega)}^2+C(K)\|\bm{f}\|_{1,*,t}^2  .
\label{uni5a}
\end{align}
Applying the operator $\sigma\p_1$ to \eqref{Reg.a} implies
\begin{align}
\|\sigma \p_1W(t)\|_{L^2(\varOmega)}^2 \leq C(K)\|(\bm{f},W)\|_{1,*,t}^2.
\nonumber
\end{align}
Combining the last estimate with \eqref{uni1}, \eqref{uni3b}, \eqref{uni5}--\eqref{uni5a},
and
\begin{align*}
 \| \p_k\mathrm{D}_{x'}^{2}\psi\|_{L^2(\varSigma_t)}^2
\leq \boldsymbol{\epsilon}  \|  \p_k\mathrm{D}_{x'}^{3}\psi\|_{L^2(\varSigma_t)}^2
+C(\boldsymbol{\epsilon} )\|  \p_k\mathrm{D}_{x'}\psi\|_{L^2(\varSigma_t)}^2
\end{align*}
for $k=0,2,3$,
we take $\boldsymbol{\epsilon}>0$ sufficiently small to obtain
\begin{align}
\nonumber&
\sum_{\langle\beta\rangle\leq 1}\|\mathrm{D}_*^{\beta}W(t)\|_{L^2(\varOmega)}^2
+\|\p_1 W_{2}(t)\|_{L^2(\varOmega)}^2
+  \sum_{|\alpha|\leq 2} \|(W_2, \mathrm{D}_{x'}^{\alpha}\psi, \mathrm{D}_{x'}\p_t\psi)(t)\|_{L^2(\varSigma)}^2
\\
&\qquad
+\|\p_t\psi\|_{L^2(\varSigma_t)}^2
+\varepsilon
 \sum_{2\leq |\alpha|\leq 3} \|(W_2,\mathrm{D}_{x'}W_2,\p_t W_2, \mathrm{D}_{x'}^{\alpha} \mathrm{D}_{x'} \psi,\mathrm{D}_{x'}^{\alpha}\p_t  \psi)\|_{L^2(\varSigma_t)}^2
\nonumber \\
&\quad  \leq  C(K) \bigg(\|(\bm{f},W)\|_{1,*,t}^2
+\|\p_1 W_2\|_{L^2(\varOmega_t)}^2
+
\sum_{ |\alpha|\leq 2} \|(W_2,\mathrm{D}_{x'}^{\alpha}\psi, \mathrm{D}_{x'}\p_t\psi)\|_{L^2(\varSigma_t)}^2 \bigg).
\nonumber
\end{align}
Then it follows from Gr\"{o}nwall's inequality that
\begin{align}
\nonumber&
\sum_{\langle\beta\rangle\leq 1}\|\mathrm{D}_*^{\beta}W(t)\|_{L^2(\varOmega)}^2
+  \sum_{|\alpha|\leq 2} \|(W_2, \mathrm{D}_{x'}^{\alpha}\psi, \mathrm{D}_{x'}\p_t\psi)(t)\|_{L^2(\varSigma)}^2
+\|\p_t\psi\|_{L^2(\varSigma_t)}^2
\\
&\ \ +\varepsilon
\sum_{2\leq |\alpha|\leq 3} \|(W_2,\mathrm{D}_{x'}W_2,\p_t W_2,\mathrm{D}_{x'} \mathrm{D}_{x'}^{\alpha}  \psi,\mathrm{D}_{x'}^{\alpha}\p_t  \psi)\|_{L^2(\varSigma_t)}^2
  \leq  C(K)  \| \bm{f} \|_{1,*,t}^2
.
\label{uni6}
\end{align}
Combining \eqref{uni6}  with \eqref{Reg.d} and \eqref{uni5a}, we derive the {\it uniform-in-$\varepsilon$} estimate
\begin{multline}
\|W\|_{1,*,T}+ \|\partial_1W_{\rm nc}\|_{L^2(\varOmega_T)}+\|W_{\rm nc}\|_{L^2(\varSigma_T)}\\
+\|(\psi, \mathrm{D}_{x'}\psi)\|_{H^1(\varSigma_T)}
+\sqrt{\varepsilon} \| \mathrm{D}_{x'}^{4}  \psi\|_{L^2(\varSigma_T)}\leq  C(K)  \|\bm{f}\|_{1,*,T},
\label{unest}
\end{multline}
which allows us to construct the unique solution of the linearized problem \eqref{ELP3} by passing to the limit $\varepsilon\to 0$.
Indeed, due to the last estimate, we can extract a subsequence weakly convergent to $(W,\psi )\in H^1_*(\varOmega_T)\times H^1((-\infty,T],H^2(\mathbb{R}^2))$ with $\partial_1W_{\rm nc}\in L^2(\varOmega_T)$ and $W_{\rm nc}|_{x_1=0}\in L^2(\varSigma_T)$.
Since $\partial_1W_{2}$ and $\sqrt{\varepsilon} (\p_2^4+\p_3^4)\psi$ are uniformly bounded in $L^2(\varOmega_T)$ and $L^2(\varSigma_T)$ respectively ({\it cf.}\;\eqref{unest}),
the passage to the limit $\varepsilon\to 0$ in \eqref{Reg.a}--\eqref{Reg.c} verifies that $(W,\psi )$ is a solution of the problem \eqref{ELP3}.
Moreover, the uniqueness follows from the {\it a priori} estimate \eqref{est4}.

\subsection{High-Order Energy Estimates}\label{sec:linear4}

Let us derive the high-order energy estimates for solutions of the problem \eqref{ELP3}.
Let $m\in\mathbb{N}_+$ and $\alpha=(\alpha_0,\alpha_1,\alpha_2,\alpha_3,\alpha_4)\in\mathbb{N}^5$ with $\langle \alpha \rangle:=\sum_{i=0}^{4}\alpha_i+\alpha_4\leq m$.
Apply the operator $\mathrm{D}_*^{\alpha}:=\p_t^{\alpha_0}(\sigma \p_1)^{\alpha_1}\p_2^{\alpha_2}\p_3^{\alpha_3}\p_1^{\alpha_4}$ to \eqref{ELP3a} and take the scalar product of the resulting equations with $\mathrm{D}_*^{\alpha} W$ to obtain
\begin{align}\label{est:alpha}
\mathcal{R}_{\alpha}(t)+\mathcal{Q}_{\alpha}(t)
= \int_{\varOmega}{\bm{A}}_0\mathrm{D}_*^{\alpha}W\cdot \mathrm{D}_*^{\alpha}W(t,x)\d x
\gtrsim \|\mathrm{D}_*^{\alpha}W(t)\|_{L^{2}(\varOmega)}^2,
\end{align}
where
\begin{align}
\label{Q.alpha}
\mathcal{Q}_{\alpha}(t):=\;&2\int_{\varSigma_t}\mathrm{D}_*^{\alpha}W_1\mathrm{D}_*^{\alpha}W_2 ,\\
\label{R.alpha}
\mathcal{R}_{\alpha}(t):=\;&\int_{\varOmega_t}\mathrm{D}_*^{\alpha}W\cdot \bigg(
2R_{\alpha}+\sum_{i=0}^{3}\p_i {\bm{A}}_i\mathrm{D}_*^{\alpha}W \bigg) ,
\end{align}
with $R_{\alpha}:=\mathrm{D}_*^{\alpha}\bm{f}-\mathrm{D}_*^{\alpha}({\bm{A}}_4 W)
-\sum_{i=0}^{3}[\mathrm{D}_*^{\alpha},{\bm{A}}_i\p_i]W.$
Since the estimate of $\mathcal{R}_{\alpha}(t)$ does not involve any boundary condition,
from \cite[Lemma 3.5]{TW21MR4201624}, we obtain
\begin{align}
\label{est:R.alpha}
\sum_{\langle\alpha\rangle \leq m}
 \mathcal{R}_{\alpha}(t)\leq C(K)
  \mathcal{M}_1(t),
\end{align}
where
\begin{align}
\mathcal{M}_1(t):= {\|}(\bm{f},W){\|}_{m,*,t}^2
+\mathring{\rm C}_{m+4}  \|(\bm{f},W)\|_{W_*^{2,\infty}(\varOmega_t)}^2,
\label{M1.cal}
\end{align}
with
$\mathring{\rm C}_{m}$ being defined by \eqref{C.rm:def} and
\begin{align*}
\|u\|_{W_*^{2,\infty}(\varOmega_t) }:=\sum_{\langle \alpha\rangle\leq 1}\| \mathrm{D}_*^{\alpha} u\|_{W^{1,\infty}(\varOmega_t) }.
\end{align*}

\vspace*{4mm}

\noindent{\bf $\bullet$ Case $\alpha_1>0$.}\quad
For $\alpha_1>0$, we have $\mathcal{Q}_{\alpha}(t)=0$.
Hence it follows directly from \eqref{est:alpha} and \eqref{est:R.alpha} that
\begin{align}
\sum_{\langle\alpha\rangle\leq m,\, \alpha_1>0}\|\mathrm{D}_*^{\alpha}W(t)\|_{L^2(\varOmega)}^2
\leq C(K) \mathcal{M}_1(t),
\label{est:alpha1}
\end{align}
where $\mathcal{M}_1(t)$ is given in \eqref{M1.cal}.

\vspace*{4mm}

\noindent{\bf $\bullet$ Case $\alpha_1=0$ and $\alpha_4>0$.}\quad
Next let us consider the case with $\alpha_1=0$ and $\alpha_4>0$.
From \eqref{iden5}, we have
\begin{align}
\mathcal{Q}_{\alpha}(t)\lesssim
\sum_{i=0,2,3}\|\mathrm{D}_*^{\alpha-{\bm{e}}} (\bm{f},  {\bm{A}}_4 W,{\bm{A}}_i\p_i W,{\bm{A}}_1^{(0)}\p_1W)\|_{L^2(\varSigma_t)}^2,
\nonumber 
\end{align}
for the multi-index ${\bm{e}}:=(0,0,0,0,1)$.
Then we can employ the proof of \cite[Lemma 3.7]{TW21MR4201624} to deduce
\begin{align}
&\sum_{\substack{\langle\alpha\rangle \leq m,\, \alpha_1=0,\,\alpha_4>0}}
\|\mathrm{D}_*^{\alpha} W(t)\|_{L^2(\varOmega)}^2
\nonumber\\
&\qquad\qquad \lesssim
\boldsymbol{\epsilon} \sum_{\langle \beta\rangle\leq m} \| \mathrm{D}_*^{\beta} W(t)\|_{L^2(\varOmega)}^2
+ C(\boldsymbol{\epsilon},K) \mathcal{M}_1(t)
\quad \textrm{for  }  \boldsymbol{\epsilon}\in (0,1).
\label{est:alpha2}
\end{align}


\noindent{\bf $\bullet$ Case $\alpha_1=\alpha_4=0$.}\quad
Now we let $\alpha_1=\alpha_4=0$ and $\langle \alpha\rangle\leq m$.
Then $\mathrm{D}_*^{\alpha}=\p_t^{\alpha_0}\p_2^{\alpha_2}\p_2^{\alpha_3}$ and $\alpha_{0}+\alpha_2+\alpha_3\leq m$.
Using \eqref{ELP3b}--\eqref{ELP3c}, we calculate
\begin{align} \label{est:Q3a}
\mathcal{Q}_{\alpha}(t)=\sum_{i=1}^{4}\mathcal{Q}_{\alpha}^{(i)}(t),
\end{align}
where
\begin{align*}
&\mathcal{Q}_{\alpha}^{(1)}(t):=-2\int_{\varSigma_t}
\Big\{\p_1 \mathring{q} \mathrm{D}_*^{\alpha}\psi+[\mathrm{D}_*^{\alpha},\p_1 \mathring{q} ]\psi\Big\}
 (\p_t+\mathring{v}_2\p_2+\mathring{v}_3\p_3)\mathrm{D}_*^{\alpha} \psi ,\\
&\mathcal{Q}_{\alpha}^{(2)}(t):=-2\mathfrak{s}\int_{\varSigma_t} \mathrm{D}_*^{\alpha}\bigg(\dfrac{\mathrm{D}_{x'}\psi}{|\mathring{N}|}- \dfrac{\mathrm{D}_{x'}\mathring{\varphi}\cdot\mathrm{D}_{x'}\psi}{|\mathring{N}|^3}\mathrm{D}_{x'}\mathring{\varphi}\bigg) \cdot (\p_t+\mathring{v}_2\p_2+\mathring{v}_3\p_3)\mathrm{D}_*^{\alpha} \mathrm{D}_{x'}\psi  ,\\
&\mathcal{Q}_{\alpha}^{(3)}(t):=-2\int_{\varSigma_t} \mathrm{D}_*^{\alpha}(\p_1 \mathring{q} \psi)
\Big\{[\mathrm{D}_*^{\alpha}, \mathring{v}_2\p_2+\mathring{v}_3\p_3]\psi -\mathrm{D}_*^{\alpha}\big(\p_1 \mathring{v}\cdot \mathring{N} \psi\big) \Big\} ,\\
&\mathcal{Q}_{\alpha}^{(4)}(t):=-2\mathfrak{s}\int_{\varSigma_t}
\mathrm{D}_*^{\alpha}\bigg(\dfrac{\mathrm{D}_{x'}\psi}{|\mathring{N}|}- \dfrac{\mathrm{D}_{x'}\mathring{\varphi}\cdot\mathrm{D}_{x'}\psi}{|\mathring{N}|^3}\mathrm{D}_{x'}\mathring{\varphi}\bigg)
\\
&\qquad \qquad  \qquad \qquad   \cdot \Big\{[\mathrm{D}_*^{\alpha}\mathrm{D}_{x'}, \mathring{v}_2\p_2+\mathring{v}_3\p_3]\psi-\mathrm{D}_*^{\alpha}\mathrm{D}_{x'}\big(\p_1 \mathring{v}\cdot \mathring{N} \psi\big) \Big\}.
\end{align*}
Apply integration by parts, Cauchy's and Moser-type calculus inequalities to infer
\begin{align*}
\mathcal{Q}_{\alpha}^{(1)}(t)\leq \;& -\int_{\varSigma} \p_1 \mathring{q} (\mathrm{D}_*^{\alpha} \psi)^2\mathrm{d}x'
-2\int_{\varSigma}[\mathrm{D}_*^{\alpha},\p_1 \mathring{q}]\psi \mathrm{D}_*^{\alpha} \psi
+C(K) \|\mathrm{D}_*^{\alpha} \psi\|_{L^2(\varSigma_t)}^2 \\[1mm]
\;& +\| \p_t [\mathrm{D}_*^{\alpha},\p_1 \mathring{q}] \psi \|_{L^2(\varSigma_t)}^2
+\sum_{i=2,3}\| \p_i (\mathring{v}_i [\mathrm{D}_*^{\alpha},\p_1 \mathring{q}] \psi )\|_{L^2(\varSigma_t)}^2\\
\leq \;&
C(K) \left(\|\mathrm{D}_*^{\alpha} \psi(t)\|_{L^2(\varSigma)}^2
+\|\mathrm{D}_*^{\alpha} \psi\|_{L^2(\varSigma_t)}^2
+ \| [\mathrm{D}_*^{\alpha},\p_1 \mathring{q}] \psi \|_{H^1(\varSigma_t)}^2  \right)\\[1.5mm]
\leq \;&
C(K) \left(\|\mathrm{D}_*^{\alpha} \psi(t)\|_{L^2(\varSigma)}^2
+\|\psi\|_{H^m(\varSigma_t)}^2
+\mathring{\mathrm{C}}_{m+4} \|\psi\|_{L^{\infty}(\varSigma_t)}^2\right).
\end{align*}
It follows from the Moser-type calculus inequalities that
\begin{align*}
\mathcal{Q}_{\alpha}^{(3)}(t)
\lesssim\;&
\|(\p_1 \mathring{q} \psi,\p_1 \mathring{v}\cdot \mathring{N} \psi)\|_{H^{m}(\varSigma_t)}^2
+\|[\mathrm{D}_*^{\alpha}, \mathring{v}_2\p_2+\mathring{v}_3\p_3]\psi\|_{L^{2}(\varSigma_t)}^2
\\
\leq\;&
C(K) \left(\|\psi\|_{H^m(\varSigma_t)}^2
+ \mathring{\mathrm{C}}_{m+4} \|\psi\|_{L^{\infty}(\varSigma_t)}^2 \right),\\
\mathcal{Q}_{\alpha}^{(4)}(t)
\lesssim\;&
\|  (\mathring{\rm c}_1\mathrm{D}_{x'}\psi,\mathrm{D}_{x'}\mathring{\rm c}_1\psi)\|_{H^{m}(\varSigma_t)}^2
+\|[\mathrm{D}_*^{\alpha}\mathrm{D}_{x'}, \mathring{v}_2\p_2+\mathring{v}_3\p_3]\psi\|_{L^{2}(\varSigma_t)}^2
\\
\leq\;&
C(K) \left(\|(\psi,\mathrm{D}_{x'}\psi)\|_{H^m(\varSigma_t)}^2
+ \mathring{\mathrm{C}}_{m+4} \|(\psi,\mathrm{D}_{x'}\psi)\|_{L^{\infty}(\varSigma_t)}^2 \right).
\end{align*}
For the term $\mathcal{Q}_{\alpha}^{(2)}(t)$, we calculate
\begin{align*}
&-2 \mathrm{D}_*^{\alpha}\bigg(\dfrac{\mathrm{D}_{x'}\psi}{|\mathring{N}|}- \dfrac{\mathrm{D}_{x'}\mathring{\varphi}\cdot\mathrm{D}_{x'}\psi}{|\mathring{N}|^3}\mathrm{D}_{x'}\mathring{\varphi}\bigg) \cdot (\p_t+\mathring{v}_2\p_2+\mathring{v}_3\p_3)\mathrm{D}_*^{\alpha} \mathrm{D}_{x'}\psi\\
&  =\p_t\bigg\{
-\frac{|\mathrm{D}_*^{\alpha}\mathrm{D}_{x'}\psi|^2}{|\mathring{N}|}+\frac{|\mathrm{D}_{x'}\mathring{\varphi}\cdot\mathrm{D}_*^{\alpha}\mathrm{D}_{x'}\psi|^2}{|\mathring{N}|^3}
+[\mathrm{D}_*^{\alpha},\mathring{\rm c}_1]\mathrm{D}_{x'}\psi\cdot\mathrm{D}_*^{\alpha}\mathrm{D}_{x'}\psi\bigg\}
\\
& \quad
+
\sum_{i=2,3}\p_i\bigg\{
\mathring{v}_i\bigg(-\frac{|\mathrm{D}_*^{\alpha}\mathrm{D}_{x'}\psi|^2}{|\mathring{N}|}+\frac{|\mathrm{D}_{x'}\mathring{\varphi}\cdot\mathrm{D}_*^{\alpha}\mathrm{D}_{x'}\psi|^2}{|\mathring{N}|^3}
+[\mathrm{D}_*^{\alpha},\mathring{\rm c}_1]\mathrm{D}_{x'}\psi\cdot\mathrm{D}_*^{\alpha}\mathrm{D}_{x'}\psi\bigg)\bigg\}\\
& \quad
+\mathring{\rm c}_2\mathrm{D}_*^{\alpha}\mathrm{D}_{x'}\psi\cdot \mathrm{D}_*^{\alpha}\mathrm{D}_{x'}\psi
-\bigg\{\p_t[\mathrm{D}_*^{\alpha},\mathring{\rm c}_1]\mathrm{D}_{x'}\psi+\sum_{i=2,3} \p_i(\mathring{v}_i[\mathrm{D}_*^{\alpha},\mathring{\rm c}_1]\mathrm{D}_{x'}\psi ) \bigg\}\cdot
\mathrm{D}_*^{\alpha}\mathrm{D}_{x'}\psi.
\end{align*}
Then it follows from Cauchy's inequality, integration by parts, and Moser-type calculus inequalities that
\begin{align}
\nonumber
\mathcal{Q}_{\alpha}^{(2)}(t)
\leq\;& -\mathfrak{s}\int_{\varSigma} \frac{|\mathrm{D}_*^{\alpha}\mathrm{D}_{x'}\psi|^2}{|\mathring{N}|^3}\mathrm{d}x'
+\int_{\varSigma} |[\mathrm{D}_*^{\alpha},\mathring{\rm c}_1]\mathrm{D}_{x'}\psi||\mathrm{D}_*^{\alpha}\mathrm{D}_{x'}\psi|\mathrm{d}x'\\[1mm]
\;&
+C(K)\sum_{i=2,3}\| (\p_t [\mathrm{D}_*^{\alpha}, \mathring{\rm c}_1] \mathrm{D}_{x'}\psi,\p_i (\mathring{v}_i [\mathrm{D}_*^{\alpha},\mathring{\rm c}_1] \mathrm{D}_{x'}\psi ),\mathrm{D}_*^{\alpha} \mathrm{D}_{x'}\psi)\|_{L^2(\varSigma_t)}^2
\nonumber\\
\leq\;& -\frac{\mathfrak{s}}{2}\int_{\varSigma} \frac{|\mathrm{D}_*^{\alpha}\mathrm{D}_{x'}\psi|^2}{|\mathring{N}|^3}\d x'
+C(K)\left(
\|\mathrm{D}_*^{\alpha}\mathrm{D}_{x'}\psi\|_{L^2(\varSigma_t)}^2
+\|[\mathrm{D}_*^{\alpha},\mathring{\rm c}_1] \mathrm{D}_{x'}\psi \|_{H^1(\varSigma_t)}^2
\right)
\nonumber\\
\leq\;& -\frac{\mathfrak{s}}{2}\int_{\varSigma} \frac{|\mathrm{D}_*^{\alpha}\mathrm{D}_{x'}\psi|^2}{|\mathring{N}|^3}\d x'
+C(K) \left(\|\mathrm{D}_{x'}\psi\|_{H^m(\varSigma_t)}^2
+\mathring{\mathrm{C}}_{m+4} \|\mathrm{D}_{x'}\psi\|_{L^{\infty}(\varSigma_t)}^2\right).
\nonumber
\end{align}
Substituting the above estimates of $\mathcal{Q}_{\alpha}^{(i)}(t)$, for $i=1,2,3,4$, into \eqref{est:Q3a}, we get
\begin{align}
\mathcal{Q}_{\alpha}(t)
+\frac{\mathfrak{s}}{2}\int_{\varSigma} \frac{|\mathrm{D}_*^{\alpha}\mathrm{D}_{x'}\psi|^2}{|\mathring{N}|^3}\d x'
\leq
C(K) \Big(\|\mathrm{D}_*^{\alpha}\psi(t)\|_{L^2(\varSigma)}^2
+\mathcal{M}_2(t)\Big)
\label{est:Q3b}
\end{align}
for all $\alpha=(\alpha_0,0,\alpha_2,\alpha_3,0)\in\mathbb{N}^5$ with $|\alpha|\leq m$, where
\begin{align}
\mathcal{M}_2(t):=\|( \psi,\mathrm{D}_{x'}\psi)\|_{H^m(\varSigma_t)}^2 +\mathring{\mathrm{C}}_{m+4} \|( \psi,\mathrm{D}_{x'}\psi)\|_{L^{\infty}(\varSigma_t)}^2.
\label{M2.cal}
\end{align}

Let us estimate the first term on the right-hand side of \eqref{est:Q3b}.
If $|\alpha|\leq m-1$ or $\alpha_2+\alpha_3\geq 1$, then 
\begin{align}
\|\mathrm{D}_*^{\alpha}\psi(t)\|_{L^2(\varSigma)}^2
\lesssim \int_{\varSigma_t} |\mathrm{D}_*^{\alpha}\psi||\p_t\mathrm{D}_*^{\alpha}\psi|
\lesssim \|( \psi,\mathrm{D}_{x'}\psi)\|_{H^m(\varSigma_t)}^2.
\label{est:Q3c}
\end{align}
If $\alpha_2=\alpha_3=0$ and $\alpha_0=m$, then it follows from \eqref{ELP3b} that
$$\mathrm{D}_*^{\alpha}\psi=\p_t^{m-1}(W_2-\mathring{v}_2\p_2\psi-\mathring{v}_3\p_3\psi+\p_1\mathring{v}\cdot\mathring{N}\psi),$$
which yields
\begin{align}
\nonumber
 \|\mathrm{D}_*^{\alpha}\psi(t)\|_{L^2(\varSigma)}^2 & \nonumber
\lesssim \|\p_t^{m-1}W_2(t)\|_{L^2(\varSigma)}^2
+\|\mathring{v}_2\p_2\psi+\mathring{v}_3\p_3\psi-\p_1\mathring{v}\cdot\mathring{N}\psi\|_{H^m(\varSigma_t)}^2\\
& \lesssim  \|\p_t^{m-1}W_2(t)\|_{L^2(\varSigma)}^2
+\mathcal{M}_2(t),
\label{est:Q3d}
\end{align}
with $\mathcal{M}_2(t)$ being given in \eqref{M2.cal}.
For the first term on the right-hand side, we use integration by parts to obtain
\begin{align}
\|\p_t^{m-1}W_2(t)\|_{L^2(\varSigma)}^2
\lesssim \;& \boldsymbol{\epsilon}\|\p_t^{m-1}\p_1W_2(t)\|_{L^2(\varOmega)}^2 + {\boldsymbol{\epsilon}^{-1}}\| \p_t^{m-1}W_2(t)\|_{L^2(\varOmega)}^2 \nonumber\\
\lesssim\;& \boldsymbol{\epsilon}\|\p_t^{m-1}\p_1W_2(t)\|_{L^2(\varOmega)}^2
+ {\boldsymbol{\epsilon}^{-1}}\| W_2\|_{m,*,t}^2 \quad
\label{est:Q3d1}
\end{align}
$\textrm{for } \boldsymbol{\epsilon}\in (0,1).$
From \eqref{iden5} and \eqref{decom}, we infer
\begin{align}
\nonumber
 \|\p_t^{m-1}\p_1W_2(t)\|_{L^2(\varOmega)}^2
 \lesssim\;&
\sum_{\langle \beta\rangle\leq m} \| \mathrm{D}_*^{\beta} W(t)\|_{L^2(\varOmega)}^2
+  \|(\bm{f},{\bm{A}}_4 W)\|_{m,*,t}^2\\
&+\sum_{i=0,2,3}\| ([\p_t^{m-1},{\bm{A}}_i\p_i] W ,[\p_t^{m-1},{\bm{A}}_1^{(0)}\p_1] W )(t)\|_{L^2(\varOmega)}^2
\nonumber\\
 \lesssim\;&
 \sum_{\langle \beta\rangle\leq m} \| \mathrm{D}_*^{\beta} W(t)\|_{L^2(\varOmega)}^2
 +\mathcal{M}_1(t).
 \label{est:Q3d2}
\end{align}
Plug \eqref{est:Q3d1}--\eqref{est:Q3d2} into \eqref{est:Q3d} and combine the resulting estimate with \eqref{est:Q3b}--\eqref{est:Q3c} to deduce
\begin{align}
&\mathcal{Q}_{\alpha}(t)+\|\mathrm{D}_*^{\alpha}\psi(t)\|_{L^2(\varSigma)}^2
+\frac{\mathfrak{s}}{2}\int_{\varSigma} \frac{|\mathrm{D}_*^{\alpha}\mathrm{D}_{x'}\psi|^2}{|\mathring{N}|^3}\d x'\nonumber\\
&\quad \leq  C(K)\bigg(\boldsymbol{\epsilon} \sum_{\langle \beta\rangle\leq m} \| \mathrm{D}_*^{\beta} W(t)\|_{L^2(\varOmega)}^2
+C(\boldsymbol{\epsilon}) \mathcal{M}_1(t)+C(\boldsymbol{\epsilon}) \mathcal{M}_2(t)\bigg)
\label{est:Q3}
\end{align}
for all $\alpha\in \mathbb{N}^5$ with $ \alpha_{1}=\alpha_4=0$ and $\langle \alpha\rangle\leq m$,
where
$\mathcal{M}_1(t)$ and $\mathcal{M}_2(t)$ are defined by \eqref{M1.cal} and \eqref{M2.cal}, respectively.
Substituting \eqref{est:R.alpha} and \eqref{est:Q3} into \eqref{est:alpha} leads to
\begin{align}
&\sum_{\substack{\langle\alpha\rangle \leq m,\, \alpha_1=\alpha_4=0}}
\left(\|\mathrm{D}_*^{\alpha} W(t)\|_{L^2(\varOmega)}^2
+\|(\mathrm{D}_*^{\alpha}\psi,\mathrm{D}_*^{\alpha}\mathrm{D}_{x'}\psi)(t)\|_{L^2(\varSigma)}^2
\right)
\nonumber\\
&\qquad\qquad \leq
C(K)\bigg(\boldsymbol{\epsilon} \sum_{\langle \beta\rangle\leq m} \| \mathrm{D}_*^{\beta} W(t)\|_{L^2(\varOmega)}^2
+C(\boldsymbol{\epsilon}) \mathcal{M}_1(t)+C(\boldsymbol{\epsilon}) \mathcal{M}_2(t)\bigg)
\label{est:alpha3}
\end{align}
for $\boldsymbol{\epsilon}\in (0,1).$

\subsection{Proof of Theorem \ref{thm:linear}} \label{sec:linear5}

Combining \eqref{est:alpha1}--\eqref{est:alpha2} and \eqref{est:alpha3}, we take $\boldsymbol{\epsilon}>0$ sufficiently small to derive
\begin{align}
\mathcal{I}(t)\leq C(K) \mathcal{M}_1(t)+C(K) \mathcal{M}_2(t) ,
\label{est:alpha4}
\end{align}
where $\mathcal{M}_1(t)$ and $\mathcal{M}_2(t)$ are given in \eqref{M1.cal} and \eqref{M2.cal}, respectively, and
\begin{align}
\mathcal{I}(t):=\sum_{ \langle\alpha\rangle \leq m } \|\mathrm{D}_*^{\alpha} W(t)\|_{L^2(\varOmega)}^2
+\sum_{\substack{\langle\alpha\rangle \leq m,\, \alpha_1=\alpha_4=0}}
\|(\mathrm{D}_*^{\alpha}\psi,\mathrm{D}_*^{\alpha}\mathrm{D}_{x'}\psi)(t)\|_{L^2(\varSigma)}^2.
\nonumber 
\end{align}
It follows from definition that, for $s\in\mathbb{N}$,
\begin{align}   \label{est:psi}
\|{\mathring{\varphi}}\|_{H^s(\varSigma_t)}  \leq  {\|}\mathring{\varPsi}{\|}_{s,*,t}\lesssim \|{\mathring{\varphi}}\|_{H^s(\varSigma_t)} ,
\
\|\psi\|_{H^s(\varSigma_t)}\leq  {\|}\varPsi{\|}_{s,*,t} \lesssim \|\psi\|_{H^s(\varSigma_t)}.
\end{align}
Since
\begin{align} \nonumber 
\int_{0}^t\mathcal{I}({\tau})\d{\tau}={\|}W{\|}_{m,*,t}^2
+\|(\psi,\mathrm{D}_{x'}\psi)\|_{H^m(\varSigma_t)}^2,
\end{align}
we apply Gr\"{o}nwall's inequality to \eqref{est:alpha4} and infer
\begin{align}
\label{I.cal2}
\mathcal{I}(t)\leq C(K) \mathrm{e}^{C(K)T} \mathcal{N}(T)
\quad \textrm{for } t \in [0,T],
\end{align}
where
\begin{align} \nonumber
\mathcal{N}(t):={\|}\tilde{f}{\|}_{m,*,t}^2+\mathring{\rm C}_{{m+4}} \left(\|(\tilde{f},W)\|_{W_*^{2,\infty}(\varOmega_t)}^2 +\|(\psi,\mathrm{D}_{x'}\psi)\|_{L^{\infty}(\varSigma_t)}^2\right).
\end{align}
It follows from the embedding inequalities (see, {\it e.g.}, \cite[Lemma 3.3]{TW21MR4201624}) that
\begin{align}\label{M.cal1}
\mathcal{N}(T)\lesssim {\|}\tilde{f}{\|}_{m,*,T}^2+\mathring{\rm C}_{{m+4}} \left({\|}(\tilde{f},W){\|}_{6,*,T}^2 +\|\psi\|_{H^{3}(\varSigma_T)}^2\right).
\end{align}
Recalling the definition \eqref{C.rm:def} for $\mathring{\rm C}_{{m+4}}$, we integrate \eqref{I.cal2} over $[0,T]$ and take $T>0$ sufficiently small to obtain
\begin{multline}
{\|}W{\|}_{m,*,T}^2+\|(\psi,\mathrm{D}_{x'}\psi)\|_{H^m(\varSigma_T)}^2
\leq   C(K)T\mathrm{e}^{C(K)T}  \Big\{
{\|}\tilde{f}{\|}_{m,*,T}^2
\\
+ {\|}(\mathring{V},\mathring{\varPsi}){\|}_{{m+4},*,T}^2\Big({{\|}(\tilde{f},W){\|}_{6,*,T}^2} +\|\psi\|_{H^{3}(\varSigma_T)}^2\Big)
\Big\}
\ \ \textrm{for }  m\geq 6.   \label{est:t1}
\end{multline}
Using \eqref{est:t1} with $m=6$, \eqref{H:thm.linear}, \eqref{est:psi}, and the embedding $ H_*^{9}(\varOmega_T)\hookrightarrow W^{3,\infty}(\varOmega_T)$,
we can find a  suitably small constant $T_0>0$, depending on $K_0$, such that if $0<T\leq T_0$, then
\begin{align}\nonumber
{\|}W{\|}_{6,*,T}^2+\|(\psi,\mathrm{D}_{x'}\psi)\|_{H^6(\varSigma_T)}^2
\leq C(K_0) {\|}\tilde{f}{\|}_{6,*,T}^2.
\end{align}
Substitute the last estimate into \eqref{est:t1} to derive
\begin{align}\nonumber
& {\|}W{\|}_{m,*,T}^2+\|(\psi,\mathrm{D}_{x'}\psi)\|_{H^m(\varSigma_T)}^2 \\
&\quad  \leq C(K_0)
\label{est:t3} \left({\|}\tilde{f}{\|}_{m,*,T}^2
+  {\|}(\mathring{V},\mathring{\varPsi}){\|}_{{m+4},*,T}^2{\|}\tilde{f}{\|}_{6,*,T}^2\right)
\ \ \textrm{for }m\geq 6.
\end{align}

In Section \ref{sec:linear3}, we have proved that for $(f,g)\in H_*^1(\varOmega_T)\times H^2(\varSigma_T)$ vanishing in the past, {the} problem \eqref{ELP3} admits a unique solution $(W,\psi)\in H_*^1(\varOmega_T)\times H^1(\varSigma_T)$.
Applying the arguments in \cite[Chapter 7]{CP82MR0678605} and using the energy estimate \eqref{est:t3},
one can establish the existence and uniqueness of solutions $(W,\psi)$ of {the} problem \eqref{ELP3} in $H_*^m(\varOmega_T)\times H^{m}(\varSigma_T)$ for $m\geq 6$.

It remains to prove the tame estimates \eqref{tame} of {the} problem \eqref{ELP1}.
For this purpose, we use \eqref{V.natural} to derive
\begin{align}
\nonumber
{\|}\dot{V}{\|}_{m,*,T}^2 \leq\;& C(K_0)\Big(
{\|}W{\|}_{m,*,T}^2+\|W\|_{W_*^{1,\infty}(\varOmega_T) }^2\mathring{\rm C}_{m+2}
+\|g\|_{H^{m}(\varSigma_T)}^2
\Big),\\
\nonumber
{\|}\tilde{f}{\|}_{m,*,T}^2\leq\;&  C(K_0)\Big(
{\|} {f}{\|}_{m,*,T}^2+
\|V_{\natural}\|_{W_*^{2,\infty}(\varOmega_T) }^2\mathring{\rm C}_{m+2}
+\|g\|_{H^{m+1}(\varSigma_T)}^2
\Big),
\end{align}
which combined with \eqref{est:t3}, \eqref{est:psi}, and the embedding inequalities
yield the tame estimates \eqref{tame}. The proof of Theorem \ref{thm:linear} is complete.

\section{Proof of Theorem \ref{thm:main}} \label{sec:proof}

In this section, we prove the main theorem of this paper by using a Nash--Moser iteration to handle the loss of regularity from the coefficients ({\it i.e.}, the basic states) to the solutions in \eqref{tame}.
We refer the reader to {\sc Alinhac--G{{\'e}}rard} \cite[Chapter 3]{AG07MR2304160} and {\sc Secchi} \cite{S16MR3524197} for a general description of the method.

\subsection{Reducing the Nonlinear Problem} \label{sec:proof1}

In order to employ Theorem \ref{thm:linear}, we will reformulate the nonlinear problem \eqref{NP1} into a problem with zero initial data by absorbing the initial data into the interior equations via the approximate solutions. Before this, we introduce the compatibility conditions on the initial data that are necessary in the construction of the approximate solutions.

Let $m\in\mathbb{N}$ with $m\geq 3$.
Suppose that the initial data \eqref{NP1c}
satisfy
$\widetilde{U}_0:=U_0-\widebar{U} \in H^{m+3/2}(\varOmega)$,
$\varphi_0\in H^{m+2}(\mathbb{R}^{2})$,
and
$\|\varphi_0\|_{L^{\infty}(\mathbb{R}^{2})}\leq 1/4$.
Then
\begin{align}\label{CA1}
\p_1 \varPhi_0 \geq \frac{3}{4}\ \ \textrm{in } \varOmega,\quad
\textrm{for } \varPhi_0:= x_1+{\varPsi}_0,
\end{align}
where ${\varPsi}_0:=\chi(x_1) \varphi_0$ and $\chi\in C_0^{\infty}(\mathbb{R})$ satisfies \eqref{chi}.
We denote the perturbation $U -\widebar{U} $ by $\widetilde{U} $. Then we define
$(\widetilde{U}_{(j)},\,\varphi_{(j)})
:= (\p_t^{j}\widetilde{U},\,\p_t^{j}\varphi)|_{t=0}$ and
${\varPsi}_{(j)}:=\chi(x_1) \varphi_{(j)}$ for any integer $j$.
Setting $\xi:=\mathrm{D}_{x'}\varphi\in\mathbb{R}^2$ and
$\mathcal{W}:=(\widetilde{U},\nabla \widetilde{U},\mathrm{D}{\varPsi})^{\mathsf{T}}\in\mathbb{R}^{36}$,
we can rewrite the second condition in \eqref{NP1b} and the equations \eqref{NP1a} as
\begin{align}\label{eq.refor}
q=\mathfrak{s}\mathrm{D}_{x'}\cdot \mathfrak{f}(\xi),\quad
\p_t \widetilde{U}=\bm{G}\big(\mathcal{W}\big),
\end{align}
where
\begin{align} \label{f.frak}
	\mathfrak{f}(\xi):=\frac{\xi}{\sqrt{1+|\xi|^2}},
\end{align}
and $\bm{G}$ is a suitable $C^{\infty}$--function vanishing at the origin.
Applying the generalized Fa\`a di Bruno's formula (see \cite[Theorem 2.1]{M00MR1781515}) and the Leibniz's rule to \eqref{eq.refor} and the first condition in \eqref{NP1b}, respectively,
we compute
\begin{align}
\label{compat1}
q_{(j)}|_{\varSigma}=\;&
\sum_{\substack{\alpha_{k}\in\mathbb{N}^{2}\\ |\alpha_1|+\cdots+j |\alpha_{j}|=j}} \mathfrak{s}\mathrm{D}_{x'}\cdot
\left(\mathrm{D}^{\alpha_1+\cdots+\alpha_j}\mathfrak{f}\big(\xi_{(0)}\big)\prod_{k=1}^j\frac{j!}{\alpha_{k}!}
\left(\frac{\xi_{(k)}}{k!}\right)^{\alpha_{k}}\right),
\\
\label{trace.id2}
\widetilde{U}_{(j+1)}
=\;&
\sum_{\substack{\alpha_{k}\in\mathbb{N}^{35}\\ |\alpha_1|+\cdots+j |\alpha_{j}|=j}}
\mathrm{D}^{\alpha_1+\cdots+\alpha_j}\bm{G}\big(\mathcal{W}_{(0)}\big)\prod_{k=1}^j\frac{j!}{\alpha_{k}!}
\left(\frac{\mathcal{W}_{(k)}}{k!}\right)^{\alpha_{k}},\\
\label{trace.id1}
\varphi_{(j+1)} =\;& v_{1(j)}\big|_{\varSigma}-\sum_{k=0}^{j} \sum_{i=2,3}
\frac{j!}{(j-k)!k!} v_{i(k)}\big|_{\varSigma}\p_i\varphi_{(j-k)},
\end{align}
where $\xi_{(k)}:=\mathrm{D}_{x'}\varphi_{(k)}$ and
$\mathcal{W}_{(k)}:=(\widetilde{U}_{(k)},\nabla \widetilde{U}_{(k)},\mathrm{D} {\varPsi}_{(k)})^{\mathsf{T}}$.
The identities \eqref{trace.id2}--\eqref{trace.id1} can determine $\widetilde{U}_{(j)}$ and ${\varphi}_{(j)}$ for integers $j$ inductively. More precisely, we have the following lemma whose proof can be found from \cite[Lemma 4.2.1]{M01MR1842775}.

\begin{lemma}\label{lem:CA1}
 Let $m\in\mathbb{N}$ with $m\geq 3$, $\widetilde{U}_0:= U_0-\widebar{U} \in H^{m+3/2}(\varOmega)  $,
 and $ \varphi_0  \in   H^{m+2}(\mathbb{R}^{2})$.
 Then the equations \eqref{trace.id1} and \eqref{trace.id2} determine
 $\widetilde{U}_{(j)}\in H^{m+3/2-j}(\varOmega)$ and $\varphi_{(j)}\in H^{m+2-j}(\mathbb{R}^{2})$, for $j=1,\ldots,m$,   satisfying
 \begin{align}
  \sum_{j=0}^{m}\Big(\big\|\widetilde{U}_{(j)}\big\|_{H^{m+3/2-j}(\varOmega)} +\big\|{\varphi}_{(j)}\big\|_{H^{m+2-j}(\mathbb{R}^{2})}\Big)
 \nonumber 
 \leq C M_0,
 \end{align}
where constant $C>0$ depends only on $m$,
$\|\widetilde{U}_{0}\|_{W^{1,\infty}(\varOmega)}$,
and $\|{\varphi}_{0}\|_{W^{1,\infty}(\mathbb{R}^{2})}$, and
 \begin{align}
 \label{M0}
 M_0:=\big\|\widetilde{U}_0\big\|_{H^{m+3/2}(\varOmega)}
 +\|\varphi_0\|_{H^{m+2}(\mathbb{R}^{2})}.
 \end{align}
\end{lemma}

We are now ready to introduce the compatibility conditions on the initial data.

\begin{definition}
 \label{def:1}
Let $m\geq 3$ be an integer.
Assume that $\widetilde{U}_0:=U_0-\widebar{U}\in H^{m+3/2}(\varOmega) $ and $\varphi_0 \in H^{m+2}(\mathbb{R}^{2})$ satisfy \eqref{CA1}.
If the functions $\widetilde{U}_{(j)}$ and $\varphi_{(j)}$ determined by \eqref{trace.id2}--\eqref{trace.id1} satisfy \eqref{compat1} for $j=0,\ldots,m$,
then we say that the initial data $(U_0,\varphi_0)$ are compatible up to order $m$.
\end{definition}

Imposing the above compatibility conditions on the initial data, we can construct the approximate solution in the next lemma. We omit the detailed proof since it is similar to that of \cite[Lemma 4.2]{TW21MR4201624}.

\begin{lemma} \label{lem:app}
Let $m\geq 3$ be an integer.
Assume that the initial data $(U_0,\varphi_0)$ satisfy the constraint \eqref{inv2b},
$\widetilde{U}_0:=U_0-\widebar{U}\in H^{m+3/2}(\varOmega)$,
$\varphi_0 \in  H^{m+2}(\mathbb{R}^{2})$,
and the compatibility conditions up to order $m$.
Then we can find positive constants $T_1(M_0)$ and $C(M_0)$ ({\it cf.}~\eqref{M0}),
such that if $0<T\leq T_1(M_0)$, then there exist $U^{a}$ and $\varphi^a$ satisfying
 \begin{alignat}{3}
 &\big\|\widetilde{U}^{a}\big\|_{H^{m+1}(\varOmega_T)}
 +\|\varphi^a\|_{H^{m+5/2}(\varSigma_T)}  \leq C (M_0),
 \quad   \label{app1a}\\
 \label{app1b}
 &\rho(U^a)\in(\rho_*,\rho^*),\quad
 \p_1\varPhi^{a}\geq \frac{5}{8}\qquad \textrm{in }  \varOmega_T,
 \end{alignat}
 where $\widetilde{U}^{a}:=U^{a}-\widebar{U}$ and
 $\varPhi^{a}:=x_1 +\varPsi^{a} $ with $\varPsi^{a} :=\chi(x_1)\varphi^{a}$.
 Moreover,
 \begin{alignat}{3}
 \label{app2a}
 &\p_t^k\mathbb{L}(U^{a},\varPhi^{a})\big|_{t=0}=0\qquad
 \textrm{for } k=0,\ldots,m-1,\quad  &&\textrm{in } \varOmega,\\
 &\label{app2b}\mathbb{B}(U^a,\varphi^a)=0,
 \qquad H^{a}\cdot N^a=0\quad
 &&\textrm{on }  \varSigma_T,\\
 &\label{app2c}  (U^{a},\varphi^a)\big|_{t=0}=(U_0,\varphi_0),  &&\\
 \nonumber 
 &\big(\p_t^{ { \varPhi}^{a}}+v^{a}\cdot\nabla^{ { \varPhi}^{a}}\big)
 H^{a}-  ({H}^{a}\cdot\nabla^{ {\varPhi}^{a}} ) {v}^{a}
 + {H}^{a}\nabla^{ {\varPhi}^{a}} \cdot {v}^{a}=0
 \quad &&\textrm{in } \varOmega_T,
 \end{alignat}
 where we denote
 $ {N}^{a}:=(1,-\p_{2} \varPsi^{a}, -\p_{3} \varPsi^{a})^{\mathsf{T}}$ and
 $\nabla^{ {\varPhi}}:=(\p_1^{\varPhi},\p_2^{\varPhi},\p_3^{\varPhi})^{\mathsf{T}}$
 with $\p_t^{\varPhi}$ and $\p_i^{\varPhi}$ defined by \eqref{differential}.
\end{lemma}

The vector function $(U^a,\varphi^a)$ constructed in the Lemma \ref{lem:app} is called the {\it approximate solution} to the problem \eqref{NP1}.
Let us define
\begin{align}\label{f^a}
f^{a}:=\left\{\begin{aligned}
& -\mathbb{L}(U^{a},\varPhi^{a}) \quad &\textrm{if }t>0,\\
& 0 \quad &\textrm{if }t<0.
\end{aligned}\right.
\end{align}
Then it follows from \eqref{app1a} and \eqref{app2a} that
$f^{a}\in H^{m}(\varOmega_T)$ and
\begin{align}\label{f^a:est}
\|f^{a}\|_{ H^{m}(\varOmega_T)}\leq
\delta_0\left(T\right),
\end{align}
where $\delta_0(T)\to 0$ as $T\to 0$.
By virtue of \eqref{app2a}--\eqref{f^a},
we infer that  $(U,\varphi)=(U^a,\varphi^a)+(V,\psi)$ solves the nonlinear problem \eqref{NP1} on $[0,T]\times \varOmega$,
provided $V$, $\psi$, and $\varPsi:=\chi(x_1)\psi$ are solutions of the problem
\begin{align} \label{P.new}
\left\{
\begin{aligned}
&\mathcal{L}(V,\varPsi):=\mathbb{L}(U^a+V,\varPhi^a+\varPsi)-\mathbb{L}(U^a,\varPhi^a)=f^a
\ &&\textrm{in }\varOmega_T,\\
&\mathcal{B}(V,\psi):=\mathbb{B}(U^a+V,\varphi^a+\psi)=0
\ &&\textrm{on }\varSigma_T,\\
&(V,\psi)=0,\ &&\textrm{if }t< 0.
\end{aligned}\right.
\end{align}

\subsection{Nash--Moser Iteration Scheme} \label{sec:proof2}

We first quote the following result from \cite[Proposition 10]{T09ARMAMR2481071}.

\begin{proposition} \label{pro:smooth}
Let $T>0$ be a real number and $m\geq 3$ be an integer.
Denote $\mathscr{F}_*^s(\varOmega_T):=\big\{u\in H_*^{s}(\varOmega_T):u=0\textrm{ for }t<0\big\}$.
Then there is a family of smoothing operators $\{\mathcal{S}_{\theta}\}_{\theta\geq 1}:  \,\mathscr{F}_*^3(\varOmega_T) \to \bigcap_{s\geq 3}\mathscr{F}_*^s(\varOmega_T)$, such that
\begin{subequations}\label{smooth.p1}
\begin{alignat}{2}
&  \|\mathcal{S}_{\theta} u\|_{k,*,T}\leq C \theta^{(k-j)_+}\|u\|_{j,*,T}
&& \textrm{for  \;}k,j=1,\ldots,m,    \label{smooth.p1a}  \\[1.5mm]
&  \|\mathcal{S}_{\theta} u-u\|_{k,*,T}\leq C  \theta^{k-j}\|u\|_{j,*,T}
&& \textrm{for  \;}1\leq k\leq j \leq m,   \label{smooth.p1b} \\
&  \left\|\frac{\d}{\d \theta}\mathcal{S}_{\theta} u\right\|_{k,*,T}
\leq C \theta^{k-j-1}\|u\|_{j,*,T}
&\quad&\textrm{for \;}k,j=1,\ldots,m,    \label{smooth.p1c}
\end{alignat}
\end{subequations}
where $k,j$ are integers, $(k-j)_+:=\max\{0,\, k-j \}$, and the constant $C$ depends only on $m$.
Moreover, there exists another family of smoothing operators (still denoted by $\mathcal{S}_{\theta}$) acting on the functions defined on $\varSigma_T$ and satisfying the properties in \eqref{smooth.p1} with norms $\|\cdot\|_{H^{j}(\varSigma_T)}$.
\end{proposition}

Let us follow \cite{CS08MR2423311,T09ARMAMR2481071,TW21MR4201624} to describe the iteration scheme for \eqref{P.new}.

 \vspace*{1mm}
\noindent{\bf Assumption\;(A-1)}: {\it  Set $ (V_0, \psi_0)=0$. Let
 $(V_k,\psi_k)$ be given and vanish in the past, and set $\varPsi_k:=\chi( x_1)\psi_k$,
 for $k=0,\ldots,{n}$.}

\vspace*{1mm}
We consider
\begin{align}\label{NM0}
V_{{n}+1}=V_{{n}}+\delta V_{{n}},
\quad \psi_{{n}+1}=\psi_{{n}}+\delta \psi_{{n}},
\quad  \delta\varPsi_{{n}}:=\chi(x_1)\delta \psi_{{n}}.
\end{align}
The differences $\delta V_{{n}}$ and $\delta \psi_{{n}}$ will be specified via
\begin{align} \label{effective.NM}
\left\{\begin{aligned}
&\mathbb{L}_e'(U^a+V_{{n}+1/2},\varPhi^a+\varPsi_{{n}+1/2})\delta \dot{V}_{{n}}=f_{{n}}
\ \  &&\textrm{in }\varOmega_T,\\
& \mathbb{B}_e'(U^a+V_{{n}+1/2},\varphi^a+\psi_{{n}+1/2})(\delta \dot{V}_{{n}},\delta\psi_{{n}})=g_{{n}}
\ \ &&\textrm{on }\varSigma_T,\\
& (\delta \dot{V}_{{n}},\delta\psi_{{n}})=0\ \ &&\textrm{for }t<0,
\end{aligned}\right.
\end{align}
where $\varPsi_{{n}+1/2}:=\chi(x_1)\psi_{{n}+1/2}$,
the good unknown $\delta \dot{V}_{{n}}$ is defined by ({\it cf.}~\eqref{good})
\begin{align} \label{good.NM}
\delta \dot{V}_{{n}}:=\delta V_{{n}}-\frac{\p_1 (U^a+V_{{n}+1/2})}{\p_1 (\varPhi^a+\varPsi_{{n}+1/2})}\delta\varPsi_{{n}},
\end{align}
and $(V_{{n}+1/2},\psi_{{n}+1/2})$ is a smooth modified state to be defined in Proposition \ref{pro:modified} such that $(U^a+V_{{n}+1/2},\varphi^a+\psi_{{n}+1/2})$ satisfies \eqref{bas1a}--\eqref{bas1c}.
The source terms $f_n$ and $g_n$ will be chosen through the accumulated error terms at Step ${n}$ later on.

\vspace*{1mm}
\noindent{\bf Assumption (A-2)}: {\it Set $f_0:=\mathcal{S}_{\theta_0}f^a$ and $(e_0,\tilde{e}_0,g_0):=0$ for $\theta_0\geq 1$ sufficiently large, and let $(f_k,g_k,e_k,\tilde{e}_k)$ be given and vanish in the past for $k=1,\ldots,{n}-1$.}

\vspace*{1mm}

With Assumptions (A-1)--(A-2) in hand,
we calculate the accumulated error terms at Step $n$ for $n\geq 1$ by
\begin{align}  \label{E.E.tilde}
E_{{n}}:=\sum_{k=0}^{{n}-1}e_{k},\quad \widetilde{E}_{{n}}:=\sum_{k=0}^{{n}-1}\tilde{e}_{k}.
\end{align}
Then we compute $f_{{n}}$ and $g_{{n}}$ from
\begin{align} \label{source}
\sum_{k=0}^{{n}} f_k+\mathcal{S}_{\theta_{{n}}}E_{{n}}=\mathcal{S}_{\theta_{{n}}}f^a,
\quad
\sum_{k=0}^{{n}}g_k+\mathcal{S}_{\theta_{{n}}}\widetilde{E}_{{n}}=0,
\end{align}
where $\mathcal{S}_{\theta_{{n}}}$ are the smoothing operators defined in Proposition \ref{pro:smooth} with $\theta_0\geq 1$ and $\theta_{{n}}:=\sqrt{\theta^2_0+{n}}$.
Once $f_n$ and $g_n$ are chosen, we can use Theorem \ref{thm:linear} to obtain $(\delta \dot{V}_{{n}},\delta \psi_{{n}})$ from \eqref{effective.NM}.
Then we get $\delta V_{{n}}$ and $(V_{{n}+1},\psi_{{n}+1})$ from \eqref{good.NM} and \eqref{NM0} respectively.

The error terms at Step ${n}$ are defined as follows:
\begin{align}
\nonumber&\mathcal{L}(V_{{n}+1},\varPsi_{{n}+1})-\mathcal{L}(V_{{n}},\varPsi_{{n}})\\
\nonumber&\quad = \mathbb{L}'(U^a+V_{{n}},\varPhi^a+\varPsi_{{n}})(\delta V_{{n}},\delta\varPsi_{{n}})+e_{{n}}'\\
\nonumber&\quad = \mathbb{L}'(U^a+\mathcal{S}_{\theta_{{n}}}V_{{n}},\varPhi^a+\mathcal{S}_{\theta_{{n}}}\varPsi_{{n}})(\delta V_{{n}},\delta\varPsi_{{n}})+e_{{n}}'+e_{{n}}''\\
\nonumber&\quad = \mathbb{L}'(U^a+V_{{n}+1/2},\varPhi^a+\varPsi_{{n}+1/2})(\delta V_{{n}},\delta\varPsi_{{n}})+e_{{n}}'+e_{{n}}''+e_{{n}}'''\\
\label{decom1}&\quad = \mathbb{L}_e'(U^a+V_{{n}+1/2},\varPhi^a+\varPsi_{{n}+1/2})\delta \dot{V}_{{n}}+e_{{n}}'+e_{{n}}''+e_{{n}}'''+D_{{n}+1/2} \delta\varPsi_{{n}},
\\
\nonumber&\mathcal{B}(V_{{n}+1},\psi_{{n}+1})-\mathcal{B}(V_{{n}},\psi_{{n}})\\
\nonumber&\quad = \mathbb{B}'(U^a+V_{{n}},\varphi^a+\psi_{{n}})(\delta V_{{n}},\delta\psi_{{n}})+\tilde{e}_{{n}}'\\
\nonumber&\quad = \mathbb{B}'(U^a+\mathcal{S}_{\theta_{{n}}}V_{{n}},\varphi^a+\mathcal{S}_{\theta_{{n}}}\psi_{{n}})(\delta V_{{n}},\delta\psi_{{n}})+\tilde{e}_{{n}}'+\tilde{e}_{{n}}''\\
\label{decom2}&\quad =\mathbb{B}_e'(U^a+V_{{n}+1/2},
\varphi^a+\psi_{{n}+1/2})(\delta \dot{V}_{{n}} ,\delta\psi_{{n}})+\tilde{e}_{{n}}'+\tilde{e}_{{n}}''+\tilde{e}_{{n}}''',
\end{align}
where
\begin{align}\label{error.D}
D_{{n}+1/2}:=\frac{1}{\p_1(\varPhi^a+\varPsi_{{n}+1/2})}\p_1\mathbb{L}(U^a+V_{{n}+1/2},\varPhi^a+\varPsi_{{n}+1/2})
\end{align}
is the remaining error term
and we have used \eqref{Alinhac} to derive the last identity in \eqref{decom1}.
Setting
\begin{align} \label{e.e.tilde}
e_{{n}}:=e_{{n}}'+e_{{n}}''+e_{{n}}'''+D_{{n}+1/2} \delta\varPsi_{{n}},\quad
\tilde{e}_{{n}}:=\tilde{e}_{{n}}'+\tilde{e}_{{n}}''+\tilde{e}_{{n}}''',
\end{align}
completes the description of the iterative scheme.

Let $m\geq {13}$ be an integer and let $\widetilde{\alpha}:=m-{5}$.
Assume that the initial data $(U_0,\varphi_0)$ satisfy $\widetilde{U}_0:=U_0-\widebar{U}\in H^{m+3/2}(\varOmega)$ and $\varphi_0\in   H^{m+2}(\mathbb{R}^{2})$.
By virtue of Lemma \ref{lem:app}, we have
\begin{align} \label{small}
\big\|\widetilde{U}^a\big\|_{H^{\widetilde{\alpha}+{6}}(\varOmega_T)}
+\big\|\varphi^a\big\|_{H^{\widetilde{\alpha}+{15/2}}(\varSigma_T)}
\leq C(M_0),\
\big\|f^a\big\|_{H^{\widetilde{\alpha}+{5}}(\varOmega_T)}\leq \delta_0(T),
\end{align}
where $M_0$ is defined by \eqref{M0} and $\delta_0(T)\to 0$ as $T\to 0$.
Suppose further that Assumptions (A-1)--(A-2) are satisfied.
Our inductive hypothesis reads
\begin{align*}
(\mathbf{H}_{{n}-1})\ \left\{\begin{aligned}
\textrm{(a)}\,\,  &\|(\delta V_k,\delta \varPsi_k)\|_{s,*,T}+\|(\delta\psi_k,\mathrm{D}_{x'}\delta\psi_k)\|_{H^{s}(\varSigma_T)}\leq \boldsymbol{\epsilon} \theta_k^{s-{\alpha }-1}\varDelta_k\\
&\quad \textrm{for all } k=0,\ldots,{n}-1\textrm{ and }s=6,\ldots,\widetilde{\alpha} ;\\
\textrm{(b)}\,\, &\|\mathcal{L}( V_k,  \varPsi_k)-f^a\|_{s,*,T}\leq 2 \boldsymbol{\epsilon} \theta_k^{s-{\alpha }-1}\\
&\quad \textrm{for all } k=0,\ldots,{n}-1\textrm{ and } s= 6,\ldots,\widetilde{\alpha}-2;\\
\textrm{(c)}\,\,  &\|\mathcal{B}( V_k,  \psi_k)\|_{H^{s}(\varSigma_T)}\leq  \boldsymbol{\epsilon} \theta_k^{s-{\alpha }-1}\\
&\quad \textrm{for all } k=0,\ldots,{n}-1\textrm{ and } s=7,\ldots,{\alpha },
\end{aligned}\right.
\end{align*}
for integer $ {\alpha }\geq 7$, constant $\boldsymbol{\epsilon}>0$, and $\varDelta_{k}:=\theta_{k+1}-\theta_k$.
We are going to show that hypothesis ($\mathbf{H}_{{n}-1}$) implies ($\mathbf{H}_{{n}}$) and that ($\mathbf{H}_0$) holds,
provided $T>0$ and $\boldsymbol{\epsilon}>0$ are small enough and $\theta_0\geq 1$ is suitably large.

We first let hypothesis ($\mathbf{H}_{{n}-1}$) hold. Then we get the following lemma.

\begin{lemma}[{\cite[Lemma 7]{T09ARMAMR2481071}}] \label{lem:triangle}
 If $\theta_0$ is sufficiently large, then
 \begin{align}
&\|( V_k, \varPsi_k)\|_{s,*,T}+\|\psi_k\|_{H^{s}(\varSigma_T)}
\leq
\left\{\begin{aligned}
&\boldsymbol{\epsilon} \theta_k^{(s-{\alpha })_+}   &&\textrm{if }s\neq {\alpha },\\
&\boldsymbol{\epsilon} \log \theta_k   &&\textrm{if }s= {\alpha },
\end{aligned}\right.  \label{tri1}\\
&\| ({I}-\mathcal{S}_{\theta_k})(V_k,   \varPsi_k)\|_{s,*,T}
+\|({I}-\mathcal{S}_{\theta_k})\psi_k\|_{H^{s}(\varSigma_T)}
\leq C\boldsymbol{\epsilon} \theta_k^{s-{\alpha }},\label{tri2}
\end{align}
 for all $k=0,\ldots,{n}-1$ and  $s=6,\ldots,\widetilde{\alpha}$.
Moreover,
\begin{align}
&\|( \mathcal{S}_{\theta_k}V_k, \mathcal{S}_{\theta_k}\varPsi_k)\|_{s,*,T}
+\|\mathcal{S}_{\theta_k}\psi_k\|_{H^{s}(\varSigma_T)}\leq
\left\{\begin{aligned}
&C\boldsymbol{\epsilon} \theta_k^{(s-{\alpha })_+}  &&\textrm{if }s\neq {\alpha },\\
&C\boldsymbol{\epsilon} \log \theta_k  &&\textrm{if }s= {\alpha },
\end{aligned}\right.   \label{tri3}
 \end{align}
 for all $k=0,\ldots,{n}-1$ and $s=6,\ldots,\widetilde{\alpha}+6$.
\end{lemma}

\subsection{Error Estimates} \label{sec:proof3}

This subsection is devoted to the estimate of
the quadratic error terms $e'_{k}$ and $\tilde{e}_{k}'$,
the first substitution error terms $e_{k}''$ and $\tilde{e}_{k}''$,
the second substitution error terms $e_{k}''' $  and  $\tilde{e}_{k}'''$,
and the last error term $D_{k+1/2} \delta\varPsi_{k}$
({\it cf.}~\eqref{decom1}--\eqref{error.D}).
First we find
\begin{align*}
e_k' =\;&\int_{0}^{1}
\mathbb{L}''\big(U^a+V_{k}+\tau \delta  V_{k},
\varPhi^a+\varPsi_{k} 
+\tau \delta \varPsi_{k}\big)\big((\delta V_{k},\delta\varPsi_{k}),(\delta V_{k},\delta\varPsi_{k})\big)
(1-\tau)\d\tau,\\
\tilde{e}_k'
=\;&
\int_{0}^{1}
\mathbb{B}''\big(U^a+V_{k}+\tau \delta  V_{k},
\varphi^a+\psi_{k}  
+\tau \delta \psi_{k}\big)\big((\delta V_{k},\delta \psi_{k}),(\delta V_{k},\delta \psi_{k})\big)
(1-\tau)\d\tau,
\end{align*}
where  $\mathbb{L}''$ and $\mathbb{B}''$ are the second derivatives of the operators
$\mathbb{L}$ and $\mathbb{B}$, respectively, that is,
\begin{align*}
\mathbb{L}''\big(\mathring{U},\mathring{\varPhi}\big)
\big((V,\varPsi),(\widetilde{V},\widetilde{\varPsi})\big)
&:=\left.\frac{\d}{\d \theta}
\mathbb{L}'\big(\mathring{U}+\theta \widetilde{V},
\mathring{\varPhi}+\theta \widetilde{\varPsi}\big)
\big(V,\varPsi\big)\right|_{\theta=0},\\
\mathbb{B}''\big(\mathring{U},\mathring{\varphi}\big)
\big((V,\psi),(\widetilde{V},\tilde{\psi})\big)
&:=\left.\frac{\d}{\d \theta}
\mathbb{B}'(\mathring{U}+\theta \widetilde{V},
\mathring{\varphi}+\theta \tilde{\psi}) (V,\psi)\right|_{\theta=0}.
\end{align*}
For our problem \eqref{NP1}, from \eqref{B'.bb:def}, applying a direct calculation yields
\begin{align}
\nonumber
&\mathbb{B}''\big(\mathring{U},\mathring{\varphi}\big)
\big((V,\psi),(\widetilde{V},\tilde{\psi})\big)
\\ & \quad =
\begin{pmatrix}
(\tilde{v}_2   \p_2+\tilde{v}_3   \p_3 )\psi + (v_2  \p_2 +v_3  \p_3)\tilde{\psi} \\[1mm]
\mathfrak{s}\mathrm{D}_{x'}\cdot
\left(
\dfrac{\mathring{\zeta}\cdot\tilde{\zeta}}{|\mathring{N}|^3}\zeta
-\dfrac{\tilde{\zeta}\cdot {\zeta}}{|\mathring{N}|^3}\mathring{\zeta}
-\dfrac{\mathring{\zeta}\cdot {\zeta}}{|\mathring{N}|^3}\tilde{\zeta}
+\dfrac{3(\mathring{\zeta}\cdot {\zeta})(\mathring{\zeta}\cdot\tilde{\zeta})}{|\mathring{N}|^5}\mathring{\zeta}
  \right)
\end{pmatrix},
\label{B''.form}
\end{align}
where $\zeta:=\mathrm{D}_{x'}\psi$, $\mathring{\zeta}:=\mathrm{D}_{x'}\mathring{\varphi}$, $\tilde{\zeta}:=\mathrm{D}_{x'}\tilde{\psi}$, and $\mathring{N}:=(1,-\p_2\mathring{\varphi},-\p_3\mathring{\varphi})^{\mathsf{T}}$.
Using the Moser-type calculus and embedding inequalities, and omitting detailed calculations, we obtain the following estimates for the operators $\mathbb{L}''$ and $\mathbb{B}''$.

\begin{proposition}  \label{pro:tame2}
Let $T>0$ be a real number and $s\geq 6$ be an integer.
Suppose that $(\widetilde{V},\widetilde\varPsi)\in H_*^{s+2}(\varOmega_T)$ and $\tilde{\varphi}\in H^{s+2}(\varSigma_T)$ satisfy
$$\|(\widetilde{V},\widetilde{\varPsi} )\|_{W_*^{2,\infty}(\varOmega_T)} +\|\tilde{\varphi}\|_{W^{1,\infty}(\varSigma_T)}\leq \widetilde{K}$$
for some constant $\widetilde{K}>0$.
Then there is a constant $C(\widetilde{K})>0$, such that,
if $(V_i,\varPsi_i)\in H_*^{s+2}(\varOmega_T)$ and $(W_i,\psi_i)\in H^{s}(\varSigma_T)\times H^{s+2}(\varSigma_T)$ for $i=1,2$, then
\begin{align}
\big\|\mathbb{L}''\big(\widebar{U}+\widetilde{V},x_1+\widetilde{\varPsi} \big)&\big((V_1,\varPsi_1),(V_2,\varPsi_2) \big)\big\|_{s,*,T} \nonumber \\[2mm]
\leq C(\widetilde{K}) \sum_{i\neq j}\Big\{&
\big\|(V_i,\varPsi_i)\big\|_{6,*,T}\big\|(V_j,\varPsi_j)\big\|_{s+2,*,T} \nonumber \\
&
+ \big\|(V_1,\varPsi_1)\big\|_{6,*,T} \big\|(V_2,\varPsi_2)\big\|_{6,*,T}   \big\|\big({\widetilde{V}},\widetilde{\varPsi}\big)\big\|_{s+2,*,T}
\Big\}, \nonumber 
\end{align}
and
\begin{align}
\big\|\mathbb{B}''\big(\widebar{U}+\widetilde{V},\widetilde{\varphi} \big)&
\big((W_1,\psi_1),(W_2,\psi_2) \big)\big\|_{H^{s}(\varSigma_T)} \nonumber\\[1mm]
\leq C(\widetilde{K}) \sum_{i\neq j} \Big\{&
\big\|\psi_i\big\|_{H^{3}(\varSigma_T)}\big\|\psi_j\big\|_{H^{s+2}(\varSigma_T)}
+\big\|\psi_1\big\|_{H^{3}(\varSigma_T)}\big\|\psi_2\big\|_{H^{3}(\varSigma_T)}\big\|\tilde{\varphi}\big\|_{H^{s+2}(\varSigma_T)}
\nonumber\\
&
+\big\|W_i\big\|_{H^{s}(\varSigma_T)}\big\|\psi_j\big\|_{H^{3}(\varSigma_T)}
+ \big\|W_i\big\|_{H^{2}(\varSigma_T)} \big\|\psi_j\big\|_{H^{s+1}(\varSigma_T)}
\Big\}.
\nonumber
\end{align}
\end{proposition}

We first estimate the quadratic error terms $e'_{k}$ and $\tilde{e}_{k}'$.

\begin{lemma}\label{lem:quad}
 Let ${\alpha }\geq 7$.
 If $\boldsymbol{\epsilon}>0$ is small enough and $\theta_0\geq 1$ is sufficiently large, then
 \begin{align}\label{quad.est}
 \|e_k'\|_{s,*,T}+\|\tilde{e}_k'\|_{H^{s}(\varSigma_T)}
 \lesssim \boldsymbol{\epsilon}^2 \theta_k^{\varsigma_1(s)-1}\varDelta_k,
 \end{align}
 for $k=0,\ldots,{n}-1$ and $s=6,\ldots,\widetilde{\alpha}-2$,
 where  $$\varsigma_1(s):=\max\{s+6-2{\alpha },(s+2-{\alpha })_++10-2{\alpha } \}.$$
\end{lemma}
\begin{proof}
Using the embedding theorem, hypothesis $(\mathbf{H}_{{n}-1})$, and \eqref{small}--\eqref{tri1} yields
\begin{align*}
\|(\widetilde{ U}^a,V_{k},\delta  V_{k},\varPsi^a,\varPsi_{k}, \delta \varPsi_{k})\|_{W_*^{2,\infty}(\varOmega_T)}
+\|(\varphi^a,\psi_{k}, \delta \psi_{k})\|_{W^{1,\infty}(\varSigma_T)}
\lesssim 1.
\end{align*}
So Proposition \ref{pro:tame2} can be applied for estimating $e'_{k}$ and $\tilde{e}_{k}'$.
Precisely, we use the trace theorem, hypothesis $(\mathbf{H}_{{n}-1})$, and \eqref{small} to obtain that
for $s=6,\ldots,\widetilde{\alpha}-2$,
\begin{align*}
\|\tilde{e}_k'\|_{H^s(\varSigma_T)}
\lesssim \;&
\|\delta \psi_{k}\|_{H^6(\varSigma_T)}\|\delta \psi_{k}\|_{H^{s+2}(\varSigma_T)}
+\|\delta \psi_{k}\|_{H^{6}(\varSigma_T)}^2\|(\varphi^a,\psi_{k}, \delta \psi_{k})\|_{H^{s+2}(\varSigma_T)}\\
&+\|\delta V_{k}\|_{s+1,*,T}\|\delta \psi_{k}\|_{H^6(\varSigma_T)}
+\|\delta V_{k}\|_{6,*,T}\|\delta \psi_{k}\|_{H^{s+1}(\varSigma_T)}\\
\lesssim \;&
\boldsymbol{\epsilon}^2\theta_{k}^{s+6-2{\alpha }}\varDelta_k^2
+\boldsymbol{\epsilon}^2\theta_k^{10-2{\alpha }}\varDelta_k^2\big(1+\|\psi_k\|_{H^{s+2}(\varSigma_T)}\big).
\end{align*}

If $s+2\neq {\alpha }$, then
\begin{align*}
\|\tilde{e}_k'\|_{H^s(\varSigma_T)}
\lesssim  \boldsymbol{\epsilon}^2 \varDelta_k^2
\big(\theta_k^{s+6-2{\alpha }} +\theta_k^{(s+2-{\alpha })_++10-2{\alpha }}\big)
\lesssim  \boldsymbol{\epsilon}^2 \theta_k^{\varsigma_1(s)-1 }\varDelta_k,
\end{align*}
due to \eqref{tri1} and $\varDelta_k\lesssim \theta_k^{-1}$.

If $s+2= {\alpha }$, then it follows from \eqref{tri1} and ${\alpha }\geq 7$ that
\begin{align*}
\|\tilde{e}_k'\|_{H^s(\varSigma_T)}
\lesssim  \boldsymbol{\epsilon}^2 \varDelta_k^2\big( \theta_k^{4-{\alpha }}+ \theta_k^{11-2{\alpha }} \big)
\lesssim \boldsymbol{\epsilon}^2 \theta_k^{\varsigma_1({\alpha }-2)-1 }\varDelta_k.
\end{align*}

The estimate of ${e}_k'$ can be obtained similarly, so we omit the details and finish the proof of the lemma.
\end{proof}

The following lemma concerns the estimate of  $e_{k}''$ and $\tilde{e}_{k}''$ defined in \eqref{decom1}--\eqref{decom2}.
\begin{lemma} \label{lem:1st}
Let ${\alpha }\geq 7$.
If $\boldsymbol{\epsilon}>0$ is small enough and $\theta_0\geq 1$ is sufficiently large, then
\begin{align}\label{1st.sub}
\|e_k'' \|_{s,*,T} +\|\tilde{e}_k''\|_{H^{s}(\varSigma_T)}
\lesssim \boldsymbol{\epsilon}^2 \theta_k^{\varsigma_2(s)-1}\varDelta_k,
\end{align}
for $k=0,\ldots,{n}-1$ and $s=6,\ldots,\widetilde{\alpha}-2$,
where
\begin{align}\label{varsigma2.def}
\varsigma_2(s):=\max\{s+8-2{\alpha },(s+2-{\alpha })_++12-2{\alpha } \}.
\end{align}
\end{lemma}
\begin{proof}
We can rewrite the term $\tilde{e}_k''$ as
\begin{align}
\nonumber \tilde{e}_k''=
\int_{0}^{1}&\mathbb{B}''\Big(U^a+\mathcal{S}_{\theta_k}V_k+\tau({I}-\mathcal{S}_{\theta_k})V_k,\,
\varphi^a+\mathcal{S}_{\theta_k}\psi_k \\
&\quad\ \  +\tau({I}-\mathcal{S}_{\theta_k})\psi_k\Big)
\Big(\big(\delta V_k ,\delta \psi_k\big),\, \big(({I}-\mathcal{S}_{\theta_k})V_k,({I}-\mathcal{S}_{\theta_k})\psi_k\big) \Big)\d\tau.
\nonumber
\end{align}
It follows from the embedding theorem, hypothesis $(\mathbf{H}_{{n}-1})$, and \eqref{small}--\eqref{tri3} that
\begin{align*}
\|(\widetilde{ U}^a,\mathcal{S}_{\theta_k} V_{k},V_{k},\varPsi^a, \mathcal{S}_{\theta_k}\varPsi_{k},\varPsi_{k})\|_{W_*^{2,\infty}(\varOmega_T)}
+\|(\varphi^a,\mathcal{S}_{\theta_k}\psi_{k},  \psi_{k})\|_{W^{1,\infty}(\varSigma_T)}
\lesssim 1.
\end{align*}
Then we can employ Proposition \ref{pro:tame2} to infer that, for $s= 6,\ldots,\widetilde{\alpha}-2$,
\begin{align*}
\|\tilde{e}_k''\|_{H^{s}(\varSigma_T)}
\lesssim  \boldsymbol{\epsilon}^2 \theta_k^{s+7-2{\alpha }}\varDelta_{k}
+\boldsymbol{\epsilon}^2 \theta_k^{11-2{\alpha }}\varDelta_{k}
\big(1 + \|(\mathcal{S}_{\theta_k} \psi_k, \psi_k)\|_{H^{s+2}(\varSigma_T)}\big).
\end{align*}
Analyzing the cases $s+2\neq {\alpha }$ and $s+2= {\alpha }$ separately as in the proof of Lemma \ref{lem:quad}, we utilize \eqref{tri1}--\eqref{tri3} to derive \eqref{1st.sub} and complete the proof.
\end{proof}

In order to solve \eqref{effective.NM}, the smooth modified state $(V_{{n}+1/2},\psi_{{n}+1/2})$ will be constructed to ensure that the constraints \eqref{bas1a}--\eqref{bas1c} hold for $(U^a+V_{{n}+1/2},\varphi^a+\psi_{{n}+1/2})$.
For $T>0$ sufficiently small, $(U^a+V_{{n}+1/2},\varphi^a+\psi_{{n}+1/2})$ will satisfy \eqref{bas1a} and \eqref{bas1c}, because $(V_{{n}+1/2},\psi_{{n}+1/2})$ will be specified to vanish in the past and $(U^a,\varphi^a)$ satisfies \eqref{app1a}--\eqref{app1b} and \eqref{app2b}.
Hence it suffices to focus on the constraints \eqref{bas1b}.
As a matter of fact, we can quote the following proposition in our previous paper \cite{TW21MR4201624}, since the construction and estimate of the smooth modified state therein are independent of the second boundary condition in \eqref{NP1b}.

\begin{proposition}[{\cite[Proposition 4.8]{TW21MR4201624}}]
\label{pro:modified}
Let ${\alpha }\geq 8$.
Then there are $V_{n+1/2}$ and $\psi_{n+1/2}$ vanishing in the past,
such that $(U^a+V_{n+1/2},\varphi^a+\psi_{n+1/2})$ satisfies \eqref{bas1b} for the approximate solution $(U^a, \varphi^a)$ constructed in Lemma \ref{lem:app}.
Moreover,
\begin{alignat}{3}
&\psi_{n+1/2}=\mathcal{S}_{\theta_n}\psi_{{n}},\quad
v_{2,n+1/2}=\mathcal{S}_{\theta_n} v_{2,n},
\quad  v_{3,n+1/2}=\mathcal{S}_{\theta_n} v_{3,n},  \label{MS.id1}\\
&\|\mathcal{S}_{\theta_n}\varPsi_n-\varPsi_{n+1/2}\|_{s,*,T}\lesssim \boldsymbol{\epsilon} \theta_n^{s-{\alpha }}
\qquad  \textrm{for } s=6,\ldots,\widetilde{\alpha}+6, \label{MS.est1} \\
&\|\mathcal{S}_{\theta_n}V_n-V_{n+1/2}\|_{s,*,T}\lesssim \boldsymbol{\epsilon} \theta_n^{s+2-{\alpha }}
\qquad  \textrm{for } s=6,\ldots,\widetilde{\alpha}+{4}.  \label{MS.est2}
\end{alignat}
\end{proposition}

The second substitution error term $e_{k}'''$ given in \eqref{decom1} can be rewritten as
\begin{align}
\nonumber {e}_k'''=\int_{0}^{1}&
\mathbb{L}''\Big(U^a+\tau(\mathcal{S}_{\theta_k} V_k-V_{k+1/2})+V_{k+1/2},\,
\varPhi^a+\tau(\mathcal{S}_{\theta_k}\varPsi_k-\varPsi_{k+1/2})
\\
&\quad\ \ +\varPsi_{k+1/2} \Big)
\Big((\delta V_k ,\delta \varPsi_k),\, (\mathcal{S}_{\theta_k} V_k-V_{k+1/2},\mathcal{S}_{\theta_k}\varPsi_k-\varPsi_{k+1/2})\Big) \d\tau.
\nonumber
\end{align}
For $\tilde{e}_{k}'''$ defined in \eqref{decom2},
we have from \eqref{MS.id1} and \eqref{B''.form} that
\begin{align}
\nonumber \tilde{e}_k'''=\int_{0}^{1}&
\mathbb{B}''\Big(U^a+\tau(\mathcal{S}_{\theta_k} V_k-V_{k+1/2})+V_{k+1/2},\,
\varphi^a
\\
&\quad\ \ +\psi_{k+1/2} \Big)
\Big((\delta V_k ,\delta \psi_k),\, (\mathcal{S}_{\theta_k} V_k-V_{k+1/2},0)\Big) \d\tau=0.
\nonumber
\end{align}
Then we can obtain the following lemma by using Propositions \ref{pro:tame2} and \ref{pro:modified}.
\begin{lemma} \label{lem:2nd}
 Let ${\alpha }\geq 8$.
 If $\boldsymbol{\epsilon}>0$ is small enough and $\theta_0\geq 1$ is sufficiently large, then
 \begin{align}\label{2st.sub}
 \tilde{e}_k'''=0,\quad  \| e_k'''\|_{s,*,T}
 \lesssim \boldsymbol{\epsilon}^2 \theta_k^{\varsigma_3(s)-1}\varDelta_k,
 \end{align}
 for $k=0,\ldots,{n}-1$ and $s=6,\ldots,\widetilde{\alpha}-2$,
 where  $$\varsigma_3(s):=\max\{s+10-2{\alpha },(s+2-{\alpha })_++14-2{\alpha } \}.$$
\end{lemma}

The next lemma provides the estimate of $D_{k+1/2}\delta\varPsi_k$ defined by \eqref{error.D}.
\begin{lemma}[{\cite[Lemma 4.10]{TW21MR4201624}}]   \label{lem:last}
 Let ${\alpha }\geq 8$ and $\widetilde{\alpha}\geq {\alpha }+2$.
 If $\boldsymbol{\epsilon}>0$ is small enough and $\theta_0\geq 1$ is sufficiently large, then
 \begin{align} \label{last.e0}
 \|D_{k+1/2}\delta\varPsi_k\|_{s,*,T}\lesssim \boldsymbol{\epsilon}^2 \theta_k^{\varsigma_4 (s)-1}\varDelta_k,
 \end{align}
 for $k=0,\ldots,{n}-1$ and $s=6,\ldots,\widetilde{\alpha}-2$,
 where
 \begin{align} \label{varsigma4.def}
 \varsigma_4(s):=\max\{s+12-2{\alpha },(s-{\alpha })_++18-2{\alpha }\}.
 \end{align}
\end{lemma}

With Lemmas \ref{lem:quad}--\ref{lem:last} in hand,
we can estimate the accumulated error terms $E_n$ and $\widetilde{E}_n$
defined by \eqref{E.E.tilde} (also see \eqref{e.e.tilde}).

\begin{lemma}[{\cite[Lemma 4.12]{TW21MR4201624}}] \label{lem:sum2}
 Let ${\alpha }\geq 12$ and $\widetilde{\alpha}={\alpha }+3$.
 If $\boldsymbol{\epsilon}>0$ is small enough and $\theta_0\geq 1$ is sufficiently large, then
 \begin{align}\label{es.sum2}
 \|E_{{n}}\|_{{\alpha }+1,*,T}\lesssim \boldsymbol{\epsilon}^2 \theta_{{n}},\quad
 \|\widetilde{E}_{{n}}\|_{H^{{\alpha }+1}(\varSigma_T)}\lesssim \boldsymbol{\epsilon}^2.
 \end{align}
\end{lemma}

\subsection{Proof of Existence} \label{sec:proof4}

To derive hypothesis $(\mathbf{H}_{{n}})$ from $(\mathbf{H}_{{n}-1})$, we need the following estimates of $f_{{n}}$ and $g_{{n}}$ given in \eqref{source}. The proof can be found in \cite{TW21MR4201624}.

\begin{lemma}[{\cite[Lemma 4.13]{TW21MR4201624}}]  \label{lem:source}
Let ${\alpha }\geq 12$ and $\widetilde{\alpha}={\alpha }+3$.
If $\boldsymbol{\epsilon}>0$ is small enough and $\theta_0\geq 1$ is sufficiently large, then
\begin{align}
&
\|f_{{n}}\|_{s,*,T}
\lesssim \varDelta_{{n}}\big(\theta_{{n}}^{s-{\alpha }-1} \|f^a\|_{{\alpha },*,T}
+\boldsymbol{\epsilon}^2 \theta_{{n}}^{s-{\alpha }-1} +\boldsymbol{\epsilon}^2\theta_{{n}}^{\varsigma_4(s)-1}\big),
\nonumber \\ 
&
\|g_{{n}}\|_{H^{s+1}(\varSigma_T)}
\lesssim  \boldsymbol{\epsilon}^2 \varDelta_{{n}}\big(\theta_{{n}}^{s-{\alpha }-1}
+\theta_{{n}}^{\varsigma_2(s+1)-1}\big),\nonumber 
\end{align}
for $s=6,\ldots,\widetilde{\alpha}$,
where  $\varsigma_2(s)$ and $\varsigma_4(s)$ are given in \eqref{varsigma2.def} and \eqref{varsigma4.def} respectively.
\end{lemma}

The next lemma follows by using \eqref{MS.est1}--\eqref{MS.est2}, the tame estimate \eqref{tame}, and Lemma \ref{lem:source}.
The proof is omitted here for brevity, since it is similar to that of \cite[Lemma 15]{T09ARMAMR2481071}.

\begin{lemma}\  \label{lem:Hl1}
Let ${\alpha }\geq 12$ and $\widetilde{\alpha}={\alpha }+3$.
If $\boldsymbol{\epsilon}>0$ and $\|f^a\|_{{\alpha },*,T}/\boldsymbol{\epsilon}$ are small enough,
and if $\theta_0\geq1$ is sufficiently large, then
\begin{align} \label{Hl.a}
\|(\delta V_{{n}},\delta\varPsi_{{n}})\|_{s,*,T}+\|(\delta\psi_{{n}},\mathrm{D}_{x'}\delta\psi_{{n}})\|_{H^{s}(\varSigma_T)}
\leq \boldsymbol{\epsilon} \theta_{{n}}^{s-{\alpha }-1}\varDelta_{{n}}.
 \end{align}
 for $s=6,\ldots,\widetilde{\alpha}$.
\end{lemma}

The above lemma provides the estimate (a) in hypothesis $(\mathbf{H}_{{n}})$.
The other estimates in $(\mathbf{H}_{{n}})$ are given in the following lemma,
whose proof is similar to that of \cite[Lemma 16]{T09ARMAMR2481071}.

\begin{lemma}\ \label{lem:Hl2}
Let ${\alpha }\geq 12$ and $\widetilde{\alpha}={\alpha }+3$.
If $\boldsymbol{\epsilon}>0$ and $\|f^a\|_{{\alpha },*,T}/\boldsymbol{\epsilon}$ are small enough,
and if $\theta_0\geq1$ is sufficiently large, then
\begin{alignat}{3}\label{Hl.b}
&
\|\mathcal{L}( V_{{n}},  \varPsi_{{n}})-f^a\|_{s,*,T}\leq 2 \boldsymbol{\epsilon} \theta_{{n}}^{s-{\alpha }-1}
\quad &&\textrm{for } s=6,\ldots,\widetilde{\alpha}-1,\\
\label{Hl.c}
& \|\mathcal{B}( V_{{n}} ,  \psi_{{n}})\|_{H^{s}(\varSigma_T)}\leq  \boldsymbol{\epsilon} \theta_{{n}}^{s-{\alpha }-1}
\quad &&\textrm{for } s=7,\ldots,{\alpha }.
\end{alignat}
\end{lemma}

If ${\alpha }\geq 12$, $\widetilde{\alpha}={\alpha }+3$,
$\boldsymbol{\epsilon}>0$, and $\|f^a\|_{{\alpha },*,T}/\boldsymbol{\epsilon}$ are sufficiently small,
and $\theta_0 \geq 1$ is large enough,
then we can derive  hypothesis $(\mathbf{H}_{{n}})$ from $(\mathbf{H}_{{n}-1})$ in virtue of Lemmas \ref{lem:Hl1}--\ref{lem:Hl2}.
Fixing the constants ${\alpha }$, $\widetilde{\alpha}$, $\boldsymbol{\epsilon}>0$, and $\theta_0\geq1$,
as in \cite[Lemma 17]{T09ARMAMR2481071}, we can derive that $(\mathbf{H}_{0})$ is true for a suitably small time.

\begin{lemma}\ \label{lem:H0}
If $T>0$ is sufficiently small, then hypothesis $(\mathbf{H}_0)$ holds.
\end{lemma}

Let us prove the existence of solutions to the nonlinear problem \eqref{NP1}.

\vspace*{3mm}
\noindent  {\bf Proof of the existence part of Theorem {\rm\ref{thm:main}}.}\ \
Suppose that we are given the initial data $(U_0,\varphi_0)$ satisfying all the assumptions listed in Theorem {\rm\ref{thm:main}}.
Let $\widetilde{\alpha}=m-{5}$ and ${\alpha }=\widetilde{\alpha}-3\geq 12$.
Then the initial data $(U_0,\varphi_0)$ are compatible up to order $m=\widetilde{\alpha}+{5}$.
In view of \eqref{app1a} and \eqref{f^a:est},  taking $\boldsymbol{\epsilon}>0$ and $T>0$ sufficiently small, and $\theta_0\geq 1$ large enough, we obtain all the requirements of Lemmas \ref{lem:Hl1}--\ref{lem:H0}.
Then, for suitably small time $T>0$,
hypothesis $(\mathbf{H}_{{n}})$ holds for all ${n}\in\mathbb{N}$.
In particular,
\begin{align*}
\sum_{n=0}^{\infty}\left(\|(\delta V_n,\delta \varPsi_n)\|_{s,*,T}+\|(\delta\psi_n,\mathrm{D}_{x'}\delta\psi_n)\|_{H^{s}(\varSigma_T)} \right)
\lesssim \sum_{n=0}^{\infty}\theta_n^{s-{\alpha }-2} <\infty
\end{align*}
for $s=6,\ldots, {\alpha }-1.$
Consequently, the sequence $(V_{n},\psi_n)$ converges to some limit $(V,\psi)$ in $H_*^{{\alpha }-1}(\varOmega_T)\times H^{{\alpha }-1}({\varSigma_T})$, and also in
$H^{\lfloor ({\alpha }-1)/2\rfloor}(\varOmega_T)\times H^{{\alpha }-1}({\varSigma_T})$, owing to the embedding $H_*^s\hookrightarrow H^{\lfloor s/2 \rfloor}$.
Moreover, we have $\mathrm{D}_{x'}\psi\in H^{{\alpha }-1}({\varSigma_T})$.
Passing to the limit in \eqref{Hl.b}--\eqref{Hl.c} for $s={\alpha }-1=m-{9}$, we obtain \eqref{P.new}.
Hence $(U, \varphi)=(U^a+V, \varphi^a+\psi)$ solves the original nonlinear problem \eqref{NP1} on $[0,T]$. \qed

\subsection{Proof of Uniqueness} \label{sec:unique}

It remains to prove the uniqueness of solutions to the nonlinear problem \eqref{NP1}.
For this purpose, we assume that there exist two solutions $(U,\varphi)$ and $(\mathring{U},\mathring{\varphi})$ of the problem \eqref{NP1}.
Setting the differences $\widetilde{U}:=U-\mathring{U}$ and $\psi:=\varphi-\mathring{\varphi}$, we deduce
\begin{subequations}
\label{differ.eq1}
\begin{alignat}{3}
&L(U,\varPhi)\widetilde{ U}- L(U,\varPhi) {\varPsi} \frac{\p_1 \mathring{U}}{\p_1 \mathring{\varPhi}}
=R_{\rm int}
&\quad &\textrm{in  }  [0,T]\times \varOmega,
\label{differ.eq1a}\\[1mm]
& (\p_t+\mathring{v}_2\p_2+\mathring{v}_3\p_3){\psi}-\tilde{v}\cdot N=0
&\quad &\textrm{on  }  [0,T]\times \varSigma,
\label{differ.eq1b}\\[1mm]
& \tilde{q}-\mathfrak{s} \mathrm{D}_{x'}\cdot
\bigg(\dfrac{\mathrm{D}_{x'}\varphi}{| {N}|}-
\dfrac{\mathrm{D}_{x'}\mathring{\varphi}}{| \mathring{N}|}\bigg)=0
&\quad &\textrm{on  }  [0,T]\times \varSigma,
\label{differ.eq1c}
\end{alignat}
\end{subequations}
and we can impose
\begin{align}
	(\widetilde{U},\psi)=0\qquad \textrm{for }t<0,
\label{differ.eq2}
\end{align}
thanks to the trivial initial data $(\widetilde{U},\psi)|_{t=0}=0$.
Here
\begin{alignat*}{3}
\varPhi(t,x)&:= x_1+\chi(x_1)\varphi(t,x'), &\quad N&:= (1,-\p_2\varphi,-\p_3\varphi)^{\mathsf{T}}, \\
\mathring{\varPhi}(t,x)&:= x_1+\chi(x_1)\mathring{\varphi}(t,x'), &\quad  \mathring{N}&:= (1,-\p_2\mathring{\varphi},-\p_3\mathring{\varphi})^{\mathsf{T}},
\end{alignat*}
and $\varPsi:=\chi(x_1)\psi(t,x')=\varPhi-\mathring{\varPhi}$ with $C_0^{\infty}(\mathbb{R})$--function $\chi$ satisfying  the requirements \eqref{chi}.
Using the mean value theorem (see, {\it e.g.}, Zorich \cite[\S 8.4.1]{Z15MR3495809}), we have
\begin{align}
	R_{\rm int}:=(L(\mathring{U},\mathring{\varPhi})-L(U,\mathring{\varPhi}))\mathring{U}=
	\hat{a}_1 \widetilde{U},
\label{R.int}
\end{align}
where the matrix $\hat{a}_1$ depends on $\mathrm{D}\mathring{U}$, $\mathrm{D}\mathring{\varPhi}$, and some `mean value' $U^*$ lying between $U$ and $\mathring{U}$. Precisely, $U^*=\mathring{U}+\theta\widetilde{U}$ for some $\theta\in (0,1)$, so that its norm can be controlled by the corresponding norms of $U$ and $\mathring{U}$.
Regarding the boundary condition \eqref{differ.eq1c}, we employ the Taylor's lemma 
(see \cite[\S 8.4.4]{Z15MR3495809} for instance) to infer
\begin{align}
\mathfrak{f}(\xi)-\mathfrak{f}(\mathring{\xi})
=(\zeta_1\p_{\xi_1}+\zeta_2\p_{\xi_2})\mathfrak{f}(\mathring{\xi})+
\frac{1}{2}(\zeta_1\p_{\xi_1}+\zeta_2\p_{\xi_2})^2\mathfrak{f}(\mathring{\xi}+\theta'\zeta)
\ \ \textrm{for }0<\theta'<1,
\nonumber
\end{align}
where $\mathfrak{f}$ is defined by \eqref{f.frak}, $\xi:=\mathrm{D}_{x'}\varphi$,
$\mathring{\xi}:=\mathrm{D}_{x'}\mathring{\varphi}$, and  $\zeta:=\xi-\mathring{\xi}=\mathrm{D}_{x'}\psi$.
Moreover, we compute
\begin{align}
\mathfrak{s} \mathrm{D}_{x'}\cdot
\left(\dfrac{\mathrm{D}_{x'}\varphi}{| {N}|}-
\dfrac{\mathrm{D}_{x'}\mathring{\varphi}}{| \mathring{N}|}\right)
=\mathfrak{s}\mathrm{D}_{x'}\cdot
\bigg(\frac{\mathrm{D}_{x'}\psi}{|\mathring{N}|}-\frac{\mathrm{D}_{x'}\mathring{\varphi}\cdot\mathrm{D}_{x'}\psi}{|\mathring{N}|^3}\mathrm{D}_{x'}\mathring{\varphi}
\bigg)+R_{\rm bdy},
\label{differ.id1}
\end{align}
where
\begin{align}
	R_{\rm bdy}=\sum_{i,j=2,3}\mathrm{D}_{x'}\cdot \big(\hat{b}_{ij}\psi_{x_i}\psi_{x_j} \big),
\label{R.bdy}
\end{align}
for generic vector-valued coefficients $\hat{b}_{ij}$ whose norms can be estimated through Sobolev's norms of the interface functions $\varphi$ and $\mathring{\varphi}$.

As for the linearized problem in Section \ref{sec:linear}, we pass to the ``good unknown'' ({\it cf.}\;\eqref{good})
\begin{align}
\dot{U}:=\widetilde{ U}-\frac{\p_1 \mathring{U}}{\p_1 \mathring{\varPhi}}\varPsi
\nonumber
\end{align}
for the difference $\widetilde{ U}$ of solutions, and introduce the new unknown
$$W:=J(\varPhi)\dot{U},$$
where $J$ is defined by \eqref{J.def}.
Taking into account \eqref{R.int}--\eqref{R.bdy} and omitting detailed computations, we reformulate the problem \eqref{differ.eq1}--\eqref{differ.eq2} into
\begin{subequations}
\label{differ.eq3}
	\begin{alignat}{3}
\label{differ.eq3a}
		&\sum_{i=0}^3{\bm{A}}_i\p_i W +{\bm{A}}_4 W =\bm{f}
		&\quad &\textnormal{in }\varOmega_T,\\
\label{differ.eq3b}
		&W_2=(\p_t+\mathring{v}_2\p_2+\mathring{v}_3\p_3)\psi-\p_1\mathring{v}\cdot {N} \psi
		&\quad &\textnormal{on }\varSigma_T,\\[1mm]
\label{differ.eq3c}
		&W_1=-\p_1\mathring{q} \psi +\mathfrak{s}\mathrm{D}_{x'}\cdot
		\bigg(\dfrac{\mathrm{D}_{x'}\psi}{|\mathring{N}|}-
		\dfrac{\mathrm{D}_{x'}\mathring{\varphi}\cdot\mathrm{D}_{x'}\psi}{|\mathring{N}|^3}\mathrm{D}_{x'}\mathring{\varphi}
		\bigg)
				+R_{\rm bdy}
		&\quad &\textnormal{on }\varSigma_T,\\[1mm]
\label{differ.eq3d}
		&(W,\psi)=0 &\quad &\textnormal{if }t<0,
	\end{alignat}
\end{subequations}
where
\begin{align*}
	{\bm{A}}_1:=J({\varPhi})^{\mathsf{T}}\widetilde{A}_1( {U}, {\varPhi})J(  {\varPhi}),\
	{\bm{A}}_4:=-J(  {\varPhi})^{\mathsf{T}}\hat{a}_1 J(  {\varPhi}),\
	{\bm{A}}_{i}:=J(  {\varPhi})^{\mathsf{T}}{A}_{i}( {U})J(  {\varPhi}),
\end{align*}
for $i=0,2,3$,
the term $R_{\rm bdy}$ is given in \eqref{R.bdy},
and $\bm{f}:=J({\varPhi})^{\mathsf{T}} \hat{a}_2\varPsi$
for some suitable matrix-valued function $\hat{a}_2$ depending on $\mathrm{D}\mathring{U}$, $\mathrm{D}\mathring{\varPhi}$, and some `mean value' $U^*$ lying between $U$ and $\mathring{U}$.
Since $\p_t \varphi=v\cdot N$ on the boundary $\varSigma$,
we derive the identity
\begin{align}
	{\bm{A}}_1\big|_{x_1=0}=
	\begin{pmatrix}
		0 & 1 &0 \\
		1 & 0 & 0\\
		0 & 0 & O_6
	\end{pmatrix},
\nonumber
\end{align}
which has been proved to be important in deriving the energy estimates for solutions of the linearized problem \eqref{ELP3}.
Indeed, for the problem \eqref{differ.eq3}, we can also deduce the estimate \eqref{est2a} for $k=0,2,3$.
However, compared with the boundary condition \eqref{ELP3c}, there is one additional nonlinear (quadratic) term $R_{\rm bdy}$ in \eqref{differ.eq3c}.
As a result, the second term in \eqref{est2a} will be decomposed into the integral of the right-hand side of \eqref{iden3} over $\varSigma_t$ and the following integral:
\begin{align}
\widetilde{\mathcal{J}}_k:=-2\int_{\varSigma_t} \p_k R_{\rm bdy}\p_k W_2
=-2\sum_{i=2,3} \int_{\varSigma_t} \mathrm{D}_{x'}\cdot \p_k \big(\hat{b}_{ij}\psi_{x_i}\psi_{x_j} \big)\p_k W_2.
\label{J.cal}
\end{align}
In order to close the energy estimate in $H_*^1$, we treat in
\begin{align*}
\sum_{i=2,3}  \p_k \big( \hat{b}_{ij} \psi_{x_i}\psi_{x_j} \big)
&= \sum_{i=2,3}  (
\p_k \hat{b}_{ij} \psi_{x_i}\psi_{x_j}
+\hat{b}_{ij}  \p_k\psi_{x_i} \psi_{x_j}
+\hat{b}_{ij}  \psi_{x_i}\p_k \psi_{x_j}
) =\hat{c}_{k}\mathrm{D}_{x'}\psi
\end{align*}
the higher-order derivatives as coefficients whose norms can be bounded through Sobolev's norms of the solutions $\varphi$ and $\mathring{\varphi}$.
Then we have
\begin{align}
\widetilde{\mathcal{J}}_k=
2\int_{\varSigma_t}\p_k \left(\hat{c}_{k}\mathrm{D}_{x'}\psi \cdot \mathrm{D}_{x'} W_2\right)
-2\int_{\varSigma_t} \p_k(\hat{c}_{k}\mathrm{D}_{x'}\psi)\cdot \mathrm{D}_{x'} W_2,
\nonumber
\end{align}
which combined with \eqref{differ.eq3b} implies
\begin{multline}
\big|\widetilde{\mathcal{J}}_k\big|
\lesssim
\sum_{   |\alpha|\leq 2} \|(\mathrm{D}_{x'}^{\alpha}\psi,\mathrm{D}_{x'}\p_t \psi)\|_{L^2(\varSigma_t)}^2
+\boldsymbol{\epsilon} \|\mathrm{D}_{x'}\psi (t)\|_{L^2(\varSigma)}^2
\nonumber\\
+C(\boldsymbol{\epsilon})\sum_{   |\alpha|\leq 2} \|(\mathrm{D}_{x'}^{\alpha}\psi,\mathrm{D}_{x'}\p_t \psi) (t)\|_{L^2(\varSigma)}^2 \qquad
\textrm{for }\boldsymbol{\epsilon}>0.
\end{multline}
Employing the entirely similar arguments as in \S \ref{sec:linear2}, we finally derive the estimate \eqref{est4a} with $\bm{f}=J({\varPhi})^{\mathsf{T}} \hat{a}_2\chi(x_1)\psi$.
Estimating $\bm{f}$ through $\psi$ and using \eqref{est3g}, we obtain
\begin{multline}
	 \sum_{\langle\beta\rangle\leq 1}\|\mathrm{D}_*^{\beta}W(t)\|_{L^2(\varOmega)}^2
	+ \sum_{|\alpha|\leq 2}\|(W_2, \mathrm{D}_{x'}^{\alpha}\psi,  \mathrm{D}_{x'}\p_t\psi)(t)\|_{L^2(\varSigma)}^2\\
	\lesssim \|\psi \|_{H^1(\varSigma_t)}^2
	\lesssim \|(W_2,\psi,\mathrm{D}_{x'}\psi)\|_{L^2(\varSigma_t)}^2.
\nonumber
\end{multline}
Applying Gr\"{o}nwall's inequality to the last estimate leads to $W=0$ and $\psi=0$,
which imply the uniqueness of solutions to the nonlinear problem \eqref{NP1}, that is, $U=U'$ and $\varphi=\varphi'$.
Therefore, the proof of Theorem \ref{thm:main} is complete.


\section*{Acknowledgements}

The authors would like to thank the anonymous referee for helpful comments and suggestions to improve the quality of redaction. 


%
\section*{Conflict of interest}

 The authors declare that they have no conflict of interest.




%

\end{document}